\documentclass[10pt,reqno]{amsart}

\usepackage{amsfonts}
\usepackage{eurosym}
\usepackage{amssymb}
\usepackage{amsthm}
\usepackage{amsmath}
\usepackage{amsaddr}
\usepackage{bm}
\usepackage{cite}
\usepackage{mathrsfs}
\usepackage{xcolor}
\usepackage[OT1]{fontenc}
\usepackage[left=2.5cm, right=2.5cm, top=3cm, bottom=2.5cm]{geometry}
\usepackage{hyperref}
\usepackage{mathtools}
\hypersetup{colorlinks=true, linkcolor=blue, citecolor=red, urlcolor=blue}
\usepackage{times}

\usepackage{enumitem}
\usepackage{xpatch}
\makeatletter
\AtBeginDocument{\xpatchcmd{\@thm}{\thm@headpunct{.}}{\thm@headpunct{}}{}{}}
\makeatother

\allowdisplaybreaks
\newtheorem{theorem}{Theorem}[section]

\newtheorem{lemma}[theorem]{Lemma}

\theoremstyle{remark}
\newtheorem{remark}[theorem]{Remark}

\numberwithin{equation}{section}

\usepackage{parskip}
\DeclareMathAlphabet\mathbfcal{OMS}{cmsy}{b}{n}
\DeclareMathOperator*{\esssup}{ess\,sup}

\newcommand{\abs}[1]{| #1 |}
\newcommand{\bigabs}[1]{\big| #1 \big|}
\newcommand{\Bigabs}[1]{\Big| #1 \Big|}

\newcommand{\norm}[1]{\| #1 \|}

\newcommand{\Bignorm}[1]{\Big\| #1 \Big\|}
\newcommand{\ang}[2]{ \langle #1 , #2  \rangle}
\newcommand{\bigang}[2]{ \big< #1 , #2  \big>}

\newcommand{\scp}[2]{ ( #1 , #2  )}
\newcommand{\bigscp}[2]{\big( #1 , #2 \big)}

\newcommand{\meano}[1]{{\langle #1 \rangle}_{\Omega}}
\newcommand{\meang}[1]{{\langle #1 \rangle}_{\Gamma}}
\newcommand{\mean}[2]{\textnormal{mean}\scp{#1}{#2}}

\newcommand{\R}{\mathbb R}
\newcommand{\N}{\mathbb N}
\newcommand{\n}{\mathbf{n}}
\newcommand{\bv}{\mathbf{v}}
\newcommand{\bu}{\mathbf{u}}
\newcommand{\bw}{\mathbf{w}}
\newcommand{\ww}{\widehat{\bw}}
\newcommand{\wv}{\widehat{\bv}}
\newcommand{\tv}{\widetilde{\bv}}
\newcommand{\tw}{\widetilde{\bw}}
\newcommand{\wphi}{\widehat{\phi}}
\newcommand{\wpsi}{\widehat{\psi}}
\newcommand{\wmu}{\widehat{\mu}}
\newcommand{\wtheta}{\widehat{\theta}}

\newcommand{\J}{\mathbf{J}}
\newcommand{\K}{\mathbf{K}}
\newcommand{\T}{\mathbf{T}}
\newcommand{\I}{\mathbf{I}}
\newcommand{\D}{\mathbf{D}}
\newcommand{\f}{\mathbf{f}}
\newcommand{\g}{\mathbf{g}}
\newcommand{\A}{\mathbf{A}}
\newcommand{\M}{\mathbf{M}}
\renewcommand{\L}{\mathbf{L}}
\newcommand{\G}{\mathbf{G}}

\newcommand{\Dg}{\mathbf{D}_\Gamma}
\newcommand{\Z}{\mathbf{Z}}
\newcommand{\intO}{\int_\Omega}

\newcommand{\intG}{\int_\Gamma}

\newcommand{\dtau}{\;\mathrm d\tau}

\newcommand{\dx}{\;\mathrm{d}x}

\newcommand{\dt}{\;\mathrm dt}
\newcommand{\ds}{\;\mathrm ds}

\newcommand{\dxs}{\;\mathrm{d}x\;\mathrm{d}s}
\newcommand{\dGs}{\;\mathrm{d}\Ga\;\mathrm{d}s}

\newcommand{\dG}{\;\mathrm d\Ga}

\newcommand{\ddt}{\frac{\mathrm d}{\mathrm dt}}

\newcommand{\del}{\partial}
\newcommand{\delt}{\partial_{t}}
\newcommand{\delth}{\partial_{t}^{h}}

\newcommand{\deln}{\partial_\n}

\newcommand{\Grad}{\nabla}
\newcommand{\Lap}{\Delta}
\newcommand{\Div}{\textnormal{div}}
\newcommand{\Gradg}{\nabla_\Ga}
\newcommand{\Lapg}{\Delta_\Ga}
\newcommand{\Divg}{\textnormal{div}_\Ga}

\newcommand{\emb}{\hookrightarrow}

\newcommand{\suchthat}{\;\ifnum\currentgrouptype=16 \middle\fi|\;}

\newcommand{\e}{\mathbf{e}}

\newcommand{\Om}{\Omega}
\newcommand{\Ga}{\Gamma}

\def \no#1#2#3 {{\bf #1} (#3), #2.}
\def \eds#1#2#3 {#1, #2, #3.}

\makeatletter
\def\@settitle{\begin{center}%
  \baselineskip14\p@\relax
    \huge
  \@title
  \end{center}%
}
\makeatother

\begin{document}
\title[Global well-Posedness of strong solutions to a bulk-surface NSCH system]{Global well-posedness of strong solutions to a \\ bulk-surface Navier--Stokes--Cahn--Hilliard model with non-degenerate mobilities in two dimensions}
\author[J. Stange]{Jonas Stange}

\address{Fakult\"{a}t f\"{u}r Mathematik\\ 
Universit\"{a}t Regensburg \\
93040 Regensburg, Germany\\
\href{mailto:jonas.stange@ur.de}{jonas.stange@ur.de}
}

\subjclass[2010]{
Primary: 76T06
; Secondary:
 35Q35, 
 76D03, 
 76D05, 
 76D07, 
}
\keywords{Multi-phase flows, Navier--Stokes--Cahn--Hilliard system, bulk-surface interaction, dynamic boundary conditions, strong solutions, uniqueness.}

\begin{abstract} We examine a thermodynamically consistent diffuse interface model for bulk-surface viscous fluid mixtures. This model consists of a Navier--Stokes--Cahn--Hilliard model in the bulk coupled to a surface Navier--Stokes--Cahn--Hilliard system on the boundary. In this paper, we address the global well-posedness of strong solutions in the two-dimensional setting, also covering the physically meaningful case of non-degenerate mobility functions. Lastly, we prove the uniqueness of the corresponding strong solutions and their continuous dependence on the initial data. Our approach hinges upon new well-posedness and regularity theory for a convective bulk-surface Cahn--Hilliard equation with non-degenerate mobilities, as well as a bulk-surface Stokes equation with non-constant coefficients.
\end{abstract}

\maketitle


\section{Introduction}
\label{Section:Introduction}
\noindent

In this contribution, we study a thermodynamically consistent, diffuse interface model for bulk-surface viscous fluid mixtures with different densities in two space dimensions. The model was recently derived by Knopf and the author in \cite{Knopf2025a} by means of local mass balance laws, local energy dissipation laws, and the Lagrange multiplier approach. The corresponding bulk-surface Navier--Stokes--Cahn--Hilliard model reads as
\begin{subequations}\label{System}
\begin{alignat}{3}
    \label{System:1}
    &\delt(\rho(\phi)\bv) + \Div(\bv\otimes(\rho(\phi)\bv + \J)) = \Div\;\T , &&\qquad\Div\;\bv = 0 &&\qquad\text{in~}Q, \\
    \label{System:2}
    &\delt(\sigma(\psi)\bw) + \Divg(\bw\otimes(\sigma(\psi)\bw + \K)) = \Divg\;\T_\Ga + \Z, &&\qquad \Divg\;\bw = 0 &&\qquad\text{on~}\Sigma, \\
    \label{System:3}
    &\bw = \bv\vert_\Ga,\qquad \bv\cdot\n = 0 && &&\qquad\text{on~}\Sigma, \\[0.3em]
    \label{System:4}
    &\delt\phi  + \Div(\phi\bv) = \Div(m_\Om(\phi)\Grad\mu) && &&\qquad\text{in~}Q, \\
    \label{System:5}
    &\mu = -\Lap\phi + F^\prime(\phi) && &&\qquad\text{in~}Q, \\[0.3em]
    \label{System:6}
    &\delt\psi + \Divg(\psi\bw) = \Divg(m_\Ga(\psi)\Gradg\theta) - \beta m_\Om(\phi)\deln\mu && &&\qquad\text{on~}\Sigma, \\
    \label{System:7}
    &\theta = -\Lapg\psi + G^\prime(\psi) + \alpha\deln\phi && &&\qquad\text{on~}\Sigma, \\
    \label{System:8}
    &K\deln\phi = \alpha\psi - \phi, \qquad Lm_\Om(\phi)\deln\mu = \beta\theta - \mu, &&\qquad K,L\in[0,\infty]  &&\qquad\text{on~}\Sigma,
\end{alignat}
accompanied by the initial conditions
\begin{alignat}{2}
    \phi\vert_{t=0} = \phi_0, &\qquad \bv\vert_{t=0} = \bv_0 &&\qquad\text{in~}\Om, \label{System:InitialConditions:1} \\
    \label{System:InitialConditions:2}
    \psi\vert_{t=0} = \psi_0, &\qquad \bw\vert_{t=0} = \bw_0 &&\qquad\text{on~}\Ga.
\end{alignat}
\end{subequations}

In \eqref{System}, $\Om\subset\R^2$ is bounded domain with boundary $\Ga\coloneqq\partial\Om$. For brevity, we use the notation $Q = \Om\times(0,\infty)$ and $\Sigma = \Ga\times(0,\infty)$. The outward pointing unit normal vector on $\Ga$ is denoted by $\n$, while $\deln$ denotes the outward normal derivative on the boundary. Moreover, the symbols $\Gradg$, $\Divg$, and $\Lapg$ stand for the surface gradient, divergence, and the Laplace--Beltrami operator on $\Ga$, respectively. We refer to Section~\ref{NOT:DIFFOP} for a precise definition of these objects.

In the bulk domain $\Om$, the motion of the mixture is described by the velocity field $\bv:Q\rightarrow\R^d$, while on the boundary, the viscous dynamics are governed by the surface velocity field $\bw:\Sigma\rightarrow\R^d$. The location of the two materials within the bulk is captured by the phase-field $\phi:Q\rightarrow\R$, whereas the location of the fluids on the boundary is represented by the surface phase-field $\psi:\Sigma\rightarrow\R$. The associated chemical potentials are given by $\mu:Q\rightarrow\R$ in the bulk and $\theta:\Sigma\rightarrow\R$ on the surface.

The time evolution of the bulk and surface velocity fields, $\bv$ and $\bw$, is governed by the bulk Navier--Stokes equation \eqref{System:1} and the tangential surface Navier--Stokes equation \eqref{System:2}, respectively. In this framework, the two fields are coupled through the trace relation \eqref{System:3}$_1$, which directly links the motion of the bulk and surface materials. Both flows are assumed to be incompressible, which implies that the corresponding velocity fields are divergence-free. This boundary condition can be viewed as a generalization of the generalized Navier boundary condition (GNBC), which relates tangential stresses to slip velocities and uncompensated Young stresses. The GNBC plays a key role in modeling contact line phenomena, see, e.g., \cite{Dussan1979, Qian2006, Seppecher1996}. The inclusion of additional viscous effects in system~\eqref{System} is inspired by the classical fluid-mosaic model of membrane structure introduced in \cite{Singer1972}.  In that model, biological membranes are regarded as laterally incompressible, two-dimensional fluid layers with embedded proteins. This motivates treating the boundary as an active layer capable of supporting both lateral momentum transport and compositional dynamics. A rigorous continuum formulation of these ideas was later developed in \cite{Arroyo2009}, where lipid bilayers are modeled by a surface Navier--Stokes equation with positive, but finite viscosity, which is expected to dominate relaxation dynamics. These findings and observations motivate the inclusion of a surface Navier--Stokes equation in system \eqref{System} to account for viscous and convective effects at the boundary. 

In system~\eqref{System}, $\rho(\phi)$ stands for the phase-field dependent density in the bulk, whereas $\sigma(\psi)$ stands for the phase-field dependent density on the surface. They are given by
\begin{align*}
    \rho(\phi) \coloneqq \frac{\tilde\rho_2 - \tilde\rho_1}{2}\phi + \frac{\tilde\rho_2 + \tilde\rho_1}{2} \qquad\text{in~}Q, \qquad \sigma(\psi) \coloneqq \frac{\tilde\sigma_2 - \tilde\sigma_1}{2}\psi + \frac{\tilde\sigma_2 + \tilde\sigma_1}{2} \qquad\text{on~}\Sigma.
\end{align*}
Furthermore, $\J$ and $\K$ represent mass fluxes related to interfacial motion, and are given by
\begin{align*}
    \J = - \frac{\tilde\rho_2 - \tilde\rho_1}{2}m_\Om(\phi)\Grad\mu \qquad\text{in~}Q, \qquad \K = - \frac{\tilde\sigma_2 - \tilde\sigma_1}{2}m_\Ga(\psi)\Gradg\theta \qquad\text{on~}\Sigma.
\end{align*}
Moreover, $\T$ and $\T_\Ga$ are stress tensors, that are given by
\begin{alignat*}{2}
    \T &\coloneqq 2\nu_\Om(\phi)\D\,\bv - p\I - \Grad\phi\otimes\Grad\phi &&\qquad\text{in~}Q, \\
    \T_\Ga &\coloneqq 2\nu_\Ga(\psi)\Dg\,\bw - q\I - \Gradg\psi\otimes\Gradg\psi &&\qquad\text{on~}\Sigma.
\end{alignat*}
Here, the phase-field dependent coefficients $\nu_\Om, \nu_\Ga:[-1,1]\rightarrow\R$ denote the viscosity of the mixtures in the bulk and on the surface, respectively. A typical choice is the affine interpolation
\begin{align*}
    \nu(s) = \frac{\nu_2 - \nu_1}{2}s + \frac{\nu_2 + \nu_1}{2} \qquad\text{for~}s\in[-1,1],
\end{align*}
where $\nu_1$ and $\nu_2$ are the constant viscosities of the two underlying fluids. Moreover, 
\begin{align*}
    \D\,\bv \coloneqq \frac12\big((\Grad\bv) + (\Grad\bv)^\top\big), \qquad \Dg\,\bw \coloneqq \frac12\big((\Gradg\bw) + (\Gradg\bw)^\top\big)
\end{align*}
denote the bulk and surface symmetric gradient of $\bv$ and $\bw$, respectively.
The vector $\Z$ represents a surface force density that accounts for mechanical interactions between the bulk and the surface materials. It is of the form
\begin{align}\label{Def:Z}
    \Z = -2\nu_\Om(\phi)[\D\,\bv\n]_\tau + \tfrac12(\J\cdot\n)\bw + \tfrac12(\J_\Ga\cdot\n)\bw + \alpha\deln\phi\Gradg\psi - \gamma(\phi,\psi)\bw \qquad\text{on~}\Sigma.
\end{align}
The coefficient $\gamma:[-1,1]^2\rightarrow\R$ denotes a slip coefficient associated with tangential friction, which is allowed to depend on both the phase-fields. In \eqref{Def:Z}, $\J_\Ga$ represents a mass flux corresponding to the transfer of materials between bulk and surface, and is given by
\begin{align*}
    \J_\Ga = -\beta\frac{\tilde\sigma_2 - \tilde\sigma_1}{2}m_\Om(\phi)\Grad\mu \qquad\text{on~}\Sigma.
\end{align*}

The time evolution of $\phi$ and $\mu$ is governed by the convective bulk Cahn--Hilliard subsystem \eqref{System:4}-\eqref{System:5}, while the dynamics of $\psi$ and $\theta$ is described by the convective surface Cahn--Hilliard subsystem \eqref{System:6}-\eqref{System:7}. These two subsystems are coupled through terms involving the normal derivatives $\deln\phi$ and $\deln\mu$. Moreover, the phase-fields $\phi$ and $\psi$ are directly linked by the boundary condition \eqref{System:8}$_1$, which has to be understood in the following sense
\begin{align*}
    \begin{cases}
        \phi = \alpha\psi &\text{if~} K = 0, \\
        K\deln\phi = \alpha\psi - \phi &\text{if~} L\in(K,\infty), \\
        \deln\phi = 0 &\text{if~} K = \infty, 
    \end{cases} \qquad\text{on~}\Sigma.
\end{align*}
The boundary condition \eqref{System:8}$_2$, which couples the chemical potentials $\mu$ and $\theta$ has to be understood similarly. In these relations, the parameters $K,L\in[0,\infty]$ distinguish different coupling regimes while $\alpha,\beta\in\R$ represent material-dependent physical effects, see, for instance, \cite{Giorgini2023, Knopf2024} for a more extensive description. Such coupling conditions belong to the class of dynamic boundary conditions, which generalize the classical homogeneous Neumann boundary conditions
\begin{align*}
    \deln\phi = \deln\mu = 0 \qquad\text{on~}\Sigma,
\end{align*}
commonly used in standard Cahn--Hilliard models. While Neumann boundary conditions remain popular for their simplicity, they can be too restrictive when accurate modeling of interfacial dynamics is required, see, e.g., \cite{Giorgini2023, Knopf2024}. This limitation has motivated extensive research into dynamic boundary condition formulations. For a comprehensive overview, we refer to the survey \cite{Wu2022} and the references therein. In particular, Cahn--Hilliard systems with dynamic boundary conditions that permit mass exchange between the bulk and the boundary have attracted significant interest in recent years, see, e.g., \cite{Goldstein2011, Liu2019, Knopf2020, Knopf2021a}. 

The functions $m_\Om, m_\Ga:[-1,1]\rightarrow\R$ are the so-called Onsager mobilities. They depend, in general, on the phase-field variables $\phi$ and $\psi$, respectively. They model both the spatial locations and the intensity at which the underlying diffusion processes take place.

The functions $F^\prime$ and $G^\prime$ are the derivatives of double-well potentials $F$ and $G$, respectively. A physically motivated example of such a double-well potential, in particular in applications related to materials science, is the logarithmic potential, which is also referred to as the Flory--Huggins potential. It is given as
\begin{align*}
    W_{\mathrm{log}}(s) = \frac{\Theta}{2}\big[(1+s)\ln(1+s) + (1-s)\ln(1-s)\big] - \frac{\Theta_c}{2}s^2 \qquad\text{for~}s\in[-1,1],
\end{align*}
with the convention that $0\,\ln 0$ is interpreted as zero. Here, we assume that $0 < \Theta < \Theta_c$, where $\Theta$ denotes the temperature of the mixture, and $\Theta_c$ represents the critical temperature below which phase separation processes are to be expected. We point out that $W_{\mathrm{log}}$ is classified as a singular potential, as it satisfies $W_{\mathrm{log}}^\prime(s) \rightarrow \pm\infty$ as $s\rightarrow\pm 1$.

Employing the decompositions
\begin{alignat*}{2}
    -\Div(\Grad\phi\otimes\Grad\phi) &= -\Grad\left(\frac12\abs{\Grad\phi}^2 + F(\phi)\right) + \mu\Grad\phi &&\qquad\text{in~}Q, \\
    -\Divg(\Gradg\psi\otimes\Gradg\psi) &= -\Gradg\left(\frac12\abs{\Gradg\psi}^2 + G(\psi)\right) + \theta\Gradg\psi - \alpha\deln\phi\Gradg\psi &&\qquad\text{on~}\Sigma,
\end{alignat*}
which readily follow from \eqref{System:5} and \eqref{System:6}, respectively, we can replace the bulk and surface pressure $p$ and $q$ by
\begin{alignat*}{2}
    \overline{p} &\coloneqq p + \frac12\abs{\Grad\phi}^2 + F(\phi) &&\qquad\text{in~}Q, \\
    \overline{q} &\coloneqq q + \frac12\abs{\Gradg\psi}^2 + G(\psi) &&\qquad\text{on~}\Sigma.
\end{alignat*}
This allows us to rewrite \eqref{System:1}$_1$-\eqref{System:2}$_1$ as
\begin{subequations}\label{eqs:NSCH:pressure}
    \begin{alignat}{2}
        &\delt(\rho(\phi)\bv) + \Div(\bv\otimes(\rho(\phi)\bv + \J)) - \Div(2\nu_\Om(\phi)\D\bv) + \Grad\overline{p} = \mu\Grad\phi &&\qquad\text{in~}Q \label{eqs:NSCH:pressure:v} \\
        &\delt(\sigma(\psi)\bw) + \Divg(\bw\otimes(\sigma(\psi)\bw + \K) - \Divg(2\nu_\Ga(\psi)\Dg\bw) + \Gradg\overline{q}\nonumber  \\
        &\qquad = \theta\Gradg\psi - 2\nu_\Om(\phi)\big[ \D\bv\,\n \big]_\tau + \tfrac12(\J\cdot\n)\bw - \tfrac12(\J_\Ga\cdot\n)\bw - \gamma(\phi,\psi)\bw &&\qquad\text{on~}\Sigma. \label{eqs:NSCH:pressure:w}
    \end{alignat}
\end{subequations}
Moreover, for the subsequent analysis, it will prove useful to write \eqref{eqs:NSCH:pressure} in the following so-called non-conservative form
\begin{subequations}\label{eqs:NSCH:NC}
    \begin{alignat}{2}
        &\rho(\phi)\delt\bv + \rho(\phi)(\bv\cdot\Grad)\bv + (\J\cdot\Grad)\bv - \Div(2\nu_\Om(\phi)\D\bv) + \Grad\overline{p} = \mu\Grad\phi &&\qquad\text{in~}Q, \label{ConsForm:v}\\
        &\sigma(\psi)\delt\bw + \sigma(\psi)(\bw\cdot\Gradg)\bw + (\K\cdot\Gradg)\bw - \Divg(2\nu_\Ga(\psi)\Dg\bw) + \Gradg\overline{q}\nonumber \\
        &\qquad = \theta\Gradg\psi - 2\nu_\Om(\phi)\big[ \D\bv\,\n \big]_\tau + \tfrac12(\J\cdot\n)\bw - \tfrac12(\J_\Ga\cdot\n)\bw - \gamma(\phi,\psi)\bw &&\qquad\text{on~}\Sigma. \label{ConsForm:w}
    \end{alignat}
\end{subequations}

Lastly, multiplying \eqref{System:4} with $\tfrac{\tilde\rho_2 - \tilde\rho_1}{2}$ and \eqref{System:6} with $\tfrac{\tilde\sigma_2 - \tilde\sigma_1}{2}$, we readily obtain that the time evolution of $\rho(\phi)$ and $\sigma(\psi)$ is given by
\begin{subequations}\label{Equation:densities}
    \begin{alignat}{2}
        \delt\rho(\phi) + \Div(\rho(\phi)\bv) &= -\Div\,\J &&\qquad\text{in~}Q, \label{Equation:rho} \\
        \delt\sigma(\psi) + \Divg(\sigma(\psi)\bw) &= -\Div\,\K + \J_\Ga\cdot\n &&\qquad\text{on~}\Sigma, \label{Equation:sigma}
    \end{alignat}
\end{subequations}
which will be useful later on in the analysis.

The energy functional associated with system \eqref{System} reads as
\begin{align}\label{Definition:TotalEnergy}
    \begin{split}
        E_{\mathrm{tot}}(\bv,\bw,\phi,\psi) &= \intO \frac{\rho(\phi)}{2}\abs{\bv}^2\dx + \intG \frac{\sigma(\psi)}{2}\abs{\bw}^2\dG + \intO \frac12\abs{\Grad\phi}^2 + F(\phi)\dx \\
        &\quad + \intG \frac12\abs{\Grad\psi}^2 + G(\psi)\dG + \chi(K)\intG \frac12(\alpha\psi - \phi)^2\dG,
    \end{split}
\end{align}
where the function
\begin{align}\label{Definition:chi}
    \chi:[0,\infty]\rightarrow[0,\infty), \quad\chi(r) \coloneqq \begin{cases}
        \tfrac1r, &\text{if~}r\in(0,\infty), \\
        0, &\text{if~} r\in\{0,\infty\},
    \end{cases}
\end{align}
is used to distinguish different cases of the parameter $K$. Sufficiently regular solutions to \eqref{System} satisfy the following energy dissipation law
\begin{align}\label{EnergyDissipation}
    \begin{split}
        &\ddt E_{\mathrm{tot}}(\bv,\bw,\phi,\psi) + \intO 2\nu_\Om(\phi)\abs{\D\bv}^2\dx + \intG 2\nu_\Ga(\psi)\abs{\Dg\bw}^2\dG + \intG \gamma(\phi,\psi)\abs{\bw}^2\dG \\
        &\qquad + \intO m_\Om(\phi)\abs{\Grad\mu}^2\dx + \intG m_\Ga(\psi)\abs{\Gradg\theta}^2\dG + \chi(L)\intG (\beta\theta - \mu)^2\dG = 0 
    \end{split}
\end{align}
on $[0,\infty)$, as well as the mass conservation law
\begin{align}\label{MassConservation}
    \begin{dcases}
        \beta\intO \phi(t)\dx + \intG \psi(t)\dG = \beta\intO \phi(0) \dx + \intG \psi(0)\dG, &\textnormal{if~} L\in[0,\infty), \\
        \intO\phi(t)\dx = \intO\phi(0)\dx \quad\textnormal{and}\quad \intG\psi(t)\dG = \intG\psi(0)\dG, &\textnormal{if~} L = \infty,
    \end{dcases}
\end{align}
for all $t\in[0,\infty)$. This means that the mass is conserved separately in $\Om$ and on $\Ga$ in the case $L = \infty$, whereas, if $L\in[0,\infty)$, a transfer of material between bulk and surface is expected to occur, see, e.g., \cite{Knopf2021a}.

\textbf{Related literature.} 
For related diffuse interface models without dynamic boundary conditions, such as the Abels--Garcke--Grün (AGG) model, the existence of strong solutions has been extensively investigated, see, for instance, \cite{Abels2021, Abels2024a, Giorgini2021, Giorgini2022}, as well as \cite{Giorgini2020, Giorgini2019}. In contrast, for models with dynamic boundary conditions, the literature has so far focused primarily on the existence of weak solutions; we refer to \cite{Chan2025, Gal2019, Giorgini2023, Gal2023a, Knopf2025a} for representative results. Progress towards stronger notions of solutions has recently been made in \cite{Gal2023c}, where the authors established the existence of quasi-strong solutions for an Allen--Cahn--Navier--Stokes--Voigt system endowed with dynamic boundary conditions.

\textbf{Goals and novelties of the paper.}
System~\eqref{System} was recently derived and analyzed in \cite{Knopf2025a}, where the authors established the existence of weak solutions in the case of constant bulk and surface mobilities $m_\Om$ and $m_\Ga$ via a semi-Galerkin scheme.
The present work advances this analysis by proving the existence and uniqueness of strong solutions in the two-dimensional setting, allowing also for non-degenerate mobility functions.
Our approach is again based on a semi-Galerkin approximation, employing eigenfunctions of a bulk-surface Stokes operator to approximate the velocity fields. After constructing suitable approximate solutions by means of a fixed-point argument, we derive suitable uniform estimates. However, in contrast to the estimates performed in \cite{Knopf2025a} for weak solutions, our proof is more involved, as we have to establish several higher-order estimates. For the convective bulk-surface Cahn--Hilliard subsystem with non-degenerate mobilities, these estimates are based on the regularity theory developed in Appendix~\ref{Section:ConvCH}, which extends the author’s recent work \cite{Stange2025} on the non-convective case.
To obtain higher-order estimates for the velocity fields, we establish new regularity theory for a bulk-surface Stokes system with variable coefficients, presented in Appendix~\ref{Section:Stokes}, generalizing the framework developed in \cite[Section~5]{Knopf2025a} for the constant-coefficient case.
We further prove uniqueness and continuous dependence on the initial data of strong solutions for $L\in(0,\infty]$. The case $L = 0$ is excluded from our analysis, as the boundary condition \eqref{System:8}$_2$ no longer allows reformulation of the normal derivative $\deln\mu$, making it unmanageable in our framework. Nevertheless, the situation $K = L = 0$ can be treated under additional structural assumptions on the potentials; see Section~\ref{Section:MainResults}.
Finally, we note that our results are restricted to the two-dimensional case. In three dimensions, the main difficulty lies in the loss of the separation property, which would necessitate more delicate approximation techniques to control the nonlinearities, similar to those employed in \cite{Giorgini2022} for the AGG model. Furthermore, one cannot expect global strong solutions, but only local ones. The analysis can also be performed only for constant mobility functions, since for the convective bulk-surface Cahn--Hilliard subsystem, the uniqueness of weak solutions and the existence of strong solutions are currently known only in this setting. The extension to the three-dimensional case will be addressed in future work.

\textbf{Structure of this paper.} The remainder of this paper is structured as follows. In Section~\ref{Section:MathematicalSetting}, we present all the necessary mathematical notation and preliminaries. Then, in Section~\ref{Section:MainResults}, we state our main results regarding the existence of a unique global-in-time strong solution. Afterwards, Section~\ref{Section:LocalWellPosedness} is devoted to the construction of a global-in-time strong solution. Lastly, in Section~\ref{Section:Uniqueness}, we establish the uniqueness of the corresponding strong solution provided that $L \in(0,\infty]$.

\medskip

\section{Mathematical setting}
\label{Section:MathematicalSetting}

\subsection{General notation and function spaces}

We write $\N$ to denote the set of natural numbers excluding zero, whereas $\N_0 = \N\cup\{0\}$. For any real Banach space $X$ with $\norm{\cdot}_X$, its dual space is denoted by $X^\prime$. The corresponding duality pairing of elements $\phi\in X^\prime$ and $\zeta\in X$ is denoted by $\ang{\phi}{\zeta}_X$. If $X$ is a Hilbert space, we write $\scp{\cdot}{\cdot}_X$ to denote its inner product. Given $p\in[1,\infty]$ and an interval $I\subset[0,\infty)$, the Bochner space $L^p(I,X)$ consists of all Bochner measurable, $p$-integrable functions defined on $I$ with values in $X$. Moreover, the space $W^{1,p}(I;X)$ consists of all functions $f\in L^p(I;X)$ such that their vector-valued distributional derivative satisfies $\delt f\in L^p(I;X)$. In particular, we set $H^1(I;X) = W^{1,2}(I;X)$. Furthermore, $L^p_{\mathrm{uloc}}(I;X)$ denotes the space of functions $f\in L^p(I;X)$ such that
\begin{align*}
    \norm{f}_{L^p_{\mathrm{uloc}}(I;X)} \coloneqq \sup_{t\geq 0}\Big(\int_{I\cap [t,t+1)} \norm{f(s)}_X^p\ds\Big)^{\frac1p} < \infty.
\end{align*}
If $I\subset\R$ is a finite interval, we simply have $L^p_{\mathrm{uloc}}(I;X) = L^p(I;X)$. Further, the set of continuous functions $f:I\rightarrow X$ is denoted by $C(I;X)$. Moreover, $C_w(I;X)$ denotes the space of functions $f:I\to X$ which are continuous on $I$ with respect to the weak topology on $X$, i.e., the map $I\ni t\mapsto \ang{\phi}{f(t)}_X$ is continuous for all $\phi\in X^\prime$.

In the remainder of this section, let $\Om\subset\R^d$, $d=2,3$, be a bounded domain with sufficiently smooth boundary $\Ga\coloneqq\partial\Om$, and let $\n$ denote the exterior unit normal vector field on $\Gamma$. For any $s\geq 0$ and $p\in[1,\infty]$, the Lebesgue and Sobolev--Slobodeckij spaces for functions mapping from $\Om$ to $\R$ are denoted as $L^p(\Om)$ and $W^{s,p}(\Om)$, respectively. We write $\norm{\cdot}_{L^p(\Om)}$ and $\norm{\cdot}_{W^{s,p}(\Om)}$ to denote the standard norms on these spaces. If $p = 2$, we use the notation $H^s(\Om) = W^{s,2}(\Om)$, with the convention that $H^0(\Om)$ is identified with $L^2(\Om)$. For the Lebesgue and Sobolev--Slobodeckij spaces on the boundary $\Ga$, we use an analogous notation. The corresponding spaces of vector-valued functions mapping into $\R^d$ or matrix-valued functions mapping into $\R^{d\times d}$ are denoted by boldface letters, namely $\mathbf L^p$, $\mathbf W^{s,p}$, and $\mathbf H^s$. In particular, we write $\mathbf L^p_\tau(\Gamma)$, $\mathbf W^{s,p}_\tau(\Gamma)$ and $\mathbf H^s_\tau(\Gamma)$ to denote the corresponding spaces of tangential vector fields on $\Gamma$.

Moreover, we further introduce the first-order homogeneous Sobolev spaces
\begin{align*}
    \dot H^1\Omega) 
    &\coloneqq
    \big\{ u \in L^2_\mathrm{loc}(\Omega) 
    \,:\, \Grad u \in \mathbf L^2(\Omega) \big\}\big/_{\sim} \,,
    \\
    \dot H^1(\Gamma) 
    &\coloneqq
    \big\{ v \in L^2(\Gamma) 
    \,:\, \Gradg v \in \mathbf L^2(\Gamma) \big\}\big/_{\sim} \,.
\end{align*}
Here, $L^2_\mathrm{loc}(\Omega)$ is the collection of all functions, which belong to $L^2(K)$ for every compact subset $K\subset \Omega$. On the boundary, we simply consider $L^2(\Gamma)$, since $\Gamma$ is itself a compact submanifold.
Moreover, the notation $/_{\sim}\,$ indicates that functions differing only by an additive constant are considered equivalent and thus identified with each other. Consequently, the functionals ${\norm{\Grad\,\cdot}_{\mathbf L^2(\Omega)}}$ and ${\norm{\Gradg\,\cdot}_{\mathbf L^2(\Gamma)}}$ define norms on the spaces $\dot H^1(\Omega)$ and $\dot H^1(\Gamma)$, respectively.

Furthermore, we write $\mathbf{P}^\Gamma$ to denote the orthogonal projection of $\R^d$ into the tangent space of the submanifold $\Gamma$.
This means that $\mathbf{P}^\Gamma$ is a linear map, which can be represented as
\begin{align*}
    \mathbf{P}^\Gamma = \mathbf{I} - \n\otimes\n,
\end{align*}
where $\mathbf{I}\in\R^{d\times d}$ is the unity matrix.
For any vector field $\bw$ on $\Gamma$, we also use the notation
\begin{align*}
    [\bw]_\tau = \mathbf{P}^\Gamma(\bw).
\end{align*}

\subsection{Differential operators in the bulk and on the surface} \label{NOT:DIFFOP}
Let $f:\Omega\to\R$ and $g:\Gamma\to\R$ be scalar functions, let $\bv:\Omega\to \R^d$ and $\bw:\Gamma\to\R^d$ be vector fields, and let $\mathbf{A}:\Omega\to\R^{d\times d}$ and $\mathbf{B}:\Gamma\to\R^{d\times d}$ be matrix-valued functions. 
For now, we assume that the boundary $\Gamma$ and all these functions are sufficiently regular.
As usual, $\Grad f$, $\Div\,\bv$, and $\Lap f = \Div(\Grad f)$ denote the gradient of $f$, the divergence of $\bv$ and the Laplacian of $f$. 
Analogously, $\Gradg g$, $\Divg\,\bw$, and $\Lapg g = \Divg(\Gradg g)$ are the tangential gradient of $g$, the surface divergence of $\bw$ and the Laplace--Beltrami operator applied to $g$. 
We point out that $\Grad f$ and $\Gradg g$ are to be understood as column vectors in $\R^d$.
For bulk vector fields, the gradient is defined as
\begin{align*}
    \Grad \bv = 
    \begin{pmatrix}
        (\Grad \bv_1)^T \\ \vdots \\ (\Grad \bv_d)^T
    \end{pmatrix}
    \in \R^{d\times d},
\end{align*}
where $\bv_1,...,\bv_d$ denote the components of $\bv$. In other words, $\Grad\bv$ is the Jacobian of $\bv$. On the surface $\Gamma$, the projective (covariant) gradient of $\bw$ is defined as 
\begin{align*}
    \Gradg\bw = \mathbf{P}^\Gamma \Grad\tw \mathbf{P}^\Gamma,
\end{align*}
where $\tw$ is an extension of $\bw$ to an open neighborhood of $\Gamma$ in $\R^d$. However, the expression $\Gradg\bw$ does not depend on the precise choice of this extension. If $\bw$ is a tangential vector field, its gradient can also be expressed as
\begin{align*}
    \Gradg \bw = \mathbf{P}^\Gamma
    \begin{pmatrix}
        (\Gradg \bw_1)^T \\ \vdots \\ (\Gradg \bw_d)^T
    \end{pmatrix}
    \in \R^{d\times d},
\end{align*}
where $\bw_1,...,\bw_d$ denote the components of $\bw$.  
For the matrix-valued functions $\mathbf{A}$ and $\mathbf{B}$, we further define the divergences
\begin{align*}
    \Div\,\mathbf{A} = 
    \begin{pmatrix}
        \Div\, \mathbf{A}_{1\ast} \\ \vdots \\ \Div\, \mathbf{A}_{d\ast}
    \end{pmatrix},
    \qquad
    \Divg\,\mathbf{B} = \mathbf{P}^\Gamma
    \begin{pmatrix}
        \Divg\, \mathbf{B}_{1\ast} 
        \\ 
        \vdots 
        \\ \Divg\, \mathbf{B}_{d\ast}  
    \end{pmatrix}.
\end{align*}
Here, $\mathbf{A}_{1\ast},...,\mathbf{A}_{d\ast}$ and $\mathbf{B}_{1\ast},...,\mathbf{B}_{d\ast}$ denote the rows of $\mathbf{A}$ and $\mathbf{B}$, respectively.

\subsection{Bulk-surface product spaces}
For any real number $s\geq 0$ and $p\in[1,\infty]$, we set
\begin{align*}
	\mathcal{L}^p \coloneqq L^p(\Om)\times L^p(\Ga), \qquad \mathcal{W}^{s,p} \coloneqq W^{s,p}(\Om)\times W^{s,p}(\Ga).
\end{align*}
We write $\mathcal{H}^s = \mathcal{W}^{s,2}$ and identify $\mathcal{L}^2$ with $\mathcal{H}^s$. Note that $\mathcal{H}^s$ is a Hilbert space with respect to the inner product
\begin{align*}
	\bigscp{\scp{\phi}{\psi}}{\scp{\zeta}{\xi}}_{\mathcal{H}^s} \coloneqq \scp{\phi}{\zeta}_{H^s(\Om)} + \scp{\psi}{\xi}_{H^s(\Ga)} \qquad\text{for all~}(\phi,\psi), (\zeta,\xi)\in\mathcal{H}^s,
\end{align*}
and its induced norm $\norm{\cdot}_{\mathcal{H}^s} \coloneqq \scp{\cdot}{\cdot}_{\mathcal{H}^s}^{\frac12}$. We further recall that the duality pairing on $\mathcal{H}^1$ satisfies
\begin{align*}
	\bigang{\scp{\phi}{\psi}}{\scp{\zeta}{\xi}}_{\mathcal{H}^1} = \scp{\phi}{\zeta}_{L^2(\Om)} + \scp{\psi}{\xi}_{L^2(\Ga)}
\end{align*}
for all $(\zeta,\xi)\in\mathcal{H}^1$ if $(\phi,\psi)\in\mathcal{L}^2$. 

For every $\phi\in H^1(\Om)^\prime$, we denote by $\meano{\phi} = \abs{\Om}^{-1}\ang{\phi}{1}_{H^1(\Om)}$ its generalized mean value over $\Om$. If $\phi\in L^1(\Om)$, its spatial mean can simply be expressed as $\meano{\phi} = \abs{\Om}^{-1}\intO\phi\dx$. The spatial mean for a function $\psi$ on $\Ga$, denoted by $\meang{\psi}$, is defined similarly. We further define 
\begin{align*}
    \mathcal{L}^2_{(0)} \coloneqq L^2_{(0)}(\Om)\times L^2_{(0)}(\Ga),
\end{align*}
where
\begin{align*}
	L^2_{(0)}(\Om) 
    \coloneqq \big\{\phi\in L^2(\Om): \meano{\phi} = 0\big\}, 
    \qquad 
    L^2_{(0)}(\Ga) 
    \coloneqq \big\{\psi\in L^2(\Ga): \meang{\psi} = 0\big\}.
\end{align*}

Now, let $L\in[0,\infty]$ and $\beta\in\R$. We introduce the subspaces
\begin{align*}
	\mathcal{H}^1_{L,\beta} \coloneqq \begin{cases}
	\mathcal{H}^1, &\text{if~} L\in(0,\infty], \\
	\big\{(\phi,\psi)\in\mathcal{H}^1: \phi = \beta\psi \text{~a.e. on~}\Ga\big\}, &\text{if~} L = 0,
	\end{cases}
\end{align*}
endowed with the inner product $\scp{\cdot}{\cdot}_{\mathcal{H}^1_{L,\beta}} \coloneqq \scp{\cdot}{\cdot}_{\mathcal{H}^1}$ and its induced norm. The space $\mathcal{H}^1_{L,\beta}$ is a Hilbert space. Moreover, we define the product
\begin{align*}
	\bigang{\scp{\phi}{\psi}}{\scp{\zeta}{\xi}}_{\mathcal{H}^1_{L,\beta}} = \scp{\phi}{\zeta}_{L^2(\Om)} + \scp{\psi}{\xi}_{L^2(\Ga)}
\end{align*}
for all $(\phi,\psi), (\zeta,\xi)\in\mathcal{L}^2$. By means of the Riesz representation theorem, this product can be extended to a duality pairing on $(\mathcal{H}^1_{L,\beta})^\prime\times\mathcal{H}^1_{L,\beta}$, which will also be denoted as $\ang{\cdot}{\cdot}_{\mathcal{H}^1_{L,\beta}}$.

Next, for $(\phi,\psi)\in(\mathcal{H}^1_{L,\beta})^\prime$, we defined the generalized bulk-surface mean
\begin{align*}
	\mean{\phi}{\psi} \coloneqq \frac{\bigang{\scp{\phi}{\psi}}{\scp{\beta}{1}}_{\mathcal{H}^1_{L,\beta}}}{\beta^2\abs{\Om} + \abs{\Ga}},
\end{align*}
which reduces to
\begin{align*}
	\mean{\phi}{\psi} = \frac{\beta\abs{\Om}\meano{\phi} + \abs{\Ga}\meang{\psi}}{\beta^2\abs{\Om} + \abs{\Ga}}
\end{align*}
if $(\phi,\psi)\in\mathcal{L}^2$. We then define the closed linear subspaces
\begin{align*}
	\mathcal{V}^1_{L,\beta} = \begin{cases}
	\{(\phi,\psi)\in\mathcal{H}^1_{L,\beta}:  \mean{\phi}{\psi} = 0\}, &\text{if~}L\in[0,\infty), \\
	\{(\phi,\psi)\in\mathcal{H}^1: \meano{\phi} = \meang{\psi} = 0\}, &\text{if~} L = \infty.
	\end{cases}
\end{align*}
Note that these subspaces are Hilbert spaces with respect to the inner product $\scp{\cdot}{\cdot}_{\mathcal{H}^1}$. We further introduce the bilinear form
\begin{align*}
	\bigscp{\scp{\phi}{\psi}}{\scp{\zeta}{\xi}}_{L,\beta} \coloneqq &\intO \Grad\phi\cdot\Grad\zeta\dx + \intG \Gradg\psi\cdot\Gradg\xi\dG \\
	&\qquad + \chi(L) \intG (\beta\psi - \phi)(\beta\xi - \zeta)\dG
\end{align*}
for all $(\phi,\psi), (\zeta,\xi)\in\mathcal{H}^1$, where
\begin{align*}
	\chi:[0,\infty]\rightarrow [0,\infty), \qquad 
    \chi(r) \coloneqq 
    \begin{cases} 
    \frac{1}{r}, &\text{if~}r\in(0,\infty), \\
	0, &\text{if~} r\in\{0,\infty\}.
	\end{cases}
\end{align*}
Moreover, we set
\begin{align*}
	\norm{(\phi,\psi)}_{L,\beta} \coloneqq \bigscp{\scp{\phi}{\psi}}{\scp{\phi}{\psi}}_{L,\beta}^{\frac12}
\end{align*}
for all $\scp{\phi}{\psi}\in\mathcal{H}^1$. The bilinear form defines and inner product on $\mathcal{V}^1_{L,\beta}$, and $\norm{\cdot}_{L,\beta}$ defines a norm on $\mathcal{V}^1_{L,\beta}$, that is equivalent to the norm $\norm{\cdot}_{\mathcal{H}^1}$, see, e.g., \cite[Corollary A.2]{Knopf2021}. The space $(\mathcal{V}^1_{L,\beta},\scp{\cdot}{\cdot}_{L,\beta},\norm{\cdot}_{L,\beta})$ is a Hilbert space. Then, we define the spaces
\begin{align*}
	\mathcal{V}^{-1}_{L,\beta} \coloneqq \begin{cases} \{(\phi,\psi)\in(\mathcal{H}^1_{L,\beta})^\prime: \mean{\phi}{\psi} = 0\}, &\text{if~} L\in[0,\infty), \\
	\{(\phi,\psi)\in(\mathcal{H}^1)^\prime: \meano{\phi} = \meang{\psi} = 0\}, &\text{if~} L = \infty.
	\end{cases}
\end{align*}

\subsection{Spaces of tangential and divergence-free vector fields}

We introduce the spaces 
\begin{align*}
	&\mathbf{L}^2_\Div(\Om) \coloneqq \{ \bv\in\mathbf{L}^2(\Om): \Div\,\bv = 0 \ \text{in~}\Om, \ \bv\cdot\n = 0 \ \text{on~}\Ga\}, \\
	&\mathbf{L}^2_\Div(\Ga) \coloneqq \{ \bw\in\mathbf{L}^2(\Ga): \Divg\,\bw = 0, \ \bw\cdot\n = 0 \ \text{on~}\Ga\},
\end{align*}
and we set
\begin{align*}
    \mathbfcal{L}^2_\Div \coloneqq \mathbf{L}^2_\Div(\Om)\times\mathbf{L}^2_\Div(\Ga).
\end{align*}
Then, for $s\geq 0$ and $p\in[2,\infty]$, we define $\mathbf{W}^{s,p}_\Div(\Om) = \mathbf{W}^{s,p}(\Om)\cap \mathbf{L}^2_\Div(\Om)$. Analogously, we define $\mathbf{W}^{s,p}_\Div(\Ga) = \mathbf{W}^{s,p}(\Ga)\cap \mathbf{L}^2_\Div(\Ga)$, and then set $\mathbfcal{W}^{s,p}_\Div = \mathbf{W}^{s,p}_\Div(\Om)\times\mathbf{W}^{s,p}_\Div(\Ga)$. Furthermore, for $s>\frac 12$, we set 
\begin{align*}
    \mathbfcal{W}^{s,p}_0 &\coloneqq \{(\bv,\bw)\in\mathbfcal{W}^{s,p}: \bv\cdot\n = 0, \ \bv\vert_\Ga = \bw \ \text{on~}\Ga\},
    \\
    \mathbfcal{W}^{s,p}_{0,\Div} &\coloneqq \mathbfcal{W}^{s,p}_0 \cap \mathbfcal{W}^{s,p}_\Div.
\end{align*}
As before, we use the notation $\mathbfcal{H}^s_0 = \mathbfcal{W}^{s,2}_0$ as well as $\mathbfcal{H}^s_{0,\Div} = \mathbfcal{W}^{s,2}_{0,\Div}$.

\medskip
\subsection{Important tools.}

Throughout this paper, we will frequently use the following bulk-surface Poincar\'{e} inequality, which has been established in \cite[Lemma~A.1]{Knopf2021}.

\begin{lemma}\label{Lemma:Poincare}
	Let $K\in[0,\infty)$ and $\alpha,\beta\in\R$ with $\alpha\beta\abs{\Om} + \abs{\Ga} \neq 0$. Then there exists a constant $C_P > 0$ depending only on $K,\alpha,\beta$ and $\Om$ such that
	\begin{align}\label{Est:Poincare}
		\norm{(\phi,\psi)}_{\mathcal{L}^2} \leq C_P \norm{(\phi,\psi)}_{K,\alpha}
	\end{align}
	for all pairs $(\phi,\psi)\in\mathcal{H}^1_{K,\alpha}$ satisfying $\mean{\phi}{\psi} = 0$.
\end{lemma}

Furthermore, we recall the following trace interpolation inequality (see, e.g., \cite{Necas2012}):

\begin{lemma}\label{Prelim:Lemma:Interpol:Trace}
    There exists a constant $C > 0$ such that
    \begin{align}\label{InterpolEst:Trace:L^2}
        \norm{u}_{L^2(\Ga)}\leq C\norm{u}_{L^2(\Om)}^{\frac12}\norm{u}_{H^1(\Om)}^{\frac12} \qquad\text{for all~}u\in H^1(\Om).
    \end{align}
\end{lemma}

Next, we recall the following interpolation inequality:
\begin{lemma}
    Let $\Om\subset\R^2$ with sufficiently regular boundary $\Ga$. Then there exists a constant $C > 0$, such that it holds that
	\begin{align}\label{InterpolEst:L^4}
		\norm{(\zeta,\xi)}_{\mathcal{L}^4}\leq C\norm{(\zeta,\xi)}_{\mathcal{L}^2}^{\frac12}\norm{(\zeta,\xi)}_{\mathcal{H}^1}^{\frac12} \qquad\text{for all~}(\zeta,\xi)\in\mathcal{H}^1.
	\end{align}
\end{lemma}

Note that, as $\Gamma$ is $1$-dimensional submanifold of $\R^2$, this bulk-surface inequality readily follows directly from the classical Ladyzhenskaya interpolation inequality applied in the bulk and on the surface, respectively.

Lastly, we present a bulk-surface Korn-type inequality, which has been recently established in \cite[Lemma~5.1]{Knopf2025a}.

\begin{lemma}\label{Lemma:Korn}
    There exists a constant $C_K > 0$ such that
    \begin{align}\label{Est:Korn}
        \norm{(\bv,\bw)}_{\mathbfcal{H}^1} \leq C_K\big(\norm{\D\bv}_{\mathbf{L}^2(\Om)}^2 + \norm{\Dg\bw}_{\mathbf{L}^2(\Ga)}^2 + \norm{\bw}_{\mathbf{L}^2(\Ga)}^2\big)^{\frac12}
    \end{align}
    for all $(\bv,\bw)\in\mathbfcal{H}^1_0$.
\end{lemma}

In particular, as a consequence of Lemma~\ref{Lemma:Korn}, the norm
\begin{align*}
    \norm{(\bv,\bw)}_{\mathbfcal{H}^1_0} \coloneqq \big(\norm{\D\bv}_{\mathbf{L}^2(\Om)}^2 + \norm{(\Dg\bw}_{\mathbf{L}^2(\Ga)}^2 + \norm{\bw}_{\mathbf{L}^2(\Ga)}^2\big)^{\frac12}
\end{align*}
is  equivalent to the standard $\norm{\cdot}_{\mathbfcal{H}^1}$ norm on $\mathbfcal{H}^1_0$.

\medskip
\subsection{Elliptic problems with bulk-surface interaction.}
We end this section by presenting results for two different elliptic equations with bulk-surface interaction. The results are extracted from \cite{Knopf2021} and \cite{Stange2025}, respectively. We refer to these works for a more comprehensive exposition.

We start with the following problem
\begin{alignat}{2}\label{SYSTEM:EBS}
    -\Lap u &= f &&\qquad\text{in~}\Om, \nonumber \\
    -\Lapg v + \beta\deln u &= g &&\qquad\text{on~}\Ga, \\
    L\deln u &= \beta v - u &&\qquad\text{on~}\Ga, \nonumber
\end{alignat}
which has been studied in \cite{Knopf2021}. One can show that, employing the Lax--Milgram theorem in combination with Lemma~\ref{Lemma:Poincare}, for any $(f,g)\in\mathcal{V}_{L,\beta}^{-1}$, there exists a unique weak solution $(u,v)\in\mathcal{V}^1_{L,\beta}$ to \eqref{SYSTEM:EBS} in the sense that it satisfies the weak formulation
\begin{align*}
    \big((u,v),(\zeta,\xi))_{L,\beta} = \big\langle (f,g), (\zeta,\xi) \big\rangle_{\mathcal{H}^1_{L,\beta}}
\end{align*}
for all $(\zeta,\xi)\in\mathcal{H}^1_{L,\beta}$, where the bilinear form $(\cdot,\cdot)_{L,\beta}$ on $\mathcal{H}^1_{L,\beta}\times\mathcal{H}^1_{L,\beta}$ is defined as
\begin{align*}
    \big((u,v),(\zeta,\xi)\big)_{L,\beta} &\coloneqq \intO \Grad u\cdot\Grad\zeta\dx + \intG \Gradg v\cdot\Gradg\xi\dG \\
    &\quad + \chi(L)\intG (\beta v - u)(\beta\xi - \zeta)\dG.
\end{align*}
Consequently, there exists a constant $C > 0$, depending only on the parameters of the system, such that
\begin{align*}
    \norm{(u,v)}_{L,\beta} \leq C\norm{(f,g)}_{(\mathcal{H}^1_{L,\beta})^\prime}
\end{align*}
for all $(f,g)\in\mathcal{V}_{L,\beta}^{-1}$. This allows us to define a solution operator
\begin{align*}
    \mathcal{S}_{L,\beta}:\mathcal{V}_{L,\beta}^{-1}\rightarrow \mathcal{V}_{L,\beta}^1, \qquad (f,g)\mapsto \mathcal{S}_{L,\beta}(f,g) = \big(\mathcal{S}_{L,\beta}^\Om(f,g),\mathcal{S}_{L,\beta}^\Ga(f,g)\big)
\end{align*}
in the sense that, for given $(f,g)\in\mathcal{V}^{-1}_{L,\beta}$, $\mathcal{S}_{L,\beta}(f,g)$ is the unique weak solution to \eqref{SYSTEM:EBS}. Furthermore, we can define an inner product and its induced norm on $\mathcal{V}_{L,\beta}^{-1}$ via
\begin{align*}
    \big((f,g),(\zeta,\xi)\big)_{L,\beta,\ast} &\coloneqq \big(\mathcal{S}_{L,\beta}(f,g),\mathcal{S}_{L,\beta}(\zeta,\xi)\big)_{L,\beta}, \\
    \norm{(f,g)}_{L,\beta,\ast} &\coloneqq \big((f,g),(f,g)\big)_{L,\beta,\ast}^{\frac12}
\end{align*}
for all $(f,g), (\zeta,\xi)\in\mathcal{V}_{L,\beta}^{-1}$. This norm is equivalent to the standard norm $\norm{\cdot}_{(\mathcal{H}^1_{L,\beta})^\prime}$ on $\mathcal{V}^{-1}_{L,\beta}$.

Next, let $m_\Om, m_\Ga\in C([-1,1])$ be such that there exist two constants $m_\ast,m^\ast > 0$ with
\begin{align}\label{EBS:Mob:Assumptions}
    0 < m_\ast \leq m_\Om(s), m_\Ga(s) \leq m^\ast \qquad\text{for all~}s\in[-1,1].
\end{align}
Then, given two measurable functions $\phi:\Om\rightarrow[-1,1]$ and $\psi:\Ga\rightarrow[-1,1]$, we introduce the following elliptic problem
\begin{alignat}{2}\label{SYSTEM:EBS:MOB}
    -\Div(m_\Om(\phi)\Grad u) &= f &&\qquad\text{in~}\Om, \nonumber \\
    -\Divg(m_\Ga(\psi)\Gradg v) + \beta m_\Om(\phi)\deln u &= g &&\qquad\text{on~}\Ga, \\
    Lm_\Om(\phi)\deln u &= \beta v - u &&\qquad\text{on~}\Ga,\nonumber
\end{alignat}
which has been recently studied in \cite[Section~4]{Stange2025}. The system \eqref{SYSTEM:EBS:MOB} can be seen as a generalization of \eqref{SYSTEM:EBS}.
Similarly to the definition of $\mathcal{S}_{L,\beta}$, we can introduce a solution operator $\mathcal{S}_{L,\beta}[\phi,\psi]:\mathcal{V}_{L,\beta}^{-1}\rightarrow\mathcal{V}_{L,\beta}^1$ in the following way: for any $(f,g)\in\mathcal{V}_{L,\beta}^{-1}$, $\mathcal{S}_{L,\beta}[\phi,\psi](f,g)$ is the unique function satisfying
\begin{align}\label{EBS:MOB:WF}
    \big((u,v),(\zeta,\xi)\big)_{L,\beta,[\phi,\psi]} = \big\langle (f,g),(\zeta,\xi)\big\rangle_{\mathcal{H}^1_{L,\beta}}\qquad\text{for all~}(\zeta,\xi)\in\mathcal{H}^1_{L,\beta}.
\end{align}
Here, we have used the notation
\begin{align*}
    \big((u,v),(\zeta,\xi)\big)_{L,\beta,[\phi,\psi]} &\coloneqq \intO m_\Om(\phi)\Grad u\cdot\Grad\zeta\dx + \intG m_\Ga(\psi)\Gradg v\cdot\Gradg\xi\dG \\
    &\quad + \chi(L)\intG (\beta v - u)(\beta\xi - \zeta)\dG.
\end{align*}
Note that this bilinear form defines an inner product on $\mathcal{V}^1_{L,\beta}$. The corresponding norm on $\mathcal{V}^1_{L,\beta}$ is then defined as
\begin{align*}
    \norm{(f,g)}_{L,\beta,[\phi,\psi]} &\coloneqq \big((f,g),(f,g)\big)_{L,\beta,[\phi,\psi]}^{\frac12}
\end{align*}
for all $(f,g)\in\mathcal{V}^1_{L,\beta}$. We readily see that there exists a constant $C > 0$, that only depends on the parameters of the system, such that
\begin{align}\label{BSE:MOB:Est:Apriori}
    \norm{\mathcal{S}_{L,\beta}[\phi,\psi](f,g)}_{L,\beta,[\phi,\psi]} \leq C\norm{(f,g)}_{(\mathcal{H}^1_{L,\beta})^\prime}
\end{align}
for all $(f,g)\in\mathcal{V}^{-1}_{L,\beta}$. Moreover, on account of \eqref{EBS:Mob:Assumptions}, it holds that
\begin{align}\label{NormEquivalence:1}
    \min\{1,\sqrt{m_\ast}\}\norm{(f,g)}_{L,\beta,[\phi,\psi]} \leq \norm{(f,g)}_{L,\beta} \leq \max\{1,\sqrt{m^\ast}\}\norm{(f,g)}_{L,\beta,[\phi,\psi]}
\end{align}
for all $(f,g)\in\mathcal{H}^1_{L,\beta}$. In particular, the norms $\norm{\cdot}_{L,\beta}$ and $\norm{\cdot}_{L,\beta,[\phi,\psi]}$ are equivalent on $\mathcal{V}^1_{L,\beta}$.

Next, we define an inner product and its induced norm on $\mathcal{V}^{-1}_{L,\beta}$ by
\begin{align*}
    \big((f,g),(\zeta,\xi)\big)_{L,\beta,[\phi,\psi],\ast} &\coloneqq \big(\mathcal{S}_{L,\beta}[\phi,\psi](f,g),\mathcal{S}_{L,\beta}[\phi,\psi](\zeta,\xi)\big)_{L,\beta,[\phi,\psi]}, \\
    \norm{(f,g)}_{L,\beta,[\phi,\psi],\ast} &\coloneqq \big((f,g),(f,g)\big)_{L,\beta,[\phi,\psi],\ast}^{\frac12}
\end{align*}
for all $(f,g), (\zeta,\xi)\in\mathcal{V}^{-1}_{L,\beta}$. Using the respective weak formulation satisfied by the solution operators $\mathcal{S}_{L,\beta}$ and $\mathcal{S}_{L,\beta}[\phi,\psi]$, we deduce that the norms $\norm{\cdot}_{L,\beta,[\phi,\psi],\ast}$ and $\norm{\cdot}_{L,\beta,\ast}$ are equivalent on $\mathcal{V}^{-1}_{L,\beta}$ with
\begin{align}\label{NormEquivalence:EBS}
    \min\{1,\sqrt{m_\ast}\}\norm{(f,g)}_{L,\beta,[\phi,\psi],\ast} \leq \norm{(f,g)}_{L,\beta,\ast} \leq \max\{1,\sqrt{m^\ast}\}\norm{(f,g)}_{L,\beta,[\phi,\psi],\ast}
\end{align}
for all $(f,g)\in\mathcal{V}^{-1}_{L,\beta}$. In particular, we also obtain that $\norm{\cdot}_{L,\beta,[\phi,\psi],\ast}$ and $\norm{\cdot}_{(\mathcal{H}^1_{L,\beta})^\prime}$ are equivalent on $\mathcal{V}^{-1}_{L,\beta}$. Besides, in view of the weak formulation \eqref{EBS:MOB:WF} and the bulk-surface Poincar\'{e} inequality \eqref{Est:Poincare}, we have
\begin{align}\label{Est:fg:L^2:K}
    \norm{(f,g)}_{\mathcal{L}^2} \leq \max\{1,\sqrt{m^\ast}\}C_PC\norm{\mathcal{S}_{L,\beta}[\phi,\psi](f,g)}_{L,\beta}^{\frac12}\norm{(f,g)}_{K,\alpha}^{\frac12}
\end{align}
for all $(f,g)\in\mathcal{V}_{L,\beta}^{-1}\cap\mathcal{H}^1_{K,\alpha}$ and $K\in[0,\infty)$.

Now, we recall some elliptic estimates related to the problem \eqref{SYSTEM:EBS:MOB} for $d = 2$, whose proofs can be found in \cite[Section~4]{Stange2025}. To this end, we additionally assume that $m_\Om,m_\Ga\in C^1([-1,1])$ as well as $(\phi,\psi)\in\mathcal{W}^{2,4}$.  Let $(f,g)\in\mathcal{V}^{-1}_{L,\beta}\cap\mathcal{L}^p$ for some $2 \leq p < \infty.$ Then, we have
\begin{align}\label{Est:Sol:G:W2p}
    \norm{\mathcal{S}_{L,\beta}[\phi,\psi](f,g)}_{\mathcal{W}^{2,p}} \leq C\big( \norm{(\Grad\mathcal{S}_{L,\beta}^\Om[\phi,\psi](f,g)\cdot\Grad\phi, \Gradg\mathcal{S}_{L,\beta}^\Ga[\phi,\psi](f,g)\cdot\Gradg\psi)}_{\mathcal{L}^p} + \norm{(f,g)}_{\mathcal{L}^p}\big).
\end{align}
In particular, taking $p = 2$ in \eqref{Est:Sol:G:W2p}, and by using \eqref{InterpolEst:L^4}, we find
\begin{align*}
    \norm{\mathcal{S}_{L,\beta}[\phi,\psi](f,g)}_{\mathcal{H}^2} &\leq C\big( \norm{(\Grad\mathcal{S}_{L,\beta}^\Om[\phi,\psi](f,g)\cdot\Grad\phi, \Gradg\mathcal{S}_{L,\beta}^\Ga[\phi,\psi](f,g)\cdot\Gradg\psi)}_{\mathcal{L}^2} + \norm{(f,g)}_{\mathcal{L}^2}\big) \\
    &\leq C\big(\norm{(\Grad\phi,\Gradg\psi)}_{\mathbfcal{L}^4}\norm{(\Grad\mathcal{S}_{L,\beta}^\Om[\phi,\psi](f,g),\Gradg\mathcal{S}_{L,\beta}^\Ga[\phi,\psi](f,g))}_{\mathbfcal{L}^4} + \norm{(f,g)}_{\mathcal{L}^2}\big) \\
    &\leq C\big(\norm{(\Grad\phi,\Gradg\psi)}_{\mathbfcal{L}^2}^{\frac12}\norm{(\phi,\psi)}_{\mathcal{H}^2}^{\frac12}\norm{\mathcal{S}_{L,\beta}^\Om[\phi,\psi](f,g)}_{L,\beta}^{\frac12}\norm{\mathcal{S}_{L,\beta}[\phi,\psi](f,g)}_{\mathcal{H}^2}^{\frac12} \\
    &\quad + \norm{(f,g)}_{\mathcal{L}^2}\big).
\end{align*}
Consequently, we infer from Young's inequality that
\begin{align}\label{Est:Sol:G:H^2}
	\norm{\mathcal{S}_{L,\beta}[\phi,\psi](f,g)}_{\mathcal{H}^2} \leq C\big(\norm{(\Grad\phi,\Gradg\psi)}_{\mathbfcal{L}^2}\norm{(\phi,\psi)}_{\mathcal{H}^2}\norm{\mathcal{S}_{L,\beta}[\phi,\psi](f,g)}_{L,\beta} + \norm{(f,g)}_{\mathcal{L}^2}\big).
\end{align}
Lastly, if $(f,g)\in\mathcal{V}^{-1}_{L,\beta}\cap\mathcal{H}^1$ and $m_\Om, m_\Ga\in C^2([-1,1])$, then there exists $C > 0$ such that
\begin{align}\label{Est:Sol:G:H^3}
    &\norm{\mathcal{S}_{L,\beta}[\phi,\psi](f,g)}_{\mathcal{H}^3} \nonumber \\
    &\quad\leq C\big(1 + \mathbf{1}_{\{0\}}(L)\norm{(\phi,\psi)}_{\mathcal{H}^2}\big) \nonumber \\
    &\qquad\times\Bigg( \Bignorm{\bigg(\frac{f}{m_\Om(\phi)},\frac{g}{m_\Ga(\psi)}\bigg)}_{\mathcal{H}^1} \\
    &\qquad\qquad + \Bignorm{\bigg(\frac{m_\Om^\prime(\phi)\Grad\phi\cdot\Grad\mathcal{S}_{L,\beta}^\Om[\phi,\psi](f,g)}{m_\Om(\phi)},\frac{m_\Ga^\prime(\psi)\Gradg\psi\cdot\Gradg\mathcal{S}_{L,\beta}^\Ga[\phi,\psi](f,g)}{m_\Ga(\psi)}\bigg)}_{\mathcal{H}^1}\Bigg), \nonumber
\end{align}
where $\mathbf{1}_{\{0\}}$ denotes the indicator function of the set $\{0\}$.

\medskip
\subsection{Assumptions}

Now, we fix some assumptions that are supposed to hold throughout the remainder of this paper.

\begin{enumerate}[label=\textnormal{\bfseries(A\arabic*)}]
    \item  \label{Assumption:Domain} We consider a bounded domain $\emptyset\neq \Om\subset\R^2$ with $C^3$ boundary $\Ga\coloneqq\del\Om$. We further use the notation
    \begin{align*}
        Q\coloneqq \Om\times(0,\infty), \quad\Sigma\coloneqq\Ga\times(0,\infty).
    \end{align*}
    
    \item \label{Assumption:Constants} The constants occurring in the system \eqref{System} satisfy $\alpha\in[-1,1]$ and $\beta\in \R$ with $\alpha\beta\abs{\Omega} + \abs{\Gamma} \neq 0$. 
    
    \item \label{Assumption:Density} The density functions $\rho$ and $\sigma$ are given by
    \begin{align*}
    	\rho(s) = \frac{\tilde{\rho}_2 - \tilde{\rho}_1}{2}s + \frac{\tilde{\rho}_1 + \tilde{\rho}_2}{2} \quad\text{and}\quad\sigma(s) = \frac{\tilde{\sigma}_2 - \tilde{\sigma}_1}{2}s + \frac{\tilde{\sigma}_1 + \tilde{\sigma}_2}{2} \qquad\text{for all~}s\in[-1,1].
    \end{align*}
    Here, $\tilde{\rho}_1, \tilde{\rho}_2>0$ and $\tilde{\sigma}_1, \tilde{\sigma}_2>0$ are the specific densities of the two fluid components in the bulk and on the surface, respectively. Moreover, we use the notation
    \begin{alignat*}{2}
        &\rho_\ast \coloneqq \min\{\tilde\rho_1,\tilde\rho_2\}, \qquad 
        &&\rho^\ast \coloneqq \max\{\tilde\rho_1,\tilde\rho_2\},
        \\
        &\sigma_\ast \coloneqq \min\{\tilde\sigma_1,\tilde\sigma_2\}, \qquad 
        &&\sigma^\ast \coloneqq \max\{\tilde\sigma_1,\tilde\sigma_2\},
    \end{alignat*}
    By this definition, we clearly have
    \begin{align*}
        0 < \rho_\ast \leq \rho(s) \leq \rho^\ast \quad\text{and}\quad 0 < \sigma_\ast \leq \sigma(s) \leq \sigma^\ast \qquad\text{for all~}s\in[-1,1].
    \end{align*}
    
    \item \label{Assumption:Coefficients} For the coefficients we assume $m_\Om, m_\Ga\in C^2([-1,1])$, $\nu_\Om, \nu_\Ga\in C^2([-1,1])$, $\gamma\in C^{0,1}([-1,1]^2)$. Moreover, we postulate the existence of positive constants $m_\ast,m^\ast,\nu_\ast,\nu^\ast,\gamma_\ast,\gamma^\ast$ such that for all $s\in[-1,1]$,
    \begin{align}\label{Assumption:Mobility:Bound}
        0 < m_\ast \leq m_\Om(s), m_\Ga(s) \leq m^\ast \quad\text{and}\quad 0 < \nu_\ast \leq \nu_\Om(s), \nu_\Ga(s) \leq \nu^\ast \qquad\text{for all~}s\in[-1,1],
    \end{align}
    and
    \begin{align*}
    	0 < \gamma_\ast \leq \gamma(s,r) \leq \gamma^\ast \qquad\text{for all~}(s,r)\in[-1,1]^2.
    \end{align*}
    
    \item \label{Assumption:Potential} The bulk free energy density is given by
    \begin{align*}
    	F(s) = F_0(s) - \frac{c_F}{2}s^2
    \end{align*}
    for some $c_F\in\R$, satisfying $F_0 \in C([-1,1])\cap C^3(-1,1)$. We assume that
    \begin{align*}
    	\lim_{s\searrow -1} F_0^\prime(s) = - \infty \quad\text{and~}\quad \lim_{s\nearrow 1} F_0^\prime(s) = + \infty,
    \end{align*}
    and the existence of a constant $\Theta > 0$ such that
    \begin{align}\label{Pot:F_0:Convex}
    	F_0^{\prime\prime}(s) \geq \Theta
        \quad\text{for all $s\in(-1,1)$.}
    \end{align}
    As usual we extend $F_0(s) = +\infty$ for any $s\not\in[-1,1]$. Without loss of generality, we assume $F_0(0) = F_0^\prime(0) = 0$. In particular, we thus have $F_0(s) \geq 0$ for all $s\in[-1,1]$ due to the strong convexity of $F_0$. 
    Analogously, the free energy density on the boundary is given by
    \begin{align*}
    	G(s) = G_0(s) - \frac{c_G}{2}s^2
    \end{align*}
    for some $c_G\in\R$ and $G_0\in C([-1,1])\cap C^3(-1,1)$. We assume that $G_0$ satisfies analogous assumptions to $F_0$.
    \item \label{Assumption:DominationProperty} There exist constants $\kappa_1 > 0$ and $\kappa_2 > 0$ such that
    \begin{align*}
    	\abs{F_0^\prime(\alpha s)} \leq \kappa_1\abs{G_0^\prime(s)} + \kappa_2 \qquad\text{for all~}s\in(-1,1).
    \end{align*}
    \item \label{Assumption:Potential:Growth} We assume that one of the following conditions hold:
    \begin{enumerate}[label=\textnormal{\bfseries(A.7.\arabic*)}]
        \item There exist constant $C_\sharp > 0$ and $\gamma_\sharp\in[1,2)$ such that
        \begin{align*}
            F_0^{\prime\prime}(s) \leq C_\sharp \e^{C_\sharp\abs{F_0^\prime(s)}^{\gamma_\sharp}} \qquad\text{for all~}s\in(-1,1).
        \end{align*}
        \item As $\delta\searrow 0$, for some $\kappa > \frac12$, it holds that
        \begin{align*}
            \frac{1}{F_0^\prime(1-2\delta)} = O\bigg(\frac{1}{\abs{\ln \delta}^\kappa}\bigg), \qquad \frac{1}{\abs{F_0^\prime(-1+2\delta)}} = O\bigg(\frac{1}{\abs{\ln\delta}^\kappa}\bigg).
        \end{align*}
    \end{enumerate}
\end{enumerate}

\medskip

\section{Main results}
\label{Section:MainResults}

In this section, we formulate the main results of this paper. Our first main result is concerned with the existence of global strong solutions to system \eqref{System}.
\begin{theorem}\label{Theorem:GlobalExistence}
        Let the assumptions \ref{Assumption:Domain}-\ref{Assumption:Potential:Growth} hold, and let $K\in(0,\infty)$ and $L\in[0,\infty]$. Assume that $(\bv_0,\bw_0)\in\mathbfcal{H}^1_\Div$, and let $(\phi_0,\psi_0)\in\mathcal{H}^1$ be such that
        \begin{subequations}\label{Assumption:InitialCondition}
        	\begin{align}\label{Assumption:InitalCondition:Int}
            	\norm{\phi_0}_{L^\infty(\Om)} \leq 1 \quad\text{and}\quad \norm{\psi_0}_{L^\infty(\Ga)} \leq 1.
        	\end{align}
        	Furthermore, we assume that
        	\begin{align}\label{Assumption:InitalCondition:L}
            	\beta\mean{\phi_0}{\psi_0},\ \mean{\phi_0}{\psi_0}\in(-1,1), \qquad\text{if~}L\in[0,\infty),
        	\end{align}
        	and
        	\begin{align}\label{Assumption:InitalCondition:Inf}
            	\meano{\phi_0}, \ \meang{\psi_0}\in(-1,1), \qquad\text{if~} L = \infty.
        	\end{align}
        \end{subequations}
        We further assume that the following compatibility assumption holds:
        \begin{enumerate}[label=\textnormal{\bfseries(C)},topsep=0ex,leftmargin=*]
        \item \label{cond:MT:0} There exists $\scp{\mu_0}{\theta_0}\in\mathcal{H}^1_{L,\beta}$ such that for all $\scp{\eta}{\vartheta}\in\mathcal{H}^1$ it holds
        \begin{align*}
        \begin{aligned}
            &\intO\mu_0\eta\dx + \intG\theta_0\vartheta\dG 
            \\
            &= \intO\Grad\phi_0\cdot\Grad\eta + F^\prime(\phi_0)\eta\dx + \intG\Gradg\psi_0\cdot\Gradg\vartheta + G^\prime(\psi_0)\vartheta\dG 
            \\
            &\quad + \chi(K)\intG(\alpha\psi_0 - \phi_0)(\alpha\vartheta - \eta)\dG.
        \end{aligned}
        \end{align*}
    \end{enumerate}
    Moreover, in the case $L = 0$, we suppose that the following compatibility assumption on the densities
    \begin{align}\label{Densities:Comp:Strong}
        \beta(\tilde\sigma_2 - \tilde\sigma_1) = \tilde\rho_2 - \tilde\rho_1
    \end{align}
    holds.
    Then, there exists a global strong solution $(\bv,\bw,\phi,\psi,\mu,\theta)$ to \eqref{System} satisfying the regularities
    \begin{align}
        (\bv,\bw)&\in BC([0,\infty);\mathbfcal{H}^1_{0,\Div})\cap L^2_{\mathrm{uloc}}([0,\infty);\mathbfcal{H}^2)\cap H^1_{\mathrm{uloc}}(0,\infty;\mathbfcal{L}^2_\Div), \\
        (p,q)&\in L^2_{\mathrm{uloc}}(0,\infty;\mathcal{H}^1\cap\mathcal{L}^2_{(0)}), \\
        (\phi,\psi)&\in L^\infty(0,\infty;\mathcal{H}^3), \qquad (\delt\phi,\delt\psi)\in L^\infty(0,\infty;(\mathcal{H}^1_{L,\beta})^\prime) \cap L^2_{\mathrm{uloc}}(0,\infty;\mathcal{H}^1), \\
        (\mu,\theta)&\in BC([0,\infty);\mathcal{H}^1_{L,\beta})\cap L^2_{\mathrm{uloc}}([0,\infty);\mathcal{H}^3)\cap H^1_{\mathrm{uloc}}(0,\infty;(\mathcal{H}^1)^\prime), \\
        (F^\prime(\phi&), G^\prime(\psi)), \ (F^{\prime\prime}(\phi), G^{\prime\prime}(\psi))\in L^\infty(0,\infty;\mathcal{L}^p)
    \end{align}
    for any $2 \leq p < \infty$, in the sense that the equations \eqref{System:1}, \eqref{System:4} and \eqref{System:5} are satisfied a.e. on $\Om\times(0,\infty)$, while the equations \eqref{System:2} and \eqref{System:6}-\eqref{System:7} and the boundary conditions \eqref{System:3} and \eqref{System:8} are satisfied a.e. on $\Ga\times(0,\infty)$. Moreover, the solution is such that $(\bv,\bw)\vert_{t=0} = (\bv_0,\bw_0)$ a.e. in $\Om\times\Ga$ and $(\phi,\psi)\vert_{t=0} = (\phi_0,\psi_0)$ a.e. in $\Om\times\Ga$.
\end{theorem}

\begin{remark}
    The case $K = 0$ is not admissible in the setting of non-degenerate mobility functions. This issue appears in the analysis of the well-posedness of weak solutions to the convective bulk-surface Cahn--Hilliard subsystem \eqref{EQ:CONV:SYSTEM}, as discussed in Appendix~\ref{Section:ConvCH} (see also \cite{Stange2025}). In fact, only for $K \in (0,\infty)$ does the corresponding weak solution $(\phi,\psi,\mu,\theta)$ satisfy the enhanced regularity $(\phi,\psi)\in L^4_{\mathrm{uloc}}([0,\infty);\mathcal{H}^2)$ (see \cite[Theorem~3.3]{Giorgini2025}), which is a crucial step in the proof of uniqueness of weak solutions (see also \cite{Conti2024}). If $K=0$, one only obtains $(\phi,\psi)\in L^3_{\mathrm{uloc}}([0,\infty)\mathcal{H}^2)$, which appears insufficient to complete the argument.
    By contrast, in the case of constant mobilities, say $m_\Om = m_\Ga \equiv 1$, the choice $K=0$ becomes admissible since the aforementioned higher-order norms are not required in this setting. Indeed, for constant mobilities, one can establish the existence of a strong  solution with the additional regularity
    \begin{align*}
    (\bv,\bw) &\in L^2(0,\infty;\mathbfcal{H}^2)\cap H^1(0,\infty;\mathbfcal{L}^2_\Div), \\ 
    (p,q) &\in L^2(0,\infty;\mathcal{H}^1\cap\mathcal{L}^2_{(0)}), \\
    (\delt\phi,\delt\psi) &\in L^2(0,\infty;\mathcal{H}^1), \\
    (\Grad\mu,\Gradg\theta) &\in L^2(0,\infty;\mathbfcal{H}^2),
    \qquad (\mu,\theta)\in H^1(0,\infty;(\mathcal{H}^1_{K,\alpha})^\prime).
    \end{align*}
    These enhanced regularities can be achieved by using the regularity theory for the convective bulk-surface Cahn--Hilliard equation with constant mobility functions as developed in \cite{Knopf2025} and \cite{Giorgini2025}.
\end{remark}

\begin{remark}
    We comment on the compatibility condition \eqref{Densities:Comp:Strong}, which was needed in our analysis in the case $L = 0$. The compatibility condition \eqref{Densities:Comp:Strong} ensures that, formally,
        \begin{align}\label{Term:Comp:Strong}
        \frac12(\J\cdot\n)\bw - \frac12(\J_\Ga\cdot\n)\bw
        = \Big(\beta\frac{\tilde\sigma_2 - \tilde\sigma_1}{2} - \frac{\tilde\rho_2 - \tilde\rho_1}{2}\Big)\deln\mu\,\bw
        = 0 \qquad\text{on }\Ga,
        \end{align}
    which arises as a term on the right-hand side of the non-conservative form \eqref{ConsForm:w}. This assumption differs from the condition used in \cite[Theorem~4.2]{Knopf2025a} to guarantee the existence of weak solutions.
    Namely, in \cite{Knopf2025a}, it was assumed that
        \begin{align}\label{Densities:Comp:Weak}
            \beta(\tilde\sigma_2 - \tilde\sigma_1) = -(\tilde\rho_2 - \tilde\rho_1),
        \end{align}
    which implies, formally,
        \begin{align}\label{Term:Comp:Weak}
        \frac12(\J\cdot\n)\bw + \frac12(\J_\Ga\cdot\n)\bw
        = \Big(\beta\frac{\tilde\sigma_2 - \tilde\sigma_1}{2} + \frac{\tilde\rho_2 - \tilde\rho_1}{2}\Big)\deln\mu\,\bw
        = 0 \qquad\text{on }\Ga,
        \end{align}
    which instead appears on the right-hand side of \eqref{eqs:NSCH:pressure:w}. In the analysis of \cite{Knopf2025a}, the weak formulation of \eqref{eqs:NSCH:pressure} is employed, so the condition \eqref{Densities:Comp:Weak} is necessary to eliminate \eqref{Term:Comp:Weak}, as otherwise this term couldn't be handled. In contrast, to prove Theorem~\ref{Theorem:GlobalExistence} we rely on the non-conservative formulation \eqref{eqs:NSCH:NC}, which allows us to derive suitable bounds for the weak time derivative of the velocity fields. Since we could not find a way to control the boundary contribution involving $\deln\mu$ from the fluxes $\J$ and $\J_\Ga$, we imposed \eqref{Densities:Comp:Strong}, which ensures that these terms vanish, see \eqref{Term:Comp:Strong}. This explains why the two compatibility conditions \eqref{Densities:Comp:Strong} and \eqref{Densities:Comp:Weak} differ by a minus sign. Note that these conditions are only required in the case $L = 0$, since when $L\in(0,\infty]$, the boundary condition \eqref{System:8}$_2$ allows us to express the normal derivative $\deln\mu$ as $\tfrac{1}{Lm_\Om(\phi)}(\beta\theta - \mu)$.
    
\end{remark}

Our next main result is concerned with the uniqueness of strong solutions to \eqref{System}
\begin{theorem}\label{Theorem:Uniqueness}
    Let the assumptions from Theorem~\ref{Theorem:GlobalExistence} hold. Additionally, assume that $L\in(0,\infty]$. Then, the global strong solution to \eqref{System} is unique.
\end{theorem}

\begin{remark}
    We remark on the assumption $L\in(0,\infty]$ imposed in Theorem~\ref{Theorem:Uniqueness}. This limitation again stems from the boundary condition \eqref{System:8}$_2$: when $L=0$, it is not possible to reformulate the normal derivative $\deln \mu$, and we cannot control this term otherwise. Thus, our analysis does not extend to the case $L = 0$. This difficulty already appears in the bulk-surface Cahn--Hilliard system, even without coupling to the Navier--Stokes equations, see Remark~\ref{REM:REG:WEAK}.
    An exception occurs when $K=L=0$. In this setting, assuming $\alpha\neq 0$, $\beta\neq 0$, and the following compatibility condition on the potentials $F$ and $G$,
    \begin{align*}
        F(\alpha s) = \alpha\beta G(s) \qquad\text{for all~}s\in[-1,1],
    \end{align*}
    the pair $(F^\prime(\phi), G^\prime(\psi))$ then belongs to $\mathcal{H}^1_{0,\beta}$, and hence it is an admissible test function in \eqref{WF:CCH:PP}. This choice makes the problematic term containing the normal derivative vanish, which allows us to circumvent the obstruction mentioned above.
\end{remark}

\medskip

\section{Existence of a global-in-time strong solution}
\label{Section:LocalWellPosedness}

In this section, we present the proof of Theorem~\ref{Theorem:GlobalExistence} regarding the existence of a global strong solution to \eqref{System}.

\begin{proof}[Proof of Theorem~\ref{Theorem:GlobalExistence}.]
Our proof is based on a semi-Galerkin discretization of the Navier--Stokes subsystem combined with Schauder's fixed point theorem. This method was already used to prove the existence of weak solutions to \eqref{System} in \cite{Knopf2025a}. In this way, we prove the existence of suitable approximate solutions to \eqref{System}. The main part of the proof then relies on establishing higher regularity estimates on these approximate solutions that are uniform with respect to the approximation parameter. Here, we rely on regularity theory for a convective bulk-surface Cahn--Hilliard system with non-degenerate mobility as well as regularity theory for a bulk-surface Stokes system with non-constant coefficients as presented in Appendix~\ref{Section:ConvCH} and Appendix~\ref{Section:Stokes}, respectively. These uniform estimates then allow us to pass to the limit in the corresponding weak formulation of the approximate problem.

\textbf{Step 1: Definition of the approximate problem.}
We consider the bulk-surface Stokes operator (see \cite[Section~5.3]{Knopf2025a})
\begin{align*}
    \mathbfcal{A}:D(\mathbfcal{A})\subset \mathbfcal{L}^2_\Div&\rightarrow\mathbfcal{L}^2_\Div, \\
    (\bv,\bw)&\mapsto \big(-\mathbf{P}_\Div^\Om(\Div(2\D\bv)), -\mathbf{P}_\Div^\Ga(\Divg(2\Dg\bw) + 2[\D\bv\,\n]_\tau + \bw)\big),
\end{align*}
where $\mathbf{P}_\Div^\Om$ and $\mathbf{P}_\Div^\Ga$ are the Leray projections on $\Om$ and $\Ga$, respectively, and the domain $\mathcal{D}(\mathbfcal{A})$ is given by
\begin{align*}
    D(\mathbfcal{A}) = \mathbfcal{H}^2_{0,\Div}.
\end{align*}
It was shown in \cite[Theorem~5.9]{Knopf2025a} that $\mathbfcal{A}$ has countably many positive eigenvalues $\{\widetilde\lambda_j\}_{j\in\N}$ with corresponding eigenfunctions $\{(\tv_j,\tw_j)\}_{j\in\N}$. These eigenfunctions are determined by
\begin{align*}
    &\intO 2\D\tv_j:\D\wv\dx + \intG 2\Dg\tw_j:\Dg\ww\dG + \intG \tw_j\cdot\ww\dG \\
    &\qquad = \widetilde\lambda_j\intO \tv_j\cdot\wv\dx + \widetilde\lambda_j \intG \tw_j\cdot\ww\dG
\end{align*}
for all $(\wv,\ww)\in\mathbfcal{H}^1_{0,\Div}$, and can be chosen in such a way that they form an orthonormal basis of $\mathbfcal{L}^2_\Div$. In addition, the eigenfunctions $\{(\tv_j,\tw_j)\}_{j\in\N}$ form a complete orthogonal system of $D(\mathbfcal{A})$. For more details, we refer to \cite[Section~5.3]{Knopf2025a}. For any $k\in\N$, we introduce the finite-dimensional subspace
\begin{align*}
    \mathbfcal{V}_k \coloneqq \mathrm{span}\,\{(\tv_1,\tw_1),\ldots,(\tv_k,\tw_k)\}\subset\mathbfcal{H}^1_{0,\Div},
\end{align*}
and denote with $\mathbfcal{P}_{\mathbfcal{V}_k}$ the $\mathbfcal{L}^2_\Div$-orthogonal projection  onto $\mathbfcal{V}_k$. Since $\Omega$ is of class $C^3$, regularity theory for the bulk-surface Stokes operator (see \cite[Theorem~5.7]{Knopf2025a}) yields that $(\tv_j,\tw_j)\in\mathbfcal{W}^{2,r}$ for all $r\in[2,\infty)$ and all $j\in\N$. As a consequence, the following inverse Sobolev embedding inequalities hold for all $(\wv,\ww)\in\mathbfcal{V}_k$:
\begin{align}
    \label{Est:InverseSobolev}
    \begin{split}
        \norm{(\wv,\ww)}_{\mathbfcal{H}^1} &\leq C_k \norm{(\wv,\ww)}_{\mathbfcal{L}^2}, \quad \norm{(\wv,\ww)}_{\mathbfcal{H}^2} \leq C_k \norm{(\wv,\ww)}_{\mathbfcal{L}^2}, \\
        \norm{(\wv,\ww)}_{\mathbfcal{W}^{2,r}} &\leq C_{k,r}\norm{(\wv,\ww)}_{\mathbfcal{L}^2}.
    \end{split}
\end{align}
Now, fix $T > 0$. For any $k\in\N$, we seek an approximate solution $(\bv_k,\bw_k,\phi_k,\psi_k,\mu_k,\theta_k)$ to system \eqref{System} such that
\begin{subequations}\label{Reg:ApproxSolution}
    \begin{align}
        (\bv_k,\bw_k)&\in C^1([0,T];\mathbfcal{V}_k), \\
        (\phi_k,\psi_k)&\in L^\infty(0,T;\mathcal{W}^{2,p}), \quad(\delt\phi_k,\delt\psi_k)\in L^\infty(0,T;(\mathcal{H}^1_{L,\beta})^\prime)\cap L^2(0,T;\mathcal{H}^1), \\
        (\mu_k,\theta_k)&\in C([0,T];\mathcal{H}^1_{L,\beta})\cap L^2(0,T;\mathcal{H}^3)\cap H^1(0,T;(\mathcal{H}^1)^\prime), \\
        (F^\prime(\phi_k)&, G^\prime(\psi_k)), \ (F^{\prime\prime}(\phi_k),G^{\prime\prime}(\psi_k))\in L^\infty(0,T;\mathcal{L}^p),
    \end{align}
\end{subequations}
    for any $2 \leq p < \infty$, with $\abs{\phi_k} < 1$ a.e. in $Q_T$ and $\abs{\psi_k} < 1$ a.e. on $\Sigma_T$, in the sense that, a.e. on $[0,T]$, it holds that
    \begin{align}
        &\intO\rho(\phi_k)\delt\bv_k\cdot\wv\dx + \intG\sigma(\psi_k)\delt\bw_k\cdot\ww\dG \nonumber \\
        &\qquad + \intO\rho(\phi_k)(\bv_k\cdot\Grad)\bv_k\cdot\wv\dx + \intG\sigma(\psi_k)(\bw_k\cdot\Gradg)\bw_k\cdot\ww\dG \nonumber \\
        &\qquad + \frac{\tilde\rho_2 - \tilde\rho_1}{2}\intO m_\Om(\phi_k)(\Grad\mu_k\cdot\Grad)\bv_k\cdot\wv\dx + \frac{\tilde\sigma_2 - \tilde\sigma_1}{2}\intG m_\Ga(\psi_k)(\Gradg\theta_k\cdot\Gradg)\bw_k\cdot\ww\dG \label{WF:VW:DISCR} \\ 
        &\qquad  + \intO2\nu_\Om(\phi_k)\D\bv_k:\D\wv\dx + \intG2\nu_\Ga(\psi_k)\Dg\bw_k:\Dg\ww\dG + \intG \gamma(\phi_k,\psi_k)\bw_k\cdot\ww\dG \nonumber \\
        &\quad = \intO\mu_k\Grad\phi_k\cdot\wv\dx + \intG\theta_k\Gradg\psi_k\cdot\ww\dG \nonumber \\
        &\qquad + \frac{\chi(L)}{2}\Big(\beta\frac{\tilde\sigma_2 - \tilde\sigma_1}{2} - \frac{\tilde\rho_2 - \tilde\rho_1}{2}\Big)\intG (\beta\theta_k - \mu_k)\bw_k\cdot\ww\dG, \nonumber
    \end{align}
    for all $(\wv,\ww)\in\mathbfcal{V}_k$, and
    \begin{align}\label{WF:PP:DISCR}
    	&\bigang{(\delt\phi_k,\delt\psi_k)}{(\zeta,\xi)}_{\mathcal{H}^1_{L,\beta}} - \intO \phi_k\bv_k\cdot\Grad\zeta\dx - \intG 									\psi_k\bw_k\cdot\Gradg\xi\dG \\
    	&\quad = - \intO m_\Om(\phi_k)\Grad\mu_k\cdot\Grad\zeta\dx - \intG m_\Ga(\psi_k)\Gradg\theta_k\cdot\Gradg\xi \dG - \chi(L)\intG (\beta\theta_k - \mu_k)				(\beta\xi - \zeta)\dG \nonumber
    \end{align}
    for all $(\zeta,\xi)\in\mathcal{H}^1_{L,\beta}$, where
    \begin{subequations}\label{WF:MT:DISCR:STR}
    	\begin{alignat}{2}
    		&\mu_k = -\Lap\phi_k + F^\prime(\phi_k) &&\qquad\text{a.e. in~}Q_T, \label{WF:MT:DEF:MU}\\
    		&\theta_k = -\Lapg\psi_k + G^\prime(\psi_k) + \alpha\deln\phi_k &&\qquad\text{a.e. on~}\Sigma_T, \label{WF:MT:DEF:THETA}\\
    		&K\deln\phi_k = \alpha\psi_k - \phi_k, \qquad K\in(0,\infty), &&\qquad\text{a.e. on~}\Sigma_T,
    	\end{alignat}
    \end{subequations}
together with the initial conditions
\begin{align}\label{Approx:Problem:IC}
    (\bv_k,\bw_k)\vert_{t=0} = \mathbfcal{P}_{\mathbfcal{V}_k}(\bv_0,\bw_0), \qquad (\phi_k,\psi_k)\vert_{t=0} = (\phi_0, \psi_0)  \qquad \text{in~}\Om\times\Ga. 
\end{align}
Here, we have used the notation
\begin{align*}
    Q_T \coloneqq \Om \times(0,T), \qquad \Sigma_T \coloneqq \Ga\times(0,T).
\end{align*}
We note that the weak formulation corresponding to \eqref{WF:MT:DISCR:STR} reads as
    \begin{align}
    	&\intO\mu_k\eta\dx + \intG \theta_k\vartheta\dG \nonumber \\
    	&\quad = \intO \Grad\phi_k\cdot\Grad\eta + F^\prime(\phi_k)\eta\dx + \intG \Gradg\psi_k\cdot\Gradg\vartheta + G^\prime(\psi_k)\vartheta\dG \label{WF:MT:DISCR} \\
    	&\qquad + \chi(K) \intG (\alpha\psi_k - \phi_k)(\alpha\vartheta - \eta)\dG \nonumber
    \end{align}
a.e. on $[0,T]$ for all $(\eta,\vartheta)\in\mathcal{H}^1$.

To show the existence of an approximation solution satisfying \eqref{Reg:ApproxSolution}-\eqref{Approx:Problem:IC}, we perform a fixed-point argument as in \cite{Knopf2025a}. Toward this aim, let $(\bv_\ast,\bw_\ast)\in H^1(0,T;\mathbfcal{V}_k)$. We consider the convective bulk-surface Cahn--Hilliard equation
\begin{subequations}\label{CCH:Fixed}
    \begin{alignat}{2}
        &\delt\phi_k + \Div(\phi_k\bv_\ast) = \Div(m_\Om(\phi_k)\Grad\mu_k) &&\qquad\text{in~}Q_T, \label{CCH:Fixed:1} \\
        &\mu = -\Lap\phi_k + F^\prime(\phi_k) &&\qquad\text{in~}Q_T, \label{CCH:Fixed:2} \\
        &\delt\psi_k + \Divg(\psi_k\bw_\ast) = \Divg(m_\Ga(\psi_k)\Gradg\theta_k) - \beta m_\Om(\phi_k)\deln\mu_k &&\qquad\text{on~}\Sigma_T, \label{CCH:Fixed:3} \\
        &\theta_k = -\Lapg\psi_k + G^\prime(\psi_k) + \alpha\deln\phi_k &&\qquad\text{on~}\Sigma_T, \label{CCH:Fixed:4} \\
        &K\deln\phi_k = \alpha\psi_k - \phi_k \qquad K\in (0,\infty), &&\qquad\text{on~}\Sigma_T, \label{CCH:Fixed:5} \\
        &\begin{cases} 
        L m_\Om(\phi_k) \deln\mu_k = \beta\theta_k - \mu_k &\text{if} \  L\in[0,\infty), \\
        m_\Om(\phi_k)\deln\mu_k = 0 &\text{if} \ L=\infty
        \end{cases} &&\qquad\text{on~}\Sigma_T, \label{CCH:Fixed:6} \\
        &\phi_k\vert_{t=0} = \phi_0 &&\qquad\text{in~}\Om, \label{CCH:Fixed:7} \\
        &\psi_k\vert_{t=0} = \psi_0 &&\qquad\text{on~}\Ga. \label{CCH:Fixed:8}
    \end{alignat}
\end{subequations}
In view of Theorem~\ref{Theorem:CCH:Strong}, there exists a unique strong solution to \eqref{CCH:Fixed} such that
\begin{subequations}\label{Reg:ConvCH}
    \begin{align}
        (\phi_k,\psi_k) &\in L^\infty(0,T;\mathcal{H}^3), \\
        (\delt\phi_k,\delt\psi_k) &\in L^\infty(0,T;(\mathcal{H}^1_{L,\beta})^\prime)\cap H^1(0,T;\mathcal{H}^1), \\
        (\mu_k,\theta_k) &\in C([0,T];\mathcal{H}^1_{L,\beta})\cap L^2(0,T;\mathcal{H}^3)\cap H^1(0,T;(\mathcal{H}^1)^\prime), \\
        (F^\prime(\phi_k)&, G^\prime(\psi_k)), \ (F^{\prime\prime}(\phi_k), G^{\prime\prime}(\psi_k))\in L^\infty(0,T;\mathcal{L}^p),
    \end{align}
\end{subequations}
for any $2 \leq p <\infty$, and it holds
\begin{align}\label{Est:pp:<1}
    \abs{\phi_k} < 1 \quad\text{a.e.~in~}Q_T\quad\text{and} \quad\abs{\psi_k} < 1 \quad\text{a.e.~on~}\Sigma_T.
\end{align}
Furthermore, the energy inequality \eqref{CCH:EnergyIneq} in combination with \eqref{Est:pp:<1} yields that
\begin{align}\label{CCH:Fixed:EnergyEstimate}
    \begin{split}
        &E_{\mathrm{free}}(\phi(t),\psi(t)) + \frac{\min\{1,m_\ast\}}{2}\int_0^t\norm{(\mu,\theta)}_{L,\beta}^2\ds \\
        &\quad\leq E_{\mathrm{free}}(\phi_0,\psi_0) + \frac{1}{2\min\{1,m_\ast\}}\int_0^t\norm{(\bv_\ast,\bw_\ast)}_{\mathcal{L}^2}^2\ds
    \end{split}
\end{align}
for all $t\in[0,T]$.
Now, for any $t\in[0,T]$, we make the ansatz
\begin{align*}
    \big(\bv_k(t),\bw_k(t)\big) = \sum_{j=1}^k a_j^k(t)\big(\tv_j,\tw_j) \qquad\text{a.e.~in~}\Om\times\Ga
\end{align*}
as a solution to the Galerkin approximation of \eqref{WF:VW:DISCR} that reads as
\begin{align}
    \begin{split}
        &\intO\rho(\phi_k)\delt\bv_k\cdot\tv_l\dx + \intG\sigma(\psi_k)\delt\bw_k\cdot\tw_l\dG  \\
        &\qquad + \intO\rho(\phi_k)(\bv_\ast\cdot\Grad)\bv_k\cdot\tv_l\dx + \intG\sigma(\psi_k)(\bw_\ast\cdot\Gradg)\bw_k\cdot\tw_l\dG \\
        &\qquad - \frac{\tilde\rho_2 - \tilde\rho_1}{2}\intO m_\Om(\phi_k)(\Grad\mu_k\cdot\Grad)\bv_k\cdot\tv_l\dx - \frac{\tilde\sigma_2 - \tilde\sigma_1}{2}\intG m_\Ga(\psi_k)(\Gradg\theta_k\cdot\Gradg)\bw_k\cdot\tw_l\dG \label{GalerkinApprox} \\
        &\qquad  + \intO2\nu_\Om(\phi_k)\D\bv_k:\D\tv_l\dx + \intG2\nu_\Ga(\psi_k)\Dg\bw_k:\Dg\tw_l\dG  + \intG \gamma(\phi_k,\psi_k)\bw_k\cdot\tw_l\dG \\
        &\quad = \intO\mu_k\Grad\phi_k\cdot\tv_l\dx + \intG\theta_k\Gradg\psi_k\cdot\tw_l\dG \\
        &\qquad + \frac{\chi(L)}{2}\Big(\beta\frac{\tilde\sigma_2 -\tilde\sigma_1}{2} - \frac{\tilde\rho_2 - \tilde\rho_1}{2}\Big)\intG(\beta\theta_k - \mu_k)\bw_k\cdot\tw_l\dG,
    \end{split}
\end{align}
a.e. on $[0,T]$ for all $l=1,\ldots,k$, supplemented with the initial condition 
\begin{align}\label{Galerkin:InitialCondition}
    \scp{\bv_k}{\bw_k}\vert_{t=0} = \mathbfcal{P}_{\mathbfcal{V}_k}(\bv_0,\bw_0) \qquad\text{a.e. in~}\Om\times\Ga.
\end{align}
Setting $\A_k(t) = (a_1^k(t),\ldots,a_k^k(t))^\top$, the system \eqref{GalerkinApprox} is equivalent to the system of differential equations
\begin{align}
    \label{ODE:Galerkin}
    \M^k(t)\ddt \A^k(t) + \L^k(t)\A^k(t) = \G^k(t),
\end{align}
where
\begin{align*}
    (\M^k(t))_{l,j} &\coloneqq \intO \rho(\phi_k)\tv_j\cdot\tv_l \dx + \intG \sigma(\psi_k)\tw_j\cdot\tw_l\dG, \\
    (\L^k(t))_{l,j} &\coloneqq \intO \rho(\phi_k)(\bv_\ast\cdot\Grad)\tv_j\cdot\tv_l\dx + \intG \sigma(\psi_k)(\bw_\ast\cdot\Gradg)\tw_j\cdot\tw_l\dG \\
    &\quad - \frac{\tilde\rho_2 - \tilde\rho_1}{2}\intO m_\Om(\phi_k)(\Grad\mu_k\cdot\Grad)\tv_j\cdot\tv_l\dx \\
    &\quad - \frac{\tilde\sigma_2 - \tilde\sigma_1}{2}\intG m_\Ga(\psi_k)(\Gradg\theta_k\cdot\Gradg)\tw_j\cdot\tw_l\dG + \intO 2\nu_\Om(\phi_k)\D\tv_j:\D\tv_l\dx \\
    &\quad + \intG 2\nu_\Ga(\psi_k)\Dg\tw_j:\Dg\tw_l\dG + \intG \gamma(\phi_k,\psi_k)\tw_j\cdot\tw_l\dG \\
    &\quad - \frac{\chi(L)}{2}\Big(\beta\frac{\tilde\sigma_2 - \tilde\sigma_1}{2} - \frac{\tilde\rho_2 - \tilde\rho_1}{2}\Big)\intG (\beta\theta_k - \mu_k)\tw_j\cdot\tw_l\dG, \\
    (\G^k(t))_l &\coloneqq \intO \mu_k\Grad\phi_k\cdot\tv_l\dx + \intG \theta_k\Gradg\psi_k\cdot\tw_l\dG,
\end{align*}
and the initial condition reads as
\begin{align}\label{IC:Galerkin}
    \A^k(0)_j = \big(\mathbfcal{P}_{\mathbfcal{V}_k}(\bv_0,\bw_0),(\tv_j,\tw_j)\big)_{\mathbfcal{L}^2}, \qquad j=1,\ldots,k.
\end{align}
By \eqref{Reg:ConvCH} we have $(\phi_k,\psi_k)\in C([0,T];\mathcal{H}^2)$, which in turn implies that
\begin{align*}
    \rho(\phi_k), \nu_\Om(\phi_k), m_\Om(\phi_k)\in C(\overline\Om\times[0,T])
\end{align*}
as well as 
\begin{align*}
    \sigma(\psi_k), \nu_\Ga(\psi_k), m_\Ga(\psi_k), \gamma(\phi_k,\psi_k)\in C(\Ga\times[0,T]).
\end{align*}
Moreover, as we know that $(\bv_\ast,\bw_\ast)\in C([0,T];\mathbfcal{L}^2_\Div)$ and $(\mu_k,\theta_k)\in C([0,T];\mathcal{H}^1_{L,\beta})$, it follows that $\mathbf{M}^k,\mathbf{L}^k\in C([0,T];\R^{k\times k})$ and $\mathbf{G}^k\in C([0,T];\R^k)$. Furthermore, the matrix $\mathbf{M}^k(\cdot)$ is positive definite on $[0,T]$, which entails that $\det\big(\M^k(\cdot)\big)$ is continuous and uniformly positive on $[0,T]$. Consequently, according to Cramer's rule for the inverse, we also have $(\M^k)^{-1}\in C([0,T];\R^{k\times k})$. Therefore, the classical existence and uniqueness theorem for systems of linear ODEs entails that there exists a unique vector $\mathbf{A}^k\in C^1([0,T];\mathbb{R}^k)$ solving \eqref{ODE:Galerkin} together with the initial condition \eqref{IC:Galerkin}. This shows that there exists a unique solution $(\bv_k,\bw_k)\in C^1([0,T];\mathbfcal{V}_k)$ to \eqref{GalerkinApprox}-\eqref{Galerkin:InitialCondition}.

Now, we multiply \eqref{GalerkinApprox} by $a_l^k$ and sum over $l=1,\ldots,m$ to find
\begin{align*}
    &\frac12\intO \rho(\phi_k)\delt\abs{\bv_k}^2\dx + \frac12\intG \sigma(\psi_k)\delt\abs{\bw_k}^2\dG \\
    &\qquad + \frac12\intO \rho(\phi_k)\bv_\ast\cdot\Grad\abs{\bv_k}^2\dx + \frac12\intG \sigma(\psi_k)\bw_\ast\cdot\Gradg\abs{\bw_k}^2\dG \\ 
    &\qquad + \intO 2\nu_\Om(\phi_k)\abs{\D\bv_k}^2\dx + \intG 2\nu_\Ga(\psi_k)\abs{\Dg\bw_k}^2\dG + \intG \gamma(\phi_k,\psi_k)\abs{\bw_k}^2\dG \\
    &\qquad - \frac12\frac{\tilde\rho_2 - \tilde\rho_1}{2}\intO m_\Om(\phi_k)\Grad\mu_k\cdot\Grad\abs{\bv_k}^2\dx - \frac12\frac{\tilde\sigma_2 - \tilde\sigma_1}{2}\intG m_\Ga(\psi_k)\Gradg\theta_k\cdot\Gradg\abs{\bw_k}^2\dG \\
    &\quad = \intO \mu_k\Grad\phi_k\cdot\bv_k\dx + \intG \theta_k\Gradg\psi_k\cdot\bw_k\dG \\
    &\qquad + \frac{\chi(L)}{2}\Big(\beta\frac{\tilde\sigma_2 - \tilde\sigma_1}{2} - \frac{\tilde\rho_2 - \tilde\rho_1}{2}\Big)\intG(\beta\theta_k - \mu_k)\abs{\bw_k}^2\dG. 
\end{align*}
Then, using integration by parts, we observe that
\begin{align*}
    &\ddt\frac12\intO \rho(\phi_k)\abs{\bv_k}^2\dx + \ddt\frac12\intG \sigma(\psi_k)\abs{\bw_k}^2\dG \nonumber \\
    &\qquad - \frac12\intO \Big(\delt\rho(\phi_k) + \Div(\rho(\phi_k)\bv_\ast)\Big)\abs{\bv_k}^2\dx - \frac12\intG \Big(\delt\sigma(\psi_k) + \Divg(\sigma(\psi_k)\bw_\ast)\Big)\abs{\bw_k}^2\dG \nonumber \\
    &\qquad + \intO 2\nu_\Om(\phi_k)\abs{\D\bv_k}^2\dx + \intG 2\nu_\Ga(\psi_k)\abs{\Dg\bw_k}^2\dG + \intG \gamma(\phi_k,\psi_k)\abs{\bw_k}^2\dG \\
    &\qquad + \frac12\frac{\tilde\rho_2 - \tilde\rho_1}{2}\intO
        \Div(m_\Om(\phi_k)\Grad\mu_k)\abs{\bv_k}^2\dx \nonumber \\
    &\qquad + \frac12\frac{\tilde\sigma_2 - \tilde\sigma_1}{2}\intG\Big(\Divg(m_\Ga(\psi_k)\Gradg\theta_k) - \beta m_\Om(\phi_k)\deln\mu_k\Big)\abs{\bw_k}^2\dG \nonumber \\
    &\quad = \intO \mu_k\Grad\phi_k\cdot\bv_k\dx + \intG \theta_k\Gradg\psi_k\cdot\bw_k\dG. \nonumber
\end{align*}
Here, we have additionally employed the assumption $\beta(\tilde\sigma_2 - \tilde\sigma_1) = \tilde\rho_2 - \tilde\rho_1$ in the case $L = 0$ to rewrite the boundary contribution arising from the integration by parts as
\begin{align*}
    \frac{\tilde\rho_2 - \tilde\rho_1}{2}\intG m_\Om(\phi_k)\deln\mu_k\abs{\bv_k}^2\dG = \beta\frac{\tilde\sigma_2 - \tilde\sigma_1}{2}\intG m_\Om(\phi_k)\deln\mu_k\abs{\bw_k}^2\dG.
\end{align*}
Then, as $\bv_\ast$ and $\bw_\ast$ are both divergence-free on their respective domains, we find that
\begin{align*}
        &- \frac12\intO \Big(\delt\rho(\phi_k) + \Div(\rho(\phi_k)\bv_\ast)\Big)\abs{\bv_k}^2\dx - \frac12\intG \Big(\delt\sigma(\psi_k) + \Divg(\sigma(\psi_k)\bw_\ast)\Big)\abs{\bw_k}^2\dG \\
        &\quad= - \frac12\frac{\tilde\rho_2 - \tilde\rho_1}{2}\intO\Big(\delt\phi_k + \bv_\ast\cdot\Grad\phi_k\Big)\abs{\bv_k}^2\dx - \frac12\frac{\tilde\sigma_2 - \tilde\sigma_1}{2}\intG\Big(\delt\psi_k + \bw_\ast\cdot\Gradg\psi_k)\Big)\abs{\bw_k}^2\dG \\ 
        &\quad= - \frac12\frac{\tilde\rho_2 - \tilde\rho_1}{2}\intO\Div(m_\Om(\phi_k)\Grad\mu_k)\abs{\bv_k}^2\dx \\
        &\qquad - \frac12\frac{\tilde\sigma_2 - \tilde\sigma_1}{2}\intG\Big(\Divg(m_\Ga(\psi_k)\Gradg\theta_k) - \beta m_\Om(\phi_k)\deln\mu_k\Big)\abs{\bw_k}^2\dG,
\end{align*}
see also \eqref{Equation:densities}.
Thus, we deduce that
\begin{align}\label{Energy:kin:id:}
    &\ddt\frac12\Big(\intO \rho(\phi_k)\abs{\bv_k}^2\dx + \intG \sigma(\psi_k)\abs{\bw_k}^2\dG\Big) \nonumber \\
    &\qquad + \intO 2\nu_\Om(\phi_k)\abs{\D\bv_k}^2\dx + \intG 2\nu_\Ga(\psi_k)\abs{\Dg\bw_k}^2\dG + \intG \gamma(\phi_k,\psi_k)\abs{\bw_k}^2\dG \\
    &\quad= - \intO \phi_k\Grad\mu_k\cdot\bv_k\dx - \intG \psi_k\Gradg\theta_k\cdot\bw_k\dG. \nonumber
\end{align}
For the right-hand side, we use
\begin{align*}
    &\Bigabs{- \intO \phi_k\Grad\mu_k\cdot\bv_k\dx - \intG \psi_k\Gradg\theta_k\cdot\bw_k\dG} \\
    &\quad\leq \norm{\Grad\mu_k}_{\mathbf{L}^2(\Om)}\norm{\bv_k}_{\mathbf{L}^2(\Om)} + \norm{\Gradg\theta_k}_{\mathbf{L}^2(\Ga)}\norm{\bw_k}_{\mathbf{L}^2(\Ga)} \\
    &\quad\leq \frac{\min\{2\nu_\ast,\gamma_\ast\}}{2}\norm{(\bv_k,\bw_k)}_{\mathbfcal{H}^1_{0}}^2 + \frac{1}{2\min\{2\nu_\ast,\gamma_\ast\}}\norm{(\mu_k,\theta_k)}_{L,\beta}^2,
\end{align*}
which leads to the differential inequality
\begin{align*}
    &\ddt\frac12\Big(\intO \rho(\phi_k)\abs{\bv_k}^2\dx + \intG \sigma(\psi_k)\abs{\bw_k}^2\dG\Big) + \frac{\min\{2\nu_\ast,\gamma_\ast\}}{2}\norm{(\bv_k,\bw_k)}_{\mathbfcal{H}^1_{0}}^2 \\
    &\quad\leq \frac{\min\{2\nu_\ast,\gamma_\ast\}}{2}\norm{(\bv_k,\bw_k)}_{\mathbfcal{H}^1_{0}}^2 + \frac{1}{2\min\{2\nu_\ast,\gamma_\ast\}}\norm{(\mu_k,\theta_k)}_{L,\beta}^2.
\end{align*}
Integrating the above inequality over $[0,s]$ for some $s\in[0,T]$, and using \eqref{CCH:Fixed:EnergyEstimate}, it follows that
\begin{align*}
    &\intO \frac{\rho_\ast}{2}\abs{\bv_k(s)}^2\dx + \intG \frac{\sigma_\ast}{2}\abs{\bw_k(s)}^2\dG \\
    &\quad\leq \intO \frac{\rho^\ast}{2}\abs{\mathbf{P}_{\mathbfcal{V}_k}^\Om(\bv_0)}^2\dx + \intG \frac{\sigma^\ast}{2}\abs{\mathbf{P}_{\mathbfcal{V}_k}^\Ga(\bw_0)}^2\dG + \frac{1}{\min\{2\nu_\ast,\gamma_\ast\}\min\{1,m_\ast\}}E_{\mathrm{free}}(\phi_0,\psi_0) \\
    &\qquad + \frac{1}{2\min\{2\nu_\ast,\gamma_\ast\}\min\{1,m_\ast\}}\int_0^s \norm{(\bv_\ast(\tau),\bw_\ast(\tau))}_{\mathbfcal{L}^2}^2\dtau,
\end{align*}
which, in turn, entails that
\begin{align}\label{Est:vw:m:int:-1}
    \norm{(\bv_k(s),\bw_k(s)}_{\mathbfcal{L}^2}^2 &\leq c_\ast\big(\rho^\ast + \sigma^\ast)\norm{(\bv_0,\bw_0)}_{\mathbfcal{L}^2}^2 + \frac{2c_\ast}{\min\{2\nu_\ast,\gamma_\ast\}\min\{1,m_\ast\}}E_{\mathrm{free}}(\phi_0,\psi_0) \\
    &\qquad + \frac{c_\ast}{\min\{2\nu_\ast,\gamma_\ast\}\min\{1,m_\ast\}}\int_0^s \norm{(\bv_\ast(\tau),\bw_\ast(\tau))}_{\mathbfcal{L}^2}^2\dtau, \nonumber
\end{align}
where $c_\ast \coloneqq \tfrac{1}{\rho_\ast} + \tfrac{1}{\sigma_\ast}$. At this point, setting
\begin{align*}
    R_1 &= \rho^\ast c_\ast\big(\rho^\ast + \sigma^\ast)\norm{(\bv_0,\bw_0)}_{\mathbfcal{L}^2}^2 + \frac{2c_\ast}{\min\{2\nu_\ast,\gamma_\ast\}\min\{1,m_\ast\}}E_{\mathrm{free}}(\phi_0,\psi_0), \\
    R_2 &= \frac{c_\ast}{\min\{2\nu_\ast,\gamma_\ast\}\min\{1,m_\ast\}}, \qquad R_3 = R_1T,
\end{align*}
and assuming
\begin{align}\label{Assump:vw:_star}
    \int_0^t \norm{(\bv_\ast(\tau),\bw_\ast(\tau))}_{\mathbfcal{L}^2}^2\dtau \leq R_3 \e^{R_2t}\qquad\text{for all~}t\in[0,T],
\end{align}
we deduce from \eqref{Est:vw:m:int:-1} that
\begin{align}\label{Est:vw:m:int}
    \begin{split}
        \int_0^t \norm{(\bv_k(s),\bw_k(s))}_{\mathbfcal{L}^2}^2\ds \leq R_3 + R_2\int_0^t\int_0^s\norm{(\bv_\ast(\tau),\bw_\ast(\tau))}_{\mathbfcal{L}^2}^2\dtau\ds \leq R_3\e^{R_2t}
    \end{split}
\end{align}
for all $t\in[0,T]$. Furthermore, we also infer that
\begin{align}\label{Est:vw:m:sup}
    \sup_{t\in[0,T]} \norm{(\bv_k(t),\bw_k(t))}_{\mathbfcal{L}^2} \leq \big(R_1 + R_3R_2\e^{R_2T}\big)^{\frac12} \eqqcolon M_0.
\end{align}
Next, we aim to control the time derivative $(\delt\bv_k,\delt\bw_k)$. To this end, we multiply \eqref{GalerkinApprox} by $\ddt a_l^k$ and sum over $l=1,\ldots,k$. This yields
\begin{align*}
    &\intO \rho(\phi_k)\abs{\delt\bv_k}^2\dx + \intG \sigma(\psi_k)\abs{\delt\bw_k}^2\dG \nonumber \\
    &\quad = - \intO \rho(\phi_k)(\bv_\ast\cdot\Grad)\bv_k\cdot\delt\bv_k\dx - \intG \sigma(\psi_k)(\bw_\ast\cdot\Gradg)\bw_k\cdot\delt\bw_k\dG \nonumber \\
    &\qquad - \intO 2\nu_\Om(\phi_k)\D\bv_k:\D\delt\bv_k\dx - \intG 2\nu_\Ga(\psi_k)\Dg\bw_k:\Dg\delt\bw_k\dG \nonumber \\
    &\qquad - \intG \gamma(\phi_k,\psi_k)\bw_k\cdot\delt\bw_k\dG \\
    &\qquad + \frac{\tilde\rho_2 - \tilde\rho_1}{2}\intO m_\Om(\phi_k)(\Grad\mu_k\cdot\Grad)\bv_k\cdot\delt\bv_k\dx + \frac{\tilde\sigma_2 - \tilde\sigma_1}{2}\intG m_\Ga(\psi_k)(\Gradg\theta_k\cdot\Gradg)\bw_k\cdot\delt\bw_k\dG \nonumber \\
    &\qquad - \intO \phi_k\Grad\mu_k\cdot\delt\bv_k\dx - \intG \psi_k\Gradg\theta_k\cdot\delt\bw_k\dG \nonumber \\
    &\qquad + \frac{\chi(L)}{2}\Big(\beta\frac{\tilde\sigma_2 - \tilde\sigma_1}{2} - \frac{\tilde\rho_2 - \tilde\rho_1}{2}\Big)\intG (\beta\theta_k - \mu_k)\bw_k\cdot\delt\bw_k\dG. \nonumber
\end{align*}
By exploiting \eqref{Est:InverseSobolev}, we use Hölder's inequality and the Sobolev embedding to find that
\begin{align}\label{Est:Galerkin:Test:ddta}
    &\rho_\ast\norm{\delt\bv_k}_{\mathbf{L}^2(\Om)}^2 + \sigma_\ast\norm{\delt\bw_k}_{\mathbf{L}^2(\Ga)}^2 \nonumber \\[0.4em]
    &\quad\leq\rho^\ast\norm{\bv_\ast}_{\mathbf{L}^2(\Om)}\norm{\Grad\bv_k}_{\mathbf{L}^\infty(\Om)}\norm{\delt\bv_k}_{\mathbf{L}^2(\Om)} + \sigma^\ast\norm{\bw_\ast}_{\mathbf{L}^2(\Ga)}\norm{\Gradg\bw_k}_{\mathbf{L}^\infty(\Ga)}\norm{\delt\bw_k}_{\mathbf{L}^2(\Ga)} \nonumber \\[0.3em]
    &\qquad  + 2\nu^\ast\norm{\D\bv_k}_{\mathbf{L}^2(\Om)}\norm{\D\delt\bv_k}_{\mathbf{L}^2(\Om)} + 2\nu^\ast\norm{\Dg\bw_k}_{\mathbf{L}^2(\Ga)}\norm{\Dg\delt\bw_k}_{\mathbf{L}^2(\Ga)} \nonumber \\[0.3em]
    &\qquad + \gamma^\ast\norm{\bw_k}_{\mathbf{L}^2(\Ga)}\norm{\delt\bw_k}_{\mathbf{L}^2(\Ga)} + \Bigabs{\frac{\tilde\rho_1 - \tilde\rho_2}{2}}m^\ast\norm{\Grad\mu_k}_{\mathbf{L}^2(\Om)}\norm{\Grad\bv_k}_{\mathbf{L}^\infty(\Om)}\norm{\delt\bv_k}_{\mathbf{L}^2(\Om)} \nonumber \\[0.3em]
    &\qquad + \Bigabs{\frac{\tilde\sigma_1 - \tilde\sigma_2}{2}}m^\ast\norm{\Gradg\theta_k}_{\mathbf{L}^2(\Ga)}\norm{\Gradg\bw_k}_{\mathbf{L}^\infty(\Ga)}\norm{\delt\bw_k}_{\mathbf{L}^2(\Ga)} \nonumber \\[0.3em]
    &\qquad + \norm{\Grad\mu_k}_{\mathbf{L}^2(\Om)}\norm{\delt\bv_k}_{\mathbf{L}^2(\Om)} + \norm{\Gradg\theta_k}_{\mathbf{L}^2(\Ga)}\norm{\delt\bw_k}_{\mathbf{L}^2(\Ga)} \nonumber \\[0.2em]
    &\qquad + \tfrac12\chi(L)\Bigabs{\beta\frac{\tilde\sigma_2 - \tilde\sigma_1}{2} - \frac{\tilde\rho_2 - \tilde\rho_1}{2}}\norm{\beta\theta_k - \mu_k}_{L^2(\Ga)}\norm{\bw_k}_{\mathbf{L}^\infty(\Ga)}\norm{\delt\bw_k}_{\mathbf{L}^2(\Ga)} \nonumber \\[0.4em]
    &\quad\leq C(\rho^\ast + \sigma^\ast) \norm{(\bv_\ast,\bw_\ast)}_{\mathbfcal{L}^2}\norm{(\bv_k,\bw_k)}_{\mathbfcal{W}^{2,4}}\norm{(\delt\bv_k,\delt\bw_k)}_{\mathbfcal{L}^2} \nonumber \\[0.3em]
    &\qquad + \nu^\ast C_k^2\norm{(\bv_k,\bw_k)}_{\mathbfcal{L}^2}\norm{(\delt\bv_k,\delt\bw_k)}_{\mathbfcal{L}^2} + \gamma^\ast\norm{\bw_k}_{\mathbf{L}^2(\Ga)}\norm{\delt\bw_k}_{\mathbf{L}^2(\Ga)} \\[0.3em]
    &\qquad + Cm^\ast\bigg(\Bigabs{\frac{\tilde\rho_1 - \tilde\rho_2}{2}} + \Bigabs{\frac{\tilde\sigma_1 - \tilde\sigma_2}{2}}\bigg)\norm{(\mu_k, \theta_k)}_{L,\beta}\norm{(\bv_k,\bw_k)}_{\mathbfcal{W}^{2,4}}\norm{(\delt\bv_k,\delt\bw_k)}_{\mathbfcal{L}^2} \nonumber \\[0.3em]
    &\qquad + C\norm{(\mu_k,\theta_k)}_{L,\beta}\norm{(\delt\bv_k,\delt\bw_k)}_{\mathbfcal{L}^2} \nonumber \\[0.3em]
    &\qquad + \chi(L)C\Bigabs{\beta\frac{\tilde\sigma_2 - \tilde\sigma_1}{2} - \frac{\tilde\rho_2 - \tilde\rho_1}{2}}\norm{\beta\theta - \mu_k}_{L^2(\Ga)}\norm{\bw_k}_{\mathbf{H}^2(\Ga)}\norm{\delt\bw_k}_{\mathbf{L}^2(\Ga)} \nonumber \\[0.4em]
    &\quad\leq C_k(\rho^\ast + \sigma^\ast)\norm{(\bv_\ast,\bw_\ast)}_{\mathbfcal{L}^2}\norm{(\bv_k,\bw_k)}_{\mathbfcal{L}^2}\norm{(\delt\bv_k,\delt\bw_k)}_{\mathbfcal{L}^2} \nonumber \\[0.3em]
    &\qquad + \nu^\ast C_k^2\norm{(\bv_k,\bw_k)}_{\mathbfcal{L}^2}\norm{(\delt\bv_k,\delt\bw_k)}_{\mathbfcal{L}^2} + \gamma^\ast\norm{\bw_k}_{\mathbf{L}^2(\Ga)}\norm{\delt\bw_k}_{\mathbf{L}^2(\Ga)} \nonumber \\[0.3em]
    &\qquad + C_km^\ast\bigg(\Bigabs{\frac{\tilde\rho_1 - \tilde\rho_2}{2}} + \Bigabs{\frac{\tilde\sigma_1 - \tilde\sigma_2}{2}}\bigg)\norm{(\mu_k, \theta_k)}_{L,\beta}\norm{(\bv_k,\bw_k)}_{\mathbfcal{L}^2}\norm{(\delt\bv_k,\delt\bw_k)}_{\mathbfcal{L}^2} \nonumber \\[0.3em]
    &\qquad + C_k\norm{(\mu_k,\theta_k)}_{L,\beta}\norm{(\delt\bv_k,\delt\bw_k)}_{\mathbfcal{L}^2} \nonumber \\[0.3em]
    &\qquad + \chi(L)C_k\Bigabs{\beta\frac{\tilde\sigma_1 - \tilde\sigma_2}{2} - \frac{\tilde\rho_1 - \tilde\rho_2}{2}}\norm{\beta\theta_k - \mu_k}_{L^2(\Ga)}\norm{(\bv_k,\bw_k)}_{\mathbfcal{L}^2}\norm{\delt\bw_k}_{\mathbf{L}^2(\Ga)}. \nonumber
\end{align}
Eventually, in view of \eqref{CCH:Fixed:EnergyEstimate} and \eqref{Assump:vw:_star}-\eqref{Est:vw:m:sup}, we infer that
\begin{align}\label{Est:delt:vw:m:int}
    &\int_0^T \norm{(\delt\bv_k(\tau),\delt\bw_k(\tau))}_{\mathbfcal{L}^2}^2\dtau \nonumber\\
    &\quad\leq 3\Big(C_k(\rho^\ast + \sigma^\ast)\Big(\frac{1}{\rho_\ast} + \frac{1}{\sigma_\ast}\Big)M_0\Big)^2\int_0^T\norm{(\bv_\ast,\bw_\ast)}_{\mathbfcal{L}^2}^2\dtau  \nonumber\\
    &\qquad + 3\Big(\nu^\ast C_k^2\Big(\frac{1}{\rho_\ast} + \frac{1}{\sigma_\ast}\Big)\Big)^2R_4\e^{R_2T} + 3\Big(\gamma^\ast\Big(\frac{1}{\rho_\ast} + \frac{1}{\sigma_\ast}\Big)\Big)^2R_4\e^{R_2T} \nonumber \\
    &\qquad + 3\Big(C_km^\ast\bigg(\Bigabs{\frac{\tilde\rho_1 - \tilde\rho_2}{2}} + \Bigabs{\frac{\tilde\sigma_1 - \tilde\sigma_2}{2}}\bigg)\Big(\frac{1}{\rho_\ast} + \frac{1}{\sigma_\ast}\Big)M_0 + M_0\Big)^2\int_0^T \norm{(\mu_k(\tau),\theta_k(\tau))}_{L,\beta}^2\dtau \nonumber \\
    &\qquad + \Big(C_k\Bigabs{\beta\frac{\tilde\sigma_1 - \tilde\sigma_2}{2} - \frac{\tilde\rho_1 - \tilde\rho_2}{2}}\Big(\frac{1}{\rho_\ast} + \frac{1}{\sigma_\ast}\Big)M_0\Big)^2\int_0^T\norm{\beta\theta_k(\tau) - \mu_k(\tau)}_{L^2(\Ga)}^2\dtau \\
    &\quad\leq \Bigg(3\Big(C_k(\rho^\ast + \sigma^\ast)\Big(\frac{1}{\rho_\ast} + \frac{1}{\sigma_\ast}\Big)M_0\Big)^2 +  3\Big(\nu^\ast C_k^2\Big(\frac{1}{\rho_\ast} + \frac{1}{\sigma_\ast}\Big)\Big)^2 + 3\Big(\gamma^\ast\Big(\frac{1}{\rho_\ast} + \frac{1}{\sigma_\ast}\Big)\Big)\Bigg)R_4\e^{R_2T} \nonumber \\
    &\qquad + 3\Big(C_km^\ast\bigg(\Bigabs{\frac{\tilde\rho_1 - \tilde\rho_2}{2}} + \abs{\frac{\tilde\sigma_1 - \tilde\sigma_2}{2}}\bigg)\Big(\frac{1}{\rho_\ast} + \frac{1}{\sigma_\ast}\Big)M_0 + M_0\Big)^2\Big(2E_{\mathrm{free}}(\phi_0,\psi_0) + R_4\e^{R_2T}\Big) \nonumber \\
    &\qquad + \Big(C_k\Bigabs{\beta\frac{\tilde\sigma_1 - \tilde\sigma_2}{2} - \frac{\tilde\rho_1 - \tilde\rho_2}{2}}\Big(\frac{1}{\rho_\ast} + \frac{1}{\sigma_\ast}\Big)M_0\Big)^2\Big(2E_{\mathrm{free}}(\phi_0,\psi_0) + R_4\e^{R_2T}\Big) \nonumber \\
    &\quad\eqqcolon M_1^2, \nonumber
\end{align}
where $M_1$ only depends on $\rho_\ast,\rho^\ast,\sigma_\ast,\sigma^\ast,\gamma^\ast, E_{\mathrm{free}}(\phi_0,\psi_0), L,T,\Om, m_\ast, m^\ast$ and $k$.

Now we define the setting of the fixed-point argument. We introduce the set
\begin{align*}
    \mathbfcal{S} = \{(\bv,\bw)\in H^1(0,T;\mathbfcal{V}_k): &\int_0^t\norm{(\bv(\tau),\bw(\tau))}_{\mathbfcal{L}^2}^2\dtau \leq R_4\e^{R_2t} \ \text{for all~}t\in[0,T],\\ &\qquad\norm{(\delt\bv,\delt\bw)}_{L^2(0,T;\mathbfcal{V}_k)}\leq M_1\},
\end{align*}
which is a subset of $L^2(0,T;\mathbfcal{V}_k)$, and define the map
\begin{align*}
    \Lambda:\mathbfcal{S}\rightarrow L^2(0,T;\mathbfcal{V}_k), \qquad \Lambda(\bv_\ast,\bw_\ast) = (\bv_k,\bw_k),
\end{align*}
where $(\bv_k,\bw_k)$ is the solution to the system \eqref{GalerkinApprox}. By \eqref{Est:vw:m:int} and \eqref{Est:delt:vw:m:int} it follows immediately that $\Lambda:\mathbfcal{S}\rightarrow\mathbfcal{S}$. We readily notice that $\mathbfcal{S}$ is convex and closed. Moreover, as a subset of $L^2(0,T;\mathbfcal{V}_k)$, $\mathbfcal{S}$ is compact. We are left to prove the continuity of the map $\Lambda$. This is done by adapting the argument in \cite[Theorem~6.1]{Knopf2025a} to the case with non-degenerate mobilities by using the results established in Appendix~\ref{Section:ConvCH}. To this end, let us consider a sequence $\{(\wv_n,\ww_n)\}_{n\in\N}\subset\mathbfcal{S}$ such that $(\wv_n,\ww_n)\rightarrow (\wv_\ast,\ww_\ast)$ in $L^2(0,T;\mathbfcal{V}_k)$. By arguing as above, there exists a sequence $(\wphi_n,\wpsi_n,\wmu_n,\wtheta_n)$, $n\in\N$, and $(\wphi_\ast,\wpsi_\ast,\wmu_\ast,\wtheta_\ast)$ that solve the convective bulk-surface Cahn--Hilliard equation \eqref{CCH:Fixed}, where $(\bv_\ast,\bw_\ast)$ is replaced by $(\wv_n,\ww_n)$ and $(\wv_\ast,\ww_\ast)$, respectively. Observing that $(\wphi_n(0), \wpsi_n(0)) = (\wphi_\ast(0), \wpsi_\ast(0))$ for all $n\in\N$, an application of \eqref{Est:ContinuousDependenceConvective} shows that
\begin{align}\label{ContDep:n}
    \begin{split}
        &\norm{((\wphi_n - \wphi, \wpsi_n - \wpsi_\ast}_{(\mathcal{H}^1_{L,\beta})^\prime}^2 + \int_0^T\norm{(\wphi_n - \wphi_\ast, \wpsi_n - \wpsi_\ast)}_{K,\alpha}^2\dt \\
        &\quad\leq \exp\Big(\int_0^T P_n\dt\Big)\int_0^T\norm{(\wv_n - \wv_\ast, \ww_n - \ww_\ast)}_{\mathbfcal{L}^2}^2\dt,
    \end{split}
\end{align}
where
\begin{align*}
    P_n(\cdot) = C\big(1 + \norm{(\wv_\ast,\ww_\ast)}_{\mathbfcal{L}^2}^2 + \norm{(\wmu_\ast,\wtheta_\ast)}_{L,\beta}^2 + \norm{(\delt\wphi_n,\delt\wpsi_n)}_{(\mathcal{H}^1_{L,\beta})^\prime}^2 + \norm{(\wphi_n,\wpsi_n)}_{\mathcal{H}^2}^4\big),
\end{align*}
and the constant $C > 0$ is independent of $n\in\N$. To pass to the limit in \eqref{ContDep:n}, we have to show that the exponential term is uniformly bounded in $n\in\N$. To do so, we first notice that, since $\{(\wv_n,\ww_n)\}_{n\in\N}$ belongs to the space $\mathbfcal{S}$, it is especially uniformly bounded in $L^2(0,T;\mathbfcal{L}^2)$. Thus, invoking the estimates \eqref{CCH:Est:Energy:Thm}-\eqref{CCH:Est:H^2:Thm}, we find a constant $C = C(p) > 0$ such that
\begin{align}\label{Est:All:n}
    &\norm{(\wphi_n,\wpsi_n)}_{L^\infty(0,T;\mathcal{H}^1)} + \norm{(\delt\wphi_n,\delt\wpsi_n)}_{L^2(0,T;(\mathcal{H}^1_{L,\beta})^\prime)} + \norm{(\Grad\wmu_n,\Gradg\wtheta_n)}_{L^2(0,T;\mathbfcal{L}^2)} \nonumber \\
    &\quad + \chi(L)^{\frac12}\norm{\beta\wtheta_n - \wmu_n}_{L^2(0,T;L^2(\Ga))} + \norm{(\wphi_n,\wpsi_n)}_{L^2(0,T;\mathcal{W}^{2,p})} \\
    &\quad + \norm{(F^\prime(\wphi_n),G^\prime(\wpsi_n))}_{L^2(0,T;\mathcal{L}^p)} + \norm{(\wphi_n,\wpsi_n)}_{L^4(0,T;\mathcal{H}^2)} \leq C \nonumber
\end{align}
for any $2 \leq p < \infty$. Consequently, it follows that
\begin{align*}
    \exp\Big(\int_0^T P_n(\tau)\dt\Big) \leq C.
\end{align*}
Thus, we infer from \eqref{ContDep:n} with an application of the bulk-surface Poincar\'{e} inequality that
\begin{align}\label{ContDep:Conv}
    \norm{(\wphi_n - \wphi_\ast, \wpsi_n - \wpsi_\ast)}_{L^\infty(0,T;(\mathcal{H}^1_{L,\beta})^\prime)\cap L^2(0,T;\mathcal{H}^1)} \rightarrow 0 
\end{align}
as $n\rightarrow\infty$. 
Then, the proof of \cite[Theorem~3.4]{Knopf2025} yields the estimates
\begin{align}\label{Est:mt:mean:n}
    \abs{\mean{\wmu_n}{\wtheta_n}} \leq C\big(1 + \norm{(\wmu_n,\wtheta_n)}_{L,\beta}\big).
\end{align}
In the original formulation in \cite{Knopf2025}, the right-hand side additionally contains a term involving the normal derivative $\deln\wphi_n$. However, this extra term appears only in the case $K = 0$. Since we restrict our analysis to the case $K\in(0,\infty)$, this contribution does not arise and can therefore be omitted.
Consequently, on account of the bulk-surface Poincar\'{e} inequality, we infer from \eqref{Est:mt:mean:n} with \eqref{Est:All:n} that
\begin{align}\label{Est:mt:n:H^1}
    \norm{(\wmu_n,\wtheta_n)}_{L^2(0,T;\mathcal{H}^1)} \leq C.
\end{align}
Then, in view of \eqref{Est:All:n}, we find that
\begin{align}\label{Conv:FG:n}
    (F^\prime(\wphi_n),G^\prime(\wpsi_n)) \rightarrow (F^\prime(\wphi_\ast),G^\prime(\wpsi_\ast)) \qquad\text{weakly in~} L^2(0,T;\mathcal{L}^2).
\end{align}
Here, we have additionally used the fact that $F_0^\prime$ and $G_0^\prime$ are maximal monotone operators to identify the limit functions (see, e.g., \cite{Knopf2025a} for more details). Next, in view of the weak formulation \eqref{WF:MT:DISCR} and the convergences \eqref{ContDep:Conv} and \eqref{Conv:FG:n}, we obtain 
\begin{align*}
    (\wmu_n,\wtheta_n) \rightarrow (\wmu_\ast,\wtheta_\ast) \qquad\text{weakly in~} L^2(0,T;(\mathcal{H}^1)^\prime).
\end{align*}
As we additionally have the uniform bound \eqref{Est:mt:n:H^1}, it readily follows that
\begin{align*}
    (\wmu_n,\wtheta_n) \rightarrow (\wmu_\ast,\wtheta_\ast) \qquad\text{weakly in~} L^2(0,T;\mathcal{H}^1_{L,\beta}).
\end{align*}

Now, to show the continuity of $\Lambda$, we have to prove that
\begin{align*}
    \Lambda(\wv_n,\wv_n) \rightarrow \Lambda(\wv_\ast,\ww_\ast) \qquad\text{strongly in~} L^2(0,T;\mathbfcal{V}_k)
\end{align*}
as $n\rightarrow\infty$. To this end, using the abbreviation $(\bu^\Om_n,\bu^\Ga_n) = \Lambda(\wv_n,\ww_n)$, we recall that by definition of $\Lambda$, $(\bu^\Om_n,\bu^\Ga_n)$ satisfies the variational formulation
\begin{align}
    &\intO\rho(\wphi_n)\delt\bu^\Om_n\cdot\tv_l\dx + \intG\sigma(\wpsi_n)\delt\bu^\Ga_n\cdot\tw_l\dG \nonumber \\
    &\qquad + \intO\rho(\wphi_n)(\wv_n\cdot\Grad)\bu^\Om_n\cdot\tv_l\dx + \intG\sigma(\wpsi_n)(\ww_n\cdot\Gradg)\bu^\Ga_n\cdot\tw_l\dG \nonumber \\
    &\qquad + \frac{\tilde\rho_2 - \tilde\rho_1}{2}\intO m_\Om(\wphi_n)(\Grad\wmu_n\cdot\Grad)\bu^\Om_n\cdot\tv_l\dx + \frac{\tilde\sigma_2 - \tilde\sigma_1}{2}\intG m_\Ga(\wpsi_n)(\Gradg\wtheta_n\cdot\Gradg)\bu^\Ga_n\cdot\tw_l\dG \label{GalerkinApprox:Continuity} \\
    &\qquad + \intO2\nu_\Om(\wphi_n)\D\bu^\Om_n:\D\tv_l\dx + \intG2\nu_\Ga(\wpsi_n)\Dg\bu^\Ga_n:\Dg\tw_l\dG + \intG \gamma(\wphi_n,\wpsi_n)\bu^\Ga_n\cdot\tw_l\dG  \nonumber \\
    &\quad = \intO\wmu_n\Grad\wphi_n\cdot\tv_l\dx + \intG\wtheta_n\Gradg\wpsi_n\cdot\tw_l\dG \nonumber \\
    &\qquad + \frac{\chi(L)}{2}\Big(\beta\frac{\tilde\sigma_2 - \tilde\sigma_1}{2} - \frac{\tilde\rho_2 - \tilde\rho_1}{2}\big)\intG(\beta\wtheta_n - \wmu_n)\bu^\Ga_n\cdot\tw_l\dG, \nonumber
\end{align}
for all $l=1,\ldots,k$. Now, we notice that, since $\{(\bu^\Om_n,\bu^\Ga_n)\}_{n\in\N}\subset\mathbfcal{S}$, there exists $(\bu^\Om_\ast,\bu^\Ga_\ast)\in H^1(0,T;\mathbfcal{V}_k)$ such that
\begin{alignat}{2}
    (\bu^\Om_n,\bu^\Ga_n) &\rightarrow (\bu^\Om_\ast,\bu^\Ga_\ast) &&\qquad\text{weakly in~} H^1(0,T;\mathbfcal{V}_k). \label{Conv:u_n}
\end{alignat}
In light of the compact embedding $H^1(0,T;\mathbfcal{V}_k)\emb C([0,T];\mathbfcal{V}_k)$ and \eqref{Est:InverseSobolev}, we readily deduce from \eqref{Conv:u_n} and by compactness, along a non-relabeled subsequence, that
\begin{alignat}{2}
        (\bu^\Om_n,\bu^\Ga_n) &\rightarrow (\bu^\Om_\ast,\bu^\Ga_\ast) &&\qquad\text{strongly in~} C([0,T];\mathbfcal{V}_k), \label{Conv:u_n:strong} \\
        (\Grad\bu^\Om_n,\Grad\bu^\Ga_n) &\rightarrow (\Grad\bu^\Om_\ast,\Gradg\bu^\Ga_\ast) &&\qquad\text{strongly in~} C([0,T];\mathbfcal{L}^2). \label{Conv:Grad:u_n:strong}
\end{alignat}
Now, the convergences \eqref{Conv:u_n}-\eqref{Conv:Grad:u_n:strong} allow us to pass to the limit $n\rightarrow\infty$ in \eqref{GalerkinApprox:Continuity}. As $(\bu^\Om_\ast,\bu^\Ga_\ast)$ further satisfies the initial condition $(\bu^\Om_\ast,\bu^\Ga_\ast)\vert_{t=0} = \mathbfcal{P}_{\mathbfcal{V}_k}(\bv_0,\bw_0)$, which follows readily from \eqref{Conv:u_n:strong}, we infer that $(\bu^\Om_\ast,\bu^\Ga_\ast)$ is a solution to \eqref{GalerkinApprox}-\eqref{Galerkin:InitialCondition}. In view of the uniqueness of solutions to \eqref{GalerkinApprox}-\eqref{Galerkin:InitialCondition}, this implies $(\bu^\Om_\ast,\bu^\Ga_\ast) = \Lambda(\wv_\ast,\ww_\ast)$, which entails
\begin{align*}
    \Lambda(\wv_n,\ww_n) = (\bu^\Om_n,\bu^\Ga_n)
    \rightarrow 
    (\bu^\Om_\ast, \bu^\Ga_\ast) = \Lambda(\wv_\ast,\ww_\ast)
    \qquad\text{strongly in~} L^2(0,T;\mathbfcal{V}_k)
\end{align*}
as $n\rightarrow\infty$.
This proves that the map $\Lambda$ is continuous. 
Finally, we are in a position to apply the Schauder fixed-point theorem and conclude that the map $\Lambda$ has a fixed-point in $\mathbfcal{S}$, which gives the existence of the approximate solution $(\bv_k,\bw_k,\phi_k,\psi_k,\mu_k,\theta_k)$ satisfying \eqref{WF:VW:DISCR}-\eqref{Approx:Problem:IC} for any $k\in\N$.

\textbf{A priori Estimates for the Approximate Solutions}
We first observe that, as $T > 0$ was arbitrarily chosen, the approximate solutions $(\bv_k,\bw_k,\phi_k,\psi_k,\mu_k,\theta_k)$ exist on $(0,\infty)$ for any $k\in\N$. In the following, we denote with the letter $C$ generic constants that are independent of $k\in\N$.
We start by establishing a priori estimates, which are based on an energy estimate for the approximate solutions. To this end, we take $(\wv,\ww) = (\bv_k,\bw_k)$ in \eqref{WF:VW:DISCR} and perform integration by parts. This yields
\begin{align}\label{ddt:KineticEnergy}
    &\ddt\frac12\Big(\intO \rho(\phi_k)\abs{\bv_k}^2\dx + \intG \sigma(\psi_k)\abs{\bw_k}^2\dG\Big) \nonumber \\
    &\qquad + \intO 2\nu_\Om(\phi_k)\abs{\D\bv_k}^2\dx + \intG 2\nu_\Ga(\psi_k)\abs{\Dg\bw_k}^2\dG + \intG \gamma(\phi_k,\psi_k)\abs{\bw_k}^2\dG \\
    &\quad = -\intO \phi_k\Grad\mu_k\cdot\bv_k\dx - \intG \psi_k\Gradg\theta_k\cdot\bw_k\dG, \nonumber
\end{align}
see \eqref{Energy:kin:id:}. Then, taking $(\zeta,\xi) = (\mu_k,\theta_k)$ in \eqref{WF:PP:DISCR} and $(\eta,\vartheta) = -(\delt\phi_k,\delt\psi_k)$ in \eqref{WF:MT:DISCR}, we find with integration by parts that
\begin{align}\label{ddt:FreeEnergy}
    &\ddt\Big( \intO \frac12\abs{\Grad\phi_k}^2 + F(\phi_k)\dx + \intG \frac12\abs{\Gradg\psi_k}^2 + G(\psi_k)\dG + \chi(K)\intG \frac12(\alpha\psi_k - \phi_k)^2\dG \Big) \nonumber \\
    &\qquad + \intO m_\Om(\phi_k)\abs{\Grad\mu_k}^2\dx + \intG m_\Ga(\psi_k)\abs{\Gradg\theta_k}^2\dG + \chi(L)\intG (\beta\theta_k - \mu_k)^2\dG \\
    &\quad = \intO \phi_k\Grad\mu_k\cdot\bv_k\dx + \intG \psi_k\Gradg\theta_k\cdot\bw_k. \nonumber  
\end{align}
By summming \eqref{ddt:KineticEnergy} and \eqref{ddt:FreeEnergy}, we obtain
\begin{align*}
    &\ddt E_{\mathrm{tot}}(\bv_k,\bw_k,\phi_k,\psi_k) \\
    &\quad + \intO 2\nu_\Om(\phi_k)\abs{\D\bv_k}^2\dx + \intG 2\nu_\Ga(\psi_k)\abs{\Dg\bw_k}^2\dG + \intG \gamma(\phi_k,\psi_k)\abs{\bw_k}^2\dG \\
    &\quad + \intO m_\Om(\phi_k)\abs{\Grad\mu_k}^2\dx + \intG m_\Ga(\psi_k)\abs{\Gradg\theta_k}^2\dG + \chi(L) \intG (\beta\theta_k - \mu_k)^2\dG = 0.
\end{align*}
Integrating in time from $0$ to $t$ for any $t\in(0,\infty)$, we have
\begin{align*}
    &E_{\mathrm{tot}}(\bv_k(t),\bw_k(t),\phi_k(t),\psi_k(t)) \\
    &\quad + \int_0^t\Bigg(\intO 2\nu_\Om(\phi_k)\abs{\D\bv_k}^2\dx + \intG 2\nu_\Ga(\psi_k)\abs{\Dg\bw_k}^2\dG + \intG \gamma(\phi_k,\psi_k)\abs{\bw_k}^2\dG\Bigg)\ds \\
    &\quad + \int_0^t\Bigg(\intO m_\Om(\phi_k)\abs{\Grad\mu_k}^2\dx + \intG m_\Ga(\psi_k)\abs{\Gradg\theta_k}^2\dGs + \chi(L)\intG (\beta\theta_k - \mu_k)^2\dG\Bigg)\ds \\
    &= E_{\mathrm{tot}}(\mathbfcal{P}_{\mathbfcal{V}_k}(\bv_0,\bw_0),\phi_0,\psi_0).
\end{align*}
As the initial energy can be estimated as
\begin{align*}
    E_{\mathrm{tot}}(\mathbfcal{P}_{\mathbfcal{V}_k}(\bv_0,\bw_0),\phi_0,\psi_0) \leq \frac{\max\{\rho_\ast,\sigma_\ast\}}{2}\norm{(\bv_0,\bw_0)}_{\mathcal{L}^2} + E_{\mathrm{free}}(\phi_0,\psi_0),
\end{align*}
it follows using $\abs{\phi_k} < 1$ a.e. in $\Om\times(0,\infty)$, $\abs{\psi_k} < 1$ a.e. in $\Ga\times(0,\infty)$ and the bulk-surface Korn inequality \eqref{Est:Korn} that
\begin{align}
    \norm{(\bv_k,\bw_k)}_{L^\infty(0,\infty;\mathbfcal{L}^2)} + \norm{(\bv_k,\bw_k)}_{L^2(0,\infty;\mathbfcal{H}^1)}\leq C, \label{Est:vw:Energy}\\
    \norm{(\phi_k,\psi_k)}_{L^\infty(0,\infty;\mathcal{H}^1)} \leq C, \label{Est:pp:Energy} \\
    \norm{(\Grad\mu_k,\Gradg\theta_k)}_{L^2(0,\infty;\mathbfcal{L}^2)} + \chi(L)^{\frac12}\norm{\beta\theta_k - \mu_k}_{L^2(0,\infty;L^2(\Ga))} \leq C \label{Est:mt:Energy}.
\end{align}
Then, recalling \eqref{Est:mt:mean:n}, we employ the bulk-surface Poincar\'{e} inequality to deduce that
\begin{align}\label{Est:mt:H^1:a.e.}
    \norm{(\mu_k,\theta_k)}_{\mathcal{H}^1} \leq C\big(1 + \norm{(\mu_k,\theta_k)}_{L,\beta}\big).
\end{align}
Then, by \cite[Proposition~5.5]{Giorgini2025} and another application of the bulk-surface Poincar\'{e} inequality, we obtain
\begin{align}\label{Est:pp:pot:L^p:a.e.}
    \begin{split}
        \norm{(\phi_k,\psi_k)}_{\mathcal{W}^{2,p}} + \norm{(F^\prime(\phi_k),G^\prime(\psi_k))}_{\mathcal{L}^p} &\leq C_p\big(1 + \norm{(\mu_k,\theta_k)}_{\mathcal{H}^1}\big) \\
        &\leq C_p\big(1 + \norm{(\mu_k,\theta_k)}_{L,\beta}\big)
    \end{split}
\end{align}
for any $2 \leq p < \infty$. By comparison in \eqref{WF:PP:DISCR}, it holds that
\begin{align}\label{Est:pp:delt:a.e.}
    \norm{(\delt\phi_k,\delt\psi_k)}_{(\mathcal{H}^1_{L,\beta})^\prime} \leq C\big(\norm{(\bv_k,\bw_k)}_{\mathbfcal{L}^2} + \norm{(\mu_k,\theta_k)}_{L,\beta}\big).
\end{align}
Here, we have additionally made use of \eqref{Est:pp:<1} and the boundedness of the mobility functions. Lastly, noting on $K\in(0,\infty)$, it follows from \cite[Corollary~5.4]{Giorgini2025} in combination with \eqref{Est:pp:Energy} and \eqref{Est:mt:H^1:a.e.} that
\begin{align*}
    \norm{(-\Lap\phi_k,-\Lapg\psi_k + \alpha\deln\phi_k)}_{\mathcal{L}^2}^2 \leq C\big(1 + \norm{(\mu_k,\theta_k)}_{L,\beta}\big).
\end{align*}
Thus, by elliptic regularity theory for systems with bulk-surface coupling, we obtain 
\begin{align}\label{Est:pp:H^2:4}
    \norm{(\phi_k,\psi_k)}_{\mathcal{H}^2}^4 \leq C\big(1 + \norm{(\mu_k,\theta_k)}_{L,\beta}^2\big).
\end{align}
Consequently, collecting the above estimates, we infer from \eqref{Est:vw:Energy} and \eqref{Est:mt:Energy} that
\begin{align}\label{Est:pp:all:Thm}
    \begin{split}
        &\norm{(\delt\phi_k,\delt\psi_k)}_{L^2(0,\infty;(\mathcal{H}^1_{L,\beta})^\prime)} + \norm{(\phi_k,\psi_k)}_{L^2_{\mathrm{uloc}}(0,\infty;\mathcal{W}^{2,p})} + \norm{(F^\prime(\phi_k),G^\prime(\psi_k))}_{L^2_{\mathrm{uloc}}(0,\infty;\mathcal{L}^p)} \\
        &\quad + \norm{(\mu_k,\theta_k)}_{L^2_{\mathrm{uloc}}(0,\infty;\mathcal{L}^2)} + \norm{(\phi_k,\psi_k)}_{L^4_{\mathrm{uloc}}(0,\infty;\mathcal{H}^2)} \leq C_p
    \end{split}
\end{align}
for any $2 \leq p < \infty$.

\textbf{Higher Order Estimates for the Approximate Solutions}
Next, we establish higher-order estimates for the approximate solution $(\bv_k,\bw_k,\phi_k,\psi_k,\mu_k,\theta_k)$. Based on the estimates \eqref{Est:vw:Energy}-\eqref{Est:mt:Energy} and \eqref{Est:pp:all:Thm}, we can use Theorem~\ref{Theorem:CCH:Strong}, which entails the uniform estimates
\begin{align}\label{Est:K1p}
    \norm{(\phi_k,\psi_k)}_{L^\infty(0,\infty;\mathcal{W}^{2,p})} + \norm{(F^\prime(\phi_k),G^\prime(\psi_k))}_{L^\infty(0,\infty;\mathcal{L}^p)} + \norm{(F^{\prime\prime}(\phi_k),G^{\prime\prime}(\psi_k))}_{L^\infty(0,\infty;\mathcal{L}^p)} \leq K_{1,p}
\end{align}
for any $2 \leq p \leq \infty$, as well as
\begin{align}\label{Est:K2}
    \begin{split}
        &\norm{(\delt\phi_k,\delt\psi_k)}_{L^\infty(0,\infty;(\mathcal{H}^1_{L,\beta})^\prime)} + \norm{(\delt\phi_k,\delt\psi_k)}_{L^2_{\mathrm{uloc}}(0,\infty;\mathcal{H}^1)} \\
        &\qquad + \norm{(\mu_k,\theta_k)}_{L^\infty(0,\infty;\mathcal{H}^1)} + \norm{(\mu_k,\theta_k)}_{L^2_{\mathrm{uloc}}(0,\infty;\mathcal{H}^3)} \leq K_2.
    \end{split}
\end{align}
Moreover, as a consequence of \eqref{Est:K1p}, we have
\begin{align}\label{Est:K3}
    \norm{(\phi_k,\psi_k)}_{L^\infty(0,\infty;\mathcal{H}^3)} + \norm{(F^{\prime\prime\prime}(\phi_k),G^{\prime\prime\prime}(\psi_k))}_{L^\infty(0,\infty;\mathcal{L}^p)} \leq  K_{3,p}
\end{align}
for any $2 \leq p < \infty$. Next, we aim to establish higher-order estimates of the approximate velocity fields $(\bv_k,\bw_k)$. To this end, we take $(\wv,\ww) = (\delt\bv_k,\delt\bw_k)$ as a test function in \eqref{WF:VW:DISCR} and find
\begin{align}\label{Est:H_k:-1}
    &\ddt\frac12\Big(\intO 2\nu_\Om(\phi_k)\abs{\D\bv_k}^2\dx + \intG 2\nu_\Ga(\psi_k)\abs{\Dg\bw_k}^2\dG + \intG \gamma(\phi_k,\psi_k)\abs{\bw_k}^2\dG\Big) \nonumber \\
    &\qquad + \intO \rho(\phi_k)\abs{\delt\bv_k}^2\dx + \intG \sigma(\psi_k)\abs{\delt\bw_k}^2\dG \nonumber \\
    &\quad = - \intO \rho(\phi_k)(\bv_k\cdot\Grad)\bv_k\cdot\delt\bv_k\dx - \intG \sigma(\psi_k)(\bw_k\cdot\Gradg)\bw_k\cdot\delt\bw_k\dG \nonumber \\
    &\qquad + \intO 2\nu_\Om^\prime(\phi_k)\delt\phi_k\abs{\D\bv_k}^2\dx + \intG 2\nu_\Ga^\prime(\psi_k)\delt\psi_k\abs{\Dg\bw_k}^2\dG \nonumber \\
    &\qquad + \intG \big(\del_1\gamma(\phi_k,\psi_k)\delt\phi_k + \del_2\gamma(\phi_k,\psi_k)\delt\psi_k\big)\abs{\bw_k}^2\dG \\
    &\qquad + \frac{\tilde\rho_2 - \tilde\rho_1}{2}\intO m_\Om(\phi_k)(\Grad\mu_k\cdot\Grad)\bv_k\cdot\delt\bv_k\dx + \frac{\tilde\sigma_2 - \tilde\sigma_1}{2}\intG m_\Ga(\psi_k)(\Gradg\theta_k\cdot\Gradg)\bw_k\cdot\delt\bw_k\dG \nonumber \\
    &\qquad + \frac{\chi(L)}{2}\Big(\beta\frac{\tilde\sigma_2 - \tilde\sigma_1}{2} - \frac{\tilde\rho_2 - \tilde\rho_1}{2}\Big)\intG (\beta\theta_k - \mu_k)\bw_k\cdot\delt\bw_k\dG \nonumber \\
    &\qquad - \intO \phi_k\Grad\mu_k\cdot\delt\bv_k\dx - \intG \psi_k\Gradg\theta_k\cdot\delt\bw_k\dG \nonumber \\
    &\quad \eqqcolon \sum_{j=1}^{10}R_j. \nonumber
\end{align}
We now intend to find suitable bounds for the terms $R_j$, $j=1,\ldots,10$, by utilizing the estimates \eqref{Est:K1p}-\eqref{Est:K3}. Let $\varpi$ be a positive constant whose precise value will be determined later. Exploiting \eqref{InterpolEst:L^4}, \eqref{Est:Korn} and \eqref{NormEquivalence:Stokes}, we find
\begin{align}\label{Est:R1R2}
    \begin{split}
        \abs{R_1 + R_2} &\leq \max\{\rho^\ast,\sigma^\ast\}\norm{(\bv_k,\bw_k)}_{\mathbfcal{L}^4}\norm{(\Grad\bv_k,\Gradg\bw_k)}_{\mathbfcal{L}^4}\norm{(\delt\bv_k,\delt\bw_k)}_{\mathbfcal{L}^2} \\
        &\leq \frac{\min\{\rho_\ast,\sigma_\ast\}}{8}\norm{(\delt\bv_k,\delt\bw_k)}_{\mathbfcal{L}^2}^2 + C\norm{(\bv_k,\bw_k)}_{\mathbfcal{H}^1_{0}}^2\norm{(\bv_k,\bw_k)}_{\mathbfcal{H}^2} \\
        &\leq \frac{\min\{\rho_\ast,\sigma_\ast\}}{8}\norm{(\delt\bv_k,\delt\bw_k)}_{\mathbfcal{L}^2}^2 + \frac{\varpi}{2}\norm{(\bv_k,\bw_k)}_{\mathbfcal{H}^2}^2 + C\norm{(\bv_k,\bw_k)}_{\mathbfcal{H}^1_{0}}^4.
    \end{split}
\end{align}
Then, in view of the assumptions on the coefficients, we have by \eqref{Est:Poincare}, \eqref{InterpolEst:L^4}, \eqref{Est:pp:delt:a.e.} and \eqref{NormEquivalence:Stokes}
\begin{align}\label{Est:R3R4R5}
    \begin{split}
        \abs{R_3 + R_4 + R_5} &\leq C\norm{\delt\phi_k}_{L^2(\Om)}\norm{\D\bv_k}_{\mathbf{L}^4(\Om)}^2 + C\norm{\delt\psi_k}_{L^2(\Ga)}\norm{\Dg\bw_k}_{\mathbf{L}^4(\Ga)}^2 \\
        &\quad + C\big(\norm{\delt\phi_k}_{L^2(\Ga)} + \norm{\delt\psi_k}_{L^2(\Ga)}\big)\norm{\bw_k}_{\mathbf{L}^4(\Ga)}^2 \\
        &\leq C\norm{(\delt\phi_k,\delt\psi_k)}_{\mathcal{L}^2}\norm{(\bv_k,\bw_k)}_{\mathbfcal{H}^1_{0}}\norm{(\bv_k,\bw_k)}_{\mathbfcal{H}^2} \\
        &\leq\frac{\varpi}{2}\norm{(\bv_k,\bw_k)}_{\mathbfcal{H}^2}^2 + C\norm{(\delt\phi_k,\delt\psi_k)}_{K,\alpha}^2\norm{(\bv_k,\bw_k)}_{\mathbfcal{H}^1_{0}}^2.
    \end{split}
\end{align}
Next, using interpolation of Sobolev spaces (see, e.g., \cite[Lemma~2.2]{Colli2023}), \eqref{Est:mt:H^1:a.e.} and \eqref{NormEquivalence:Stokes}, we get
\begin{align}\label{Est:R6R7}
    \begin{split}
        \abs{R_6 + R_7} &\leq C\norm{(\Grad\mu_k,\Gradg\theta_k)}_{\mathbfcal{L}^\infty}\norm{(\Grad\bv_k,\Gradg\bw_k)}_{\mathbfcal{L}^2}\norm{(\delt\bv_k,\delt\bw_k)}_{\mathbfcal{L}^2} \\
        &\leq \frac{\min\{\rho_\ast,\sigma_\ast\}}{8}\norm{(\delt\bv_k,\delt\bw_k)}_{\mathbfcal{L}^2}^2 + C\norm{(\Grad\mu_k,\Gradg\theta_k)}_{\mathbfcal{H}^2}^2\norm{(\bv_k,\bw_k)}_{\mathbfcal{H}^1_{0}}^2.
    \end{split}
\end{align}
By the trace theorem in combination with \eqref{Est:mt:H^1:a.e.} we find that
\begin{align}\label{Est:R8}
    \begin{split}
        \abs{R_8} &\leq C\norm{\deln\mu_k}_{L^4(\Ga)}\norm{\bw_k}_{\mathbf{L}^4(\Ga)}\norm{\delt\bw_k}_{\mathbf{L}^2(\Ga)} \\
        &\quad\leq \frac{\min\{\rho_\ast,\sigma_\ast\}}{8}\norm{(\delt\bv_k,\delt\bw_k)}_{\mathbfcal{L}^2}^2 + C\norm{(\Grad\mu_k,\Gradg\theta_k)}_{\mathbfcal{H}^2}^2\norm{(\bv_k,\bw_k)}_{\mathbfcal{H}^1_{0}}^2.
    \end{split}
\end{align}
Lastly, in view of \eqref{Est:pp:<1}, we get
\begin{align}\label{Est:R9R10}
    \begin{split}
        \abs{R_9 + R_{10}} &\leq \norm{(\Grad\mu_k,\Gradg\theta_k)}_{\mathbfcal{L}^2}\norm{(\delt\bv_k,\delt\bw_k)}_{\mathbfcal{L}^2} \\
        &\leq \frac{\min\{\rho_\ast,\sigma_\ast\}}{8}\norm{(\delt\bv_k,\delt\bw_k)}_{\mathbfcal{L}^2}^2 + C\norm{(\Grad\mu_k,\Gradg\theta_k)}_{\mathbfcal{L}^2}^2.
    \end{split}
\end{align}
Combining the above estimates \eqref{Est:R1R2}-\eqref{Est:R9R10}, we arrive at
\begin{align}\label{Est:H_k:1}
    \begin{split}
        &\ddt\frac12\Big(\intO 2\nu_\Om(\phi_k)\abs{\D\bv_k}^2\dx + \intG 2\nu_\Ga(\psi_k)\abs{\Dg\bw_k}^2\dG + \intG \gamma(\phi_k,\psi_k)\abs{\bw_k}^2\dG\Big) \\
        &\qquad + \frac{\min\{\rho_\ast,\sigma_\ast\}}{2}\norm{(\delt\bv_k,\delt\bw_k)}_{\mathbfcal{L}^2}^2\\
        &\quad\leq \varpi\norm{(\bv_k,\bw_k)}_{\mathbfcal{H}^2}^2 + C\norm{(\Grad\mu_k,\Gradg\theta_k)}_{\mathbfcal{L}^2}^2 \\
        &\qquad + C\big(\norm{(\delt\phi_k,\delt\psi_k)}_{K,\alpha}^2 + \norm{(\Grad\mu_k,\Gradg\theta_k)}_{\mathbfcal{H}^2}^2 + \norm{(\bv_k,\bw_k)}_{\mathbfcal{H}^1_{0}}^2\big)\norm{(\bv_k,\bw_k)}_{\mathbfcal{H}^1_{0}}^2,
    \end{split}
\end{align}
where $C > 0$ is a constant that depends on $\varpi$, but is independent of $k$. We are left to control the term $\norm{(\bv_k,\bw_k)}_{\mathbfcal{H}^2}$. This is done by employing regularity theory for the bulk-surface Stokes system as developed in Appendix~\ref{Section:Stokes}. To this end, we first notice that, for any $k\in\N$ and a.e. $t\in(0,\infty)$, there exists a pressure pair $(p_k(t),q_k(t))\in \mathcal{H}^1\cap\mathcal{L}^2_{(0)}$ such that
\begin{subequations}\label{eqs:NSCH:NC:2}
    \begin{alignat}{2}
        &\rho(\phi_k)\delt\bv_k + \rho(\phi_k)(\bv_k\cdot\Grad)\bv_k + (\J_k\cdot\Grad)\bv_k - \Div(2\nu_\Om(\phi_k)\D\bv_k) + \Grad p_k = \mu_k\Grad\phi_k &&\qquad\text{a.e.~in~}Q \label{ConsForm:v:2}\\
        &\sigma(\psi_k)\delt\bw_k + \sigma(\psi_k)(\bw_k\cdot\Gradg)\bw_k + (\K_k\cdot\Gradg)\bw_k - \Divg(2\nu_\Ga(\psi_k)\Dg\bw_k) + \Gradg q_k\nonumber \\
        &\qquad = \theta_k\Gradg\psi_k - 2\nu_\Om(\phi_k)\big[ \D\bv_k\,\n \big]_\tau + \tfrac12(\J_k\cdot\n)\bw - \tfrac12(\J_{\Ga,k}\cdot\n)\bw_k - \gamma(\phi_k,\psi_k)\bw_k &&\qquad\text{a.e.~on~}\Sigma. \label{ConsForm:k}
    \end{alignat}
\end{subequations}
Here, we have used the shorthand notation
\begin{align*}
    \J_k = -\frac{\tilde\rho_2 - \tilde\rho_1}{2}m_\Om(\phi_k)\Grad\mu_k, \qquad \K_k = -\frac{\tilde\sigma_2 - \tilde\sigma_1}{2}m_\Ga(\psi_k)\Gradg\theta_k, \qquad \J_{\Ga,k} = - \beta\frac{\tilde\sigma_2 - \tilde\sigma_1}{2}m_\Om(\phi_k)\Grad\mu_k.
\end{align*}
Consequently, in view of Theorem~\ref{App:Theorem:BulkSurfaceStokes:SS}, we find that
\begin{align*}
    &\norm{(\bv_k,\bw_k)}_{\mathbfcal{H}^2}^2 + \norm{(p_k,q_k)}_{\mathcal{H}^1}^2 \\
    &\quad \leq C\big(\norm{-\rho(\phi_k)\delt\bv_k + \rho(\phi_k)(\bv_k\cdot\Grad)\bv_k - (\J_k\cdot\Grad)\bv_k + \mu_k\Grad\phi_k}_{\mathbf{L}^2(\Om)}^2 \\
    &\qquad + \norm{-\sigma(\psi_k)\delt\bw_k - \sigma(\psi_k)(\bw_k\cdot\Gradg)\bw_k - (\K_k\cdot\Gradg)\bw_k + \theta_k\Gradg\psi_k + \tfrac12(\J_k\cdot\n)\bw_k - \tfrac12(\J_{\Ga,k}\cdot\n)\bw_k)}_{\mathbf{L}^2(\Ga)}^2\big) \\
    &\quad\leq C\big(\norm{(\delt\bv_k,\delt\bw_k)}_{\mathbfcal{L}^2}^2 + \norm{(\bv_k,\bw_k)}_{\mathbfcal{L}^4}^2\norm{(\Grad\bv_k,\Gradg\bw_k)}_{\mathbfcal{L}^4}^2 + \norm{(\Grad\mu_k,\Gradg\theta_k)}_{\mathbfcal{L}^\infty}^2\norm{(\Grad\bv_k,\Gradg\bw_k)}_{\mathbfcal{L}^2}^2 \\
    &\qquad + \norm{(\Grad\mu_k,\Gradg\theta_k)}_{\mathbfcal{L}^2}^2 + \norm{\deln\mu_k}_{L^4(\Ga)}^2\norm{\bw_k}_{\mathbfcal{L}^4}^2\big) \\
    &\quad\leq C\big(\norm{(\delt\bv_k,\delt\bw_k)}_{\mathbfcal{L}^2}^2 + C\norm{(\bv_k,\bw_k)}_{\mathbfcal{H}^1_{0}}^2\norm{(\bv_k,\bw_k)}_{\mathbfcal{H}^2} + \norm{(\Grad\mu_k,\Gradg\theta_k)}_{\mathbfcal{H}^2}^2\norm{(\bv_k,\bw_k)}_{\mathbfcal{H}^1_{0}}^2 \\
    &\qquad + \norm{(\mu_k,\theta_k)}_{L,\beta}^2 + \norm{(\Grad\mu_k,\Gradg\theta_k)}_{\mathbfcal{H}^1}^2\norm{(\bv_k,\bw_k)}_{\mathbfcal{H}^1_{0}}^2 \big).
\end{align*}
Thus, by Young's inequality, we have
\begin{align}\label{Est:vw:H^2:k}
    \begin{split}
        \norm{(\bv_k,\bw_k)}_{\mathbfcal{H}^2}^2 + \norm{(p_k,q_k)}_{\mathcal{H}^1}^2 &\leq C_0\norm{(\delt\bv_k,\delt\bw_k)}_{\mathbfcal{L}^2}^2 + C_0\norm{(\mu_k,\theta_k)}_{L,\beta}^2 \\
        &\quad + C_0\big(\norm{(\Grad\mu_k,\Gradg\theta_k)}_{\mathbfcal{H}^2}^2 + \norm{(\bv_k,\bw_k)}_{\mathbfcal{H}^1_{0}}^2\big)\norm{(\bv_k,\bw_k)}_{\mathbfcal{H}^1_{0}}^2.
    \end{split}
\end{align}
Let us now set
\begin{align*}
    \varepsilon = \frac{\min\{\rho_\ast,\sigma_\ast\}}{4C_0}, \qquad \varpi = \frac{\min\{\rho_\ast,\sigma_\ast\}}{8C_0}.
\end{align*}
Then, multiplying \eqref{Est:vw:H^2:k} by $\varepsilon$ and summing the resulting inequality to \eqref{Est:H_k:1}, we deduce the differential inequality
\begin{align}\label{Est:H_k:Final}
    \begin{split}
        &\ddt\frac12\Big(\intO 2\nu_\Om(\phi_k)\abs{\D\bv_k}^2\dx + \intG 2\nu_\Ga(\psi_k)\abs{\Dg\bw_k}^2\dG + \intG \gamma(\phi_k,\psi_k)\abs{\bw_k}^2\dG\Big) \\
        &\qquad + \frac{\min\{\rho_\ast,\sigma_\ast\}}{8C_0}\norm{(\bv_k,\bw_k)}_{\mathbfcal{H}^2}^2 + \frac{\min\{\rho_\ast,\sigma_\ast\}}{4}\norm{(\delt\bv_k,\delt\bw_k)}_{\mathbfcal{L}^2}^2 + \frac{\min\{\rho_\ast,\sigma_\ast\}}{4C_0}\norm{(p_k,q_k)}_{\mathcal{H}^1}^2 \\
        &\quad\leq C\norm{(\mu_k,\theta_k)}_{L,\beta}^2 + C\big(\norm{(\delt\phi_k,\delt\psi_k)}_{K,\alpha}^2 + \norm{(\Grad\mu_k,\Gradg\theta_k)}_{\mathbfcal{H}^2}^2 + \norm{(\bv_k,\bw_k)}_{\mathbfcal{H}^1_{0}}^2\big)\norm{(\bv_k,\bw_k)}_{\mathbfcal{H}^1_{0}}^2.
    \end{split}
\end{align}
Observing that
\begin{align*}
    \norm{(\bv_k(0),\bw_k(0))}_{\mathbfcal{H}^1_{0}} \leq C\norm{(\bv_0,\bw_0)}_{\mathbfcal{H}^1_{0}},
\end{align*}
we infer from the Gronwall lemma together with \eqref{Est:vw:Energy} and \eqref{Est:K2} that
\begin{align}\label{Est:vw:H^1:<1}
    \sup_{0\leq t \leq 1}\norm{(\bv_k(t),\bw_k(t))}_{\mathbfcal{H}^1_{0}}^2 \leq C\big(1 + \norm{(\bv_0,\bw_0)}_{\mathbfcal{H}^1_{0}}^2\big).
\end{align}
To derive a global control independent of time, we use a uniform variant of the Gronwall lemma as stated in Lemma~\ref{Lemma:Gronwall}. Since
\begin{align}
    &\sup_{t\geq 0}\int_t^{t+1} \norm{(\delt\phi_k,\delt\psi_k)}_{K,\alpha}^2 + \norm{(\Grad\mu_k,\Gradg\theta_k)}_{\mathbfcal{H}^2}^2 + \norm{(\bv_k,\bw_k)}_{\mathbfcal{H}^1_{0}}^2 \ds \leq C_1, \label{Est:C_1} \\
    &\sup_{t\geq 0}\int_t^{t+1} \norm{(\Grad\mu_k,\Gradg\theta_k)}_{\mathbfcal{H}^2}^2 \ds \leq C_2, \label{Est:C_2} \\
    &\sup_{t\geq 0}\int_t^{t+1} \norm{(\bv_k,\bw_k)}_{\mathbfcal{H}^1_{0}}^2\ds \leq C_3, \label{Est:C_3}
\end{align}
an application of Lemma~\ref{Lemma:Gronwall} with $r = 1$ and $t_0 = 0$ yields that
\begin{align}\label{Est:vw:H^1:>1}
    \sup_{t\geq 1} \norm{(\bv_k(t),\bw_k(t))}_{\mathbfcal{H}^1_{0}}^2 \leq \big(C_3 + C_2)\e^{C_1}.
\end{align}
Then, for any $t \geq 0$, we integrate \eqref{Est:H_k:Final} in time from $t$ to $t+1$. Thanks to \eqref{Est:C_1}-\eqref{Est:C_3}, it follows from \eqref{Est:vw:H^1:<1} and \eqref{Est:vw:H^1:>1} that
\begin{align}\label{Est:vw:H^2:delt:L^2:pq:uloc}
    \sup_{t\geq 0}\int_t^{t+1} \norm{(\bv_k,\bw_k)}_{\mathbfcal{H}^2}^2 + \norm{(\delt\bv_k,\delt\bw_k)}_{\mathbfcal{L}^2}^2 + \norm{(p_k,q_k)}_{\mathcal{H}^1}^2 \ds \leq C.
\end{align}

\textbf{Step 4: Passage to the Limit and Existence of global Strong Solutions.}
We are in a position to pass to the limit $k\rightarrow\infty$. More precisely, in light of the estimates \eqref{Est:K1p}-\eqref{Est:K3}, \eqref{Est:vw:H^1:<1} and \eqref{Est:vw:H^1:>1}-\eqref{Est:vw:H^2:delt:L^2:pq:uloc}, the Banach--Alaoglu theorem implies the existence of an octuple $(\bv,\bw,p,q,\phi,\psi,\mu,\theta)$ such that, up to a subsequence, as $k\rightarrow\infty$,
\begin{alignat}{3}
    (\bv_k,\bw_k) &\rightarrow (\bv,\bw) &&\qquad\text{weakly-star in~} L^\infty(0,\infty;\mathbfcal{H}^1_{0,\Div}), && \label{Conv:vw:H^1}\\
    & &&\qquad\text{weakly in~} L^2(0,T;\mathbfcal{H}^2)\cap H^1(0,T;\mathbfcal{L}^2) &&\quad\text{for any~} T > 0, \\
    (p_k,q_k) &\rightarrow (p,q) &&\qquad\text{weakly in~} L^2(0,T;\mathcal{H}^1\cap\mathcal{L}^2_{(0)}) &&\quad\text{for any~} T > 0, \\
    (\phi_k,\psi_k) &\rightarrow (\phi,\psi) &&\qquad\text{weakly-star in~} L^\infty(0,\infty;\mathcal{H}^3), \label{Conv:delt:pp} && \\
    & &&\qquad\text{weakly in~} H^1(0,T;\mathcal{H}^1) &&\quad\text{for any~} T > 0, \\
    (\mu_k,\theta_k) &\rightarrow (\mu,\theta) &&\qquad\text{weakly-star in~} L^\infty(0,\infty;\mathcal{H}^1), && \\
    & &&\qquad\text{weakly in~} L^2(0,T;\mathcal{H}^3)\cap H^1(0,T;(\mathcal{H}^1)^\prime) &&\quad \text{for any~} T > 0. 
    \label{Conv:mt:H^1}
\end{alignat}
As a consequence of the Aubin--Lions--Simon lemma, for any $T > 0$, we further have
\begin{alignat}{2}
    (\bv_k,\bw_k) &\rightarrow (\bv,\bw) &&\qquad\text{strongly in~} L^2(0,T;\mathbfcal{H}^1_{0,\Div}), \\
    (\phi_k,\psi_k) &\rightarrow (\phi,\psi) &&\qquad\text{strongly in~} C([0,T];\mathcal{W}^{2,p}) \ \text{for all~}p\in[2,\infty), \label{Conv:pp:strong}\\
    (\mu_k,\theta_k) &\rightarrow (\mu,\theta) &&\qquad\text{strongly in~} L^2(0,T;\mathcal{H}^2).
\end{alignat}
Moreover, by Minty's trick, we have
\begin{align*}
    (F^\prime(\phi_k), G^\prime(\psi_k)) \rightarrow (F^\prime(\phi),G^\prime(\psi)) \qquad\text{weakly-star in~}L^\infty(0,\infty;\mathcal{L}^p)
\end{align*}
for any $2 \leq p < \infty$. Furthermore, on account of the assumptions on the coefficients (see \ref{Assumption:Density}-\ref{Assumption:Coefficients}), we infer from \eqref{Conv:pp:strong} that
\begin{alignat*}{6}
    \rho(\phi_k) &\rightarrow \rho(\phi), &&\qquad m_\Om(\phi_k) &\rightarrow m_\Om(\phi), &&\qquad \nu_\Om(\phi_k) &&\rightarrow \nu_\Om(\phi) &&\qquad\text{strongly in~} C(\overline{\Om}\times[0,T]), \\
    \sigma(\psi_k) &\rightarrow \sigma(\psi), &&\qquad 
    m_\Ga(\psi_k) &\rightarrow m_\Ga(\psi), &&\qquad \nu_\Ga(\psi_k) &&\rightarrow \nu_\Ga(\psi) &&\qquad\text{strongly in~} C(\Ga\times[0,T]),
\end{alignat*}
and
\begin{align*}
        \gamma(\phi_k,\psi_k) &\rightarrow \gamma(\phi,\psi) \qquad\text{strongly in~} C(\Ga\times[0,T])
\end{align*}
for any $T > 0$.
The above properties entail the convergence of the nonlinear terms in \eqref{WF:VW:DISCR}, which allows us to pass to the limit as $k\rightarrow\infty$ in \eqref{WF:VW:DISCR}, \eqref{WF:PP:DISCR} and \eqref{WF:MT:DISCR} (see, e.g., \cite{Knopf2025a} for the limit in the Galerkin formulation). Lastly, we observe that the pressure pair $(p,q)$ satisfies
\begin{align*}
    \Grad p = \rho(\phi)\delt\bv + \rho(\phi)(\bv\cdot\Grad)\bv + (\J\cdot\Grad)\bv - \Div(2\nu_\Om(\phi)\D\bv) - \mu\Grad\phi \qquad\text{a.e.~in~}Q
\end{align*}
and
\begin{align*}
    \Gradg q &= \sigma(\psi)\delt\bw + \sigma(\psi)(\bw\cdot\Gradg)\bw + (\K\cdot\Gradg)\bw - \Divg(2\nu_\Ga(\psi)\Dg\bw) - \theta\Gradg\psi + 2\nu_\Om(\phi)[\D\bv\,\n]_\tau \\
    &\quad - \frac12(\J\cdot\n)\bw + \frac12(\J_\Ga\cdot\n)\bw + \gamma(\phi,\psi)\bw \qquad\text{a.e.~in~}\Sigma,
\end{align*}
which finishes the proof.
\end{proof}

\medskip

\section{Uniqueness}
\label{Section:Uniqueness}

In this section, we prove Theorem~\ref{Theorem:Uniqueness}. Namely, for $L\in(0,\infty]$, we show the uniqueness of strong solutions and their continuous dependence on the initial data. 

\begin{proof}[Proof of Theorem~\ref{Theorem:GlobalExistence}]
Let $(\bv_0^1,\bw_0^1,\phi_0^1,\psi_0^1)$ and $(\bv_0^2,\bw_0^2,\phi_0^2,\psi_0^2)$ be admissible initial data, and consider two global strong solutions $(\bv_1,\bw_1,p_1,q_1,\phi_1,\psi_1,\mu_1,\theta_1)$ and $(\bv_2,\bw_2,p_2,q_2,\phi_2,\psi_2,\mu_2,\theta_2)$ originating from $(\bv_0^1,\bw_0^1,\phi_0^1,\psi_0^1)$ and $(\bv_0^2,\bw_0^2,\phi_0^2,\psi_0^2)$, respectively. Set
\begin{align*}
    (\bv,\bw,p,q,\phi,\psi,\mu,\theta) \coloneqq (\bv_1 - \bv_2, \bw_1 - \bw_2, p_1 - p_2, q_1 - q_2, \phi_1 - \phi_2, \psi_1 - \psi_2, \mu_1 - \mu_2, \theta_1 - \theta_2).
\end{align*}
It is clear that 
\begin{subequations}
    \begin{align}\label{Uniq:v1-v2}
        \begin{split}
            &\rho(\phi_1)\delt\bv + \big(\rho(\phi_1) - \rho(\phi_2)\big)\delt\bv_2 + \big(\rho(\phi_1)(\bv_1\cdot\Grad)\bv_1 - \rho(\phi_2)(\bv_2\cdot\Grad)\bv_2\big) \\
            &\qquad - \frac{\tilde\rho_1 - \tilde\rho_2}{2}\big((m_\Om(\phi_1)\Grad\mu_1\cdot\Grad)\bv_1 - (m_\Om(\phi_2)\Grad\mu_2\cdot\Grad)\bv_2\big) \\
            &\qquad - \Div(2\nu_\Om(\phi_1)\D\bv) - \Div\big(2(\nu_\Om(\phi_1) - \nu_\Om(\phi_2))\D\bv_2\big) + \Grad p \\
            &\quad = \mu_1\Grad\phi_1 - \mu_2\Grad\phi_2 \qquad\text{a.e.~in~}Q,
        \end{split}
        \\[8pt]
        \begin{split}\label{Uniq:w1-w2}
            &\sigma(\psi_1)\delt\bw + \big(\sigma(\psi_1) - \sigma(\psi_2)\big)\delt\bw_2 + \big(\sigma(\psi_1)(\bw_1\cdot\Gradg)\bw_1 - \sigma(\psi_2)(\bw_2\cdot\Gradg)\bw_2\big) \\
            &\qquad - \frac{\tilde\sigma_1 - \tilde\sigma_2}{2}\big((m_\Ga(\psi_1)\Gradg\theta_1\cdot\Gradg)\bw_1 - (m_\Ga(\psi_2)\Gradg\theta_2\cdot\Gradg)\bw_2\big) \\
            &\qquad - \Divg(2\nu_\Ga(\psi_1)\Dg\bw) - \Divg\big(2(\nu_\Ga(\psi_1) - \nu_\Ga(\psi_2))\Dg\bw_2\big) + \Gradg q \\
            &\quad = \theta_1\Gradg\psi_1 - \theta_2\Gradg\psi_2 - 2\nu_\Om(\phi_1)\D\bv\,\n - (2\nu_\Om(\phi_1) - 2\nu_\Om(\phi_2))\D\bv_2\,\n \\
            &\qquad - \frac{\tilde\rho_1 - \tilde\rho_2}{2}m_\Om(\phi_1)\deln\mu_1\bw_1 + \frac{\tilde\rho_1 - \tilde\rho_2}{2}m_\Om(\phi_2)\deln\mu_2\bw_2 \\
            &\qquad + \beta\frac{\tilde\sigma_1 - \tilde\sigma_2}{2}m_\Om(\phi_1)\deln\mu_1\bw_1 - \beta\frac{\tilde\sigma_1 - \tilde\sigma_2}{2}m_\Om(\phi_2)\deln\mu_2\bw_2 \\
            &\qquad - \gamma(\phi_1,\psi_1)\bw_1 + \gamma(\phi_2,\psi_2)\bw_2 \qquad\text{a.e.~on~}\Sigma,
        \end{split}
    \end{align}
    as well as
    \begin{alignat}{2}
        &\delt\phi + \bv_1\cdot\Grad\phi + \bv\cdot\Grad\phi_2 = \Div(m_\Om(\phi_1)\Grad\mu) + \Div\big((m_\Om(\phi_1) - m_\Om(\phi_2))\Grad\mu_2\big) &&\text{a.e.~in~}Q, \label{Uniq:phi1-phi2} \\
        &\mu = - \Lap\phi + F^\prime(\phi_1) - F^\prime(\phi_2) &&\text{a.e.~in~}Q, \label{Uniq:mu1-mu2} \\
        &\delt\psi + \bw_1\cdot\Gradg\psi + \bw\cdot\Gradg\psi_2 = \Divg(m_\Ga(\psi_1)\Gradg\theta) + \Divg\big((m_\Ga(\psi_1) - m_\Ga(\psi_2))\Gradg\theta_2\big) \nonumber \\
        &\qquad\qquad\qquad\qquad\qquad\qquad\qquad - \beta m_\Om(\phi_1)\deln\mu - \beta(m_\Om(\phi_1) - m_\Om(\phi_2))\deln\mu_2 &&\text{a.e.~on~}\Sigma, \label{Uniq:psi1-psi2} \\
        &\theta = -\Lapg\psi + G^\prime(\psi_1) - G^\prime(\psi_2) + \alpha\deln\phi &&\text{a.e.~on~}\Sigma. \label{Uniq:theta1-theta2}
    \end{alignat}
\end{subequations}
Then, we multiply \eqref{Uniq:v1-v2} and \eqref{Uniq:w1-w2} by $\bv$ and $\bw$, respectively, and integrate the respective equations over $\Om$ and $\Ga$. Adding the resulting equations leads to
\begin{align}\label{Uniq:ddt:kinetic}
    &\ddt\frac12\Big( \intO \rho(\phi_1)\abs{\bv}^2\dx + \intG\sigma(\psi_1)\abs{\bw}^2\dG\Big) \nonumber \\
    &\qquad +
    \intO 2\nu_\Om(\phi_1)\abs{\D\bv}^2\dx + \intG 2\nu_\Ga(\psi_1)\abs{\Dg\bw}^2\dG + \intG \gamma(\phi_1,\psi_1)\abs{\bw}^2\dG \nonumber \\
    &\quad = - \intO \big(\rho(\phi_1) - \rho(\phi_2)\big)\delt\bv_2\cdot\bv\dx - \intG \big(\sigma(\psi_1) - \sigma(\psi_2)\big)\delt\bw_2\cdot\bw\dG \nonumber \\
    &\qquad - \intO \rho(\phi_1)(\bv\cdot\Grad)\bv_2\cdot\bv - \intG \sigma(\psi_1)(\bw\cdot\Gradg)\bw_2\cdot\bw\dG \nonumber \\
    &\qquad - \intO \big(\rho(\phi_1) - \rho(\phi_2)\big)(\bv_2\cdot\Grad)\bv_2\cdot\bv\dx - \intG \big(\sigma(\psi_1) - \sigma(\psi_2)\big)(\bw_2\cdot\Gradg)\bw_2\cdot\bw\dG \nonumber \\
    &\qquad + \frac{\tilde\rho_1 - \tilde\rho_2}{2}\intO m_\Om(\phi_1)(\Grad\mu\cdot\Grad)\bv_2\cdot\bv\dx + \frac{\tilde\sigma_1 - \tilde\sigma_2}{2}\intG m_\Ga(\psi_1)(\Gradg\theta\cdot\Gradg)\bw_2\cdot\bw\dG \nonumber \\
    &\qquad + \frac{\chi(L)}{2}\Big(\beta\frac{\tilde\sigma_1 - \tilde\sigma_2}{2} - \frac{\tilde\rho_1 - \tilde\rho_2}{2}\Big)\intG (\beta\theta - \mu)\bw_2\cdot\bw\dG \\ 
    &\qquad + \frac{\tilde\rho_1 - \tilde\rho_2}{2}\intO \big(m_\Om(\phi_1) - m_\Om(\phi_2)\big)(\Grad\mu_2\cdot\Grad)\bv_2\cdot\bv\dx \nonumber \\
    &\qquad + \frac{\tilde\sigma_1 - \tilde\sigma_2}{2}\intG \big(m_\Ga(\psi_1) - m_\Ga(\psi_2)\big)(\Gradg\theta_2\cdot\Gradg)\bw_2\cdot\bw\dG \nonumber \\
    &\qquad + \intO 2\big(\nu_\Om(\phi_1) - \nu_\Om(\phi_2)\big)\D\bv_2:\D\bv\dx + \intG 2\big(\nu_\Ga(\psi_1) - \nu_\Ga(\psi_2)\big)\Dg\bw_2:\Dg\bw\dG \nonumber \\
    &\qquad + \intG \big(\gamma(\phi_1,\psi_1) - \gamma(\phi_2,\psi_2)\big)\bw_2\cdot\bw\dG \nonumber \\
    &\eqqcolon \sum_{j=1}^{14} W_j. \nonumber
\end{align}
Here, we made use of
\begin{align*}
    &-\frac12\intO \delt\rho(\phi_1)\abs{\bv}^2\dx - \frac12\intG \delt\sigma(\psi_1)\abs{\bw}^2\dG \nonumber \\
    &\qquad + \frac12\intO \rho(\phi_1)\bv_1\cdot\Grad\abs{\bv}^2\dx + \frac12\intG \sigma(\psi_1)\bw_1\cdot\Gradg\abs{\bw}^2\dG \nonumber \\
    &\qquad - \frac12\frac{\tilde\rho_1 - \tilde\rho_2}{2}\intO m_\Om(\phi_1)\Grad\mu_1\cdot\Grad\abs{\bv}^2\dx - \frac12\frac{\tilde\sigma_1 - \tilde\sigma_2}{2}\intG m_\Ga(\psi_1)\Gradg\theta_1\cdot\Gradg\abs{\bw}^2\dG \\
    &\qquad - \frac{\chi(L)}{2}\Big(\beta\frac{\tilde\sigma_1 - \tilde\sigma_2}{2} - \frac{\tilde\rho_1 - \tilde\rho_2}{2}\Big)\intG (\beta\theta_1 - \mu_1)\abs{\bw}^2\dG \nonumber \\
    &\quad = 0, \nonumber
\end{align*}
which follows from the same considerations as in the proof of Theorem~\ref{Theorem:GlobalExistence} (see also \eqref{Equation:densities}). Next, noting on
\begin{align*}
    Lm_\Om(\phi_1)\deln\mu = \beta\theta - \mu + \Big(1 - \frac{m_\Om(\phi_1)}{m_\Om(\phi_2)}\Big)\big(\beta\theta_2 - \mu_2) \qquad\text{a.e.~on~}\Sigma
\end{align*}
for $L > 0$, a straightforward calculation shows in view of \eqref{Uniq:phi1-phi2}-\eqref{Uniq:theta1-theta2} that
\begin{align}\label{Uniq:ddt:Ka}
    &\ddt\frac12\norm{(\phi,\psi)}_{K,\alpha}^2 + \norm{(\mu,\theta)}_{L,\beta,[\phi_1,\psi_1]}^2 \nonumber \\
    &\quad = - \big\langle (\delt\phi,\delt\psi), (F^\prime(\phi_1) - F^\prime(\phi_2),G^\prime(\psi_1) - G^\prime(\psi_2))\big\rangle_{\mathcal{H}^1} \nonumber \\
    &\qquad - \intO \big(m_\Om(\phi_1) - m_\Om(\phi_2)\big)\Grad\mu_2\cdot\Grad\mu\dx - \intG \big(m_\Ga(\psi_1) - m_\Ga(\psi_2)\big)\Gradg\theta_2\cdot\Gradg\theta\dG \\
    &\qquad + \intO \big(\phi_1\bv + \phi\bv_2\big)\cdot\Grad\mu\dx + \intG (\psi_1\bw + \psi\bw_2\big)\cdot\Gradg\theta\dG \nonumber \\
    &\quad\eqqcolon \sum_{j=15}^{19} W_j. \nonumber
\end{align}
Hence, summing \eqref{Uniq:ddt:kinetic} and \eqref{Uniq:ddt:Ka}, we obtain
\begin{align*}
    &\ddt\frac12\Big( \intO \rho(\phi_1)\abs{\bv}^2\dx + \intG\sigma(\psi_1)\abs{\bw}^2\dG + \norm{(\phi,\psi)}_{K,\alpha}^2 \Big) \\
    &\qquad + \intO m_\Om(\phi_1)\abs{\Grad\mu}^2\dx + \intG m_\Ga(\psi_1)\abs{\Gradg\theta}^2\dG + \chi(L)\intG (\beta\theta - \mu)^2\dG \\ 
    &\qquad +
    \intO 2\nu_\Om(\phi_1)\abs{\D\bv}^2\dx + \intG 2\nu_\Ga(\psi_1)\abs{\Dg\bw}^2\dG + \intG \gamma(\phi_1,\psi_1)\abs{\bw}^2\dG \\ 
    &\quad = \sum_{j=1}^{19} W_j.
\end{align*}
To control the terms $W_j$, $j=1,\ldots,19$, we start by exploiting \eqref{Est:Korn} as well as the regularity of strong solutions to obtain
\begin{align}
    \label{W12}
    \begin{split}
    &\abs{W_1 + W_2}
        \begin{aligned}[t]
             &\leq C\norm{(\phi,\psi)}_{\mathcal{L}^6}\norm{(\delt\bv_2,\delt\bw_2)}_{\mathbfcal{L}^2}\norm{(\bv,\bw)}_{\mathbfcal{L}^3} \\
            &\leq \frac{\min\{2\nu_\ast,\gamma_\ast\}}{8}\norm{(\bv,\bw)}_{\mathbfcal{H}^1_{0}}^2 + C\norm{(\delt\bv_2,\delt\bw_2)}_{\mathbfcal{L}^2}^2\norm{(\phi,\psi)}_{K,\alpha}^2,
        \end{aligned}
    \end{split}
    \\[10pt]
    \label{W34}
    \begin{split}
    &\abs{W_3 + W_4} 
        \begin{aligned}[t]
            &\leq C\norm{(\bv,\bw)}_{\mathbfcal{L}^3}\norm{(\Grad\bv_2,\Gradg\bw_2)}_{\mathbfcal{L}^6}\norm{(\bv,\bw)}_{\mathbfcal{L}^2} \\
            &\leq \frac{\min\{2\nu_\ast,\gamma_\ast\}}{8}\norm{(\bv,\bw)}_{\mathbfcal{H}^1_{0}}^2 + C\norm{(\bv_2,\bw_2)}_{\mathbfcal{H}^2}^2\norm{(\bv,\bw)}_{\mathbfcal{L}^2}^2,
        \end{aligned}
    \end{split}
    \\[10pt]
    \label{W56}
    \begin{split}
    &\abs{W_5 + W_6} 
        \begin{aligned}[t]
            &\leq C\norm{(\phi,\psi)}_{\mathcal{L}^6}\norm{(\bv_2,\bw_2)}_{\mathbfcal{L}^6}\norm{(\Grad\bv_2,\Gradg\bw_2)}_{\mathbfcal{L}^6}\norm{(\bv,\bw)}_{\mathbfcal{L}^2} \\
            &\leq C\norm{(\bv_2,\bw_2)}_{\mathbfcal{H}^2}^2\big(\norm{(\bv,\bw)}_{\mathbfcal{L}^2}^2 + \norm{(\phi,\psi)}_{K,\alpha}^2\big).
        \end{aligned}
    \end{split}
\end{align}
Then, using additionally \eqref{InterpolEst:L^4}, we find that
\begin{align}\label{W789}
    \abs{W_7 + W_8 + W_9} &\leq C\norm{(m_\Om(\phi_1)\Grad\mu,m_\Ga(\psi_1)\Gradg\theta)}_{\mathbfcal{L}^2}\norm{(\Grad\bv_2,\Gradg\bw_2)}_{\mathbfcal{L}^4}\norm{(\bv,\bw)}_{\mathbfcal{L}^4}  \nonumber \\
    &\quad + C\norm{\beta\theta - \mu}_{\mathbf{L}^2(\Ga)}\norm{\bw_2}_{\mathbf{L}^4(\Ga)}\norm{\bw}_{\mathbf{L}^4(\Ga)} \nonumber \\
    &\leq \frac18\norm{(\mu,\theta)}_{L,\beta,[\phi_1,\psi_1]}^2 + C\norm{(\Grad\bv_2,\Gradg\bw_2)}_{\mathbfcal{L}^2}\norm{(\Grad\bv_2,\Gradg\bw_2)}_{\mathbfcal{H}^1}\norm{(\bv,\bw)}_{\mathbfcal{L}^2}\norm{(\bv,\bw)}_{\mathbfcal{H}^1_{0}}  \nonumber \\
    &\quad + C\norm{\bw_2}_{\mathbf{L}^4(\Ga)}^2\norm{\bw}_{\mathbf{L}^2(\Ga)}\norm{\bw}_{\mathbf{H}^1(\Ga)} \\
    &\leq \frac18\norm{(\mu,\theta)}_{L,\beta,[\phi_1,\psi_1]}^2 + \frac{\min\{2\nu_\ast,\gamma_\ast\}}{8}\norm{(\bv,\bw)}_{\mathbfcal{H}^1_{0}}^2 + C\big(1 + \norm{(\bv_2,\bw_2)}_{\mathbfcal{H}^2}^2\big)\norm{(\bv,\bw)}_{\mathbfcal{L}^2}^2. \nonumber
\end{align}
Moreover, we have
\begin{align}\label{W1011}
    \begin{split}
    \abs{W_{10} + W_{11}} 
        &\leq C\norm{(\phi,\psi)}_{\mathcal{L}^6}\norm{(\Grad\mu_2,\Gradg\theta_2)}_{\mathbfcal{L}^6}\norm{(\Grad\bv_2,\Gradg\bw_2)}_{\mathbfcal{L}^6}\norm{(\bv,\bw)}_{\mathbfcal{L}^2} \\
        &\leq C\big(\norm{(\Grad\mu_2,\Gradg\theta_2)}_{\mathbfcal{H}^1}^2 + \norm{(\bv_2,\bw_2)}_{\mathbfcal{H}^2}^2\big)\big(\norm{(\bv,\bw)}_{\mathbfcal{L}^2}^2 + \norm{(\phi,\psi)}_{K,\alpha}^2\big),
    \end{split}
\end{align}
and
\begin{align}\label{W121314}
    \abs{W_{12}+ W_{13} + W_{14}} &\leq C\norm{(\phi,\psi)}_{\mathcal{L}^6}\norm{(\D\bv_2,\Dg\bw_2)}_{\mathbfcal{L}^3}\norm{(\D\bv,\Dg\bw)}_{\mathbfcal{L}^2} \nonumber\\
    &\quad + C\big(\norm{\phi}_{L^4(\Ga)} + \norm{\psi}_{L^4(\Ga)}\big)\norm{\bw_2}_{\mathbf{L}^4(\Ga)}\norm{\bw}_{\mathbf{L}^2(\Ga)} \\
    &\leq \frac{\min\{2\nu_\ast,\gamma_\ast\}}{8}\norm{(\bv,\bw)}_{\mathbfcal{H}^1_{0}}^2 + C\norm{(\bv_2,\bw_2)}_{\mathbfcal{H}^2}^2\norm{(\phi,\psi)}_{K,\alpha}^2. \nonumber
\end{align}
Next, 
\begin{align}\label{Uniq:Est:W15:1}
    \abs{W_{15}} \leq \norm{(\delt\phi,\delt\psi)}_{(\mathcal{H}^1)}\norm{(F^\prime(\phi_1) - F^\prime(\phi_2), G^\prime(\psi_1) - G^\prime(\psi_2))}_{\mathcal{H}^1}.
\end{align}
By means of a comparison argument in \eqref{Uniq:phi1-phi2}-\eqref{Uniq:psi1-psi2}, we observe that
\begin{align}\label{Uniq:Est:W15:delt}
    &\norm{(\delt\phi,\delt\psi)}_{(\mathcal{H}^1)^\prime} \nonumber \\
    &= \sup_{\norm{(\zeta,\xi)}_{\mathcal{H}^1}\leq 1} \abs{\langle (\delt\phi,\delt\psi), (\zeta,\xi)\rangle_{\mathcal{H}^1}} \\
    &\leq C\big(\norm{(\bv,\bw)}_{\mathbfcal{L}^2} + \norm{(\phi,\psi)}_{\mathcal{L}^4}\norm{(\bv_2,\bw_2)}_{\mathbfcal{L}^4} + \norm{(\mu,\theta)}_{L,\beta,[\phi_1,\psi_1]} + \norm{(\phi,\psi)}_{\mathcal{L}^4}\norm{(\Grad\mu_2,\Gradg\theta_2)}_{\mathbfcal{L}^4}\big). \nonumber
\end{align}
For the second term on the right-hand side of \eqref{Uniq:Est:W15:1}, we recall the existence of $\delta\in(0,1]$ such that
\begin{align}\label{Uniq:Separation}
    \abs{\phi_j} \leq 1 - \delta \qquad\text{on~}Q, \qquad \abs{\psi_j} \leq 1 - \delta \qquad\text{on~}\Sigma
\end{align}
for $j = 1,2$. In particular, for any $s\in[0,1]$, it holds that
\begin{align*}
    \abs{s\phi_1 + (1-s)\phi_2} \leq s(1-\delta) + (1-s)(1-\delta) = 1 - \delta \qquad\text{on~}Q.
\end{align*}
Similarly, we find that
\begin{align*}
    \abs{s\psi_1 + (1-s)\psi_2} \leq 1 - \delta \qquad\text{on~}\Sigma
\end{align*}
for any $s\in[0,1]$. Noting on the decomposition $F(s) = F_0(s) - \tfrac{c_F}{2}s^2$ with $F_0\in C^3(-1,1)$, we use the fundamental theorem of calculus to infer
\begin{align}\label{Uniq:Est:F:H^1}
    &\norm{F^\prime(\phi_1) - F^\prime(\phi_2)}_{H^1(\Om)}^2 \nonumber \\
    &\quad\leq \intO \abs{F_0^\prime(\phi_1) - F_0^\prime(\phi_2)}^2\dx + \intO \abs{F_0^{\prime\prime}(\phi_1)}^2\abs{\Grad\phi}^2\dx + \intO \abs{F_0^{\prime\prime}(\phi_1) - F_0^{\prime\prime}(\phi_2)}^2\abs{\Grad\phi_2}^2\dx \nonumber \\
    &\qquad + c_F^2\intO \phi^2\dx + c_F^2\intO \abs{\Grad\phi}^2\dx \nonumber \\
    &\quad\leq \intO \Bigabs{\int_0^1 F_0^{\prime\prime}(s\phi_1 + (1-s)\phi_2)^2\ds}^2\phi^2\dx + C\norm{\Grad\phi}_{\mathbf{L}^2(\Om)}^2 \\
    &\qquad + \intO \Bigabs{\int_0^1 F_0^{\prime\prime\prime}(s\phi_1 + (1-s)\phi_2)^2\ds}^2\phi^2\abs{\Grad\phi_2}^2\dx + c_F^2\norm{\phi}_{L^2(\Om)}^2 + c_F^2\norm{\Grad\phi}_{\mathbf{L}^2(\Om)}^2 \nonumber \\
    &\quad\leq \sup_{s\in[-1+\delta,1-\delta]}\abs{F_1^{\prime\prime}(s)}^2\norm{\phi}_{L^2(\Om)}^2 + C\norm{\phi}_{H^1(\Om)}^2 + \sup_{s\in[-1+\delta,1-\delta]}\abs{F_1^{\prime\prime\prime}(s)}^2\norm{\phi}_{L^4(\Om)}^2\norm{\Grad\phi_2}_{\mathbf{L}^4(\Om)}^2 \nonumber \\
    &\quad\leq C\norm{(\phi,\psi)}_{K,\alpha}^2. \nonumber
\end{align}
Performing similar computations on the boundary, we deduce
\begin{align}\label{Uniq:Est:G:H^1}
    \norm{G^\prime(\psi_1) - G^\prime(\psi_2)}_{H^1(\Ga)}^2 \leq C\norm{(\phi,\psi)}_{K,\alpha}^2.
\end{align}
Combining \eqref{Uniq:Est:F:H^1} and \eqref{Uniq:Est:G:H^1} leads to 
\begin{align}\label{Uniq:Est:Pot:H^1}
    \norm{(F^\prime(\phi_1) - F^\prime(\phi_2), G^\prime(\psi_1) - G^\prime(\psi_2))}_{\mathcal{H}^1} \leq C\norm{(\phi,\psi)}_{K,\alpha}.
\end{align}
Thus, going back to \eqref{Uniq:Est:W15:1}, the estimates \eqref{Uniq:Est:W15:delt} and \eqref{Uniq:Est:Pot:H^1} entail that
\begin{align}\label{W15}
    \abs{W_{15}} &\leq C\big(\norm{(\bv,\bw)}_{\mathbfcal{L}^2} + \norm{(\phi,\psi)}_{\mathcal{L}^4}\norm{(\bv_2,\bw_2)}_{\mathbfcal{L}^4} + \norm{(\mu,\theta)}_{L,\beta,[\phi_1,\psi_1]} + \norm{(\phi,\psi)}_{\mathcal{L}^4}\norm{(\Grad\mu_2,\Gradg\theta_2)}_{\mathbfcal{L}^4}\big) \nonumber \\
    &\qquad\times C\norm{(\phi,\psi)}_{K,\alpha} \\
    &\leq \frac{\min\{1,m_\ast\}}{8}\norm{(\mu,\theta)}_{L,\beta}^2 + C\big(1 + \norm{(\bv_2,\bw_2)}_{\mathbfcal{H}^1}^2 + \norm{(\Grad\mu_2,\Gradg\theta_2)}_{\mathbfcal{H}^1}^2\big)\big(\norm{(\bv,\bw)}_{\mathbfcal{L}^2}^2 + \norm{(\phi,\psi)}_{K,\alpha}^2\big). \nonumber
\end{align}
Lastly, for the remaining terms, we have
\begin{align}
    \label{W1617}
    \begin{split}
        \begin{aligned}[t]
            \abs{W_{16} + W_{17}} &\leq C\norm{(\phi,\psi)}_{\mathcal{L}^4}\norm{(\Grad\mu_2,\Gradg\theta_2)}_{\mathbfcal{L}^4}\norm{(\Grad\mu,\Gradg\theta)}_{\mathbfcal{L}^2} \\
            &\leq \frac{\min\{1,m_\ast\}}{8}\norm{(\mu,\theta)}_{L,\beta}^2 + C\norm{(\Grad\mu_2,\Gradg\theta_2)}_{\mathbfcal{H}^1}^2\norm{(\phi,\psi)}_{K,\alpha}^2,
        \end{aligned}
    \end{split}
    \\[10pt]
    \label{W1819}
    \begin{split}
        \begin{aligned}[t]
            \abs{W_{18} + W_{19}} &\leq \norm{(\bv,\bw)}_{\mathbfcal{L}^2}\norm{(\Grad\mu,\Gradg\theta)}_{\mathbfcal{L}^2} + C\norm{(\phi,\psi)}_{\mathcal{L}^4}\norm{(\bv_2,\bw_2)}_{\mathbfcal{L}^4}\norm{(\Grad\mu,\Gradg\theta)}_{\mathbfcal{L}^2} \\
            &\leq \frac{\min\{1,m_\ast\}}{8}\norm{(\mu,\theta)}_{L,\beta}^2 + C\big(1 + \norm{(\bv_2,\bw_2)}_{\mathbfcal{H}^1}^2\big)\big(\norm{(\bv,\bw)}_{\mathbfcal{L}^2}^2 + \norm{(\phi,\psi)}_{K,\alpha}^2\big).
        \end{aligned}
    \end{split}
\end{align}
Therefore, by \eqref{W12}-\eqref{W121314} and \eqref{W15}-\eqref{W1819}, we obtain the differential inequality
\begin{align*}
    &\ddt\frac12\Big(\intO \rho(\phi_1)\abs{\bv}^2\dx + \intG \sigma(\psi_1)\abs{\bw}^2\dG + \norm{(\phi,\psi)}_{K,\alpha}^2\Big) + \frac{\min\{1,m_\ast\}}{2}\norm{(\mu,\theta)}_{L,\beta}^2 \\
    &\qquad + \frac{\min\{2\nu_\ast,\gamma_\ast\}}{2}\norm{(\bv,\bw)}_{\mathbfcal{H}^1_{0}}^2 \\
    &\quad\leq C\big(1 + \norm{(\delt\bv_2,\delt\bw_2)}_{\mathbfcal{L}^2}^2 + \norm{(\bv_2,\bw_2)}_{\mathbfcal{H}^2}^2 + \norm{(\Grad\mu_2,\Gradg\theta_2)}_{\mathbfcal{H}^1}^2\big)\big(\norm{(\bv,\bw)}_{\mathbfcal{L}^2}^2 + \norm{(\phi,\psi)}_{K,\alpha}^2\big).
\end{align*}
Noting on $\rho \geq \rho_\ast > 0$ a.e. on $Q$ and $\sigma \geq \sigma_\ast > 0$ a.e. on $\Sigma$, we deduce with the help of the Gronwall lemma that
\begin{align*}
    &\norm{(\bv(t),\bw(t))}_{\mathbfcal{L}^2} + \norm{(\phi(t),\psi(t))}_{K,\alpha}^2 \\
    &\quad\leq C\big(\norm{(\bv(0),\bw(0))}_{\mathbfcal{L}^2}^2 + \norm{(\phi(0),\psi(0))}_{K,\alpha}^2\big) \\
    &\qquad\times\exp\Big(\int_0^T 1 + \norm{(\delt\bv_2(\tau),\delt\bw_2(\tau))}_{\mathbfcal{L}^2}^2 + \norm{(\bv_2(\tau),\bw_2(\tau))}_{\mathbfcal{H}^2}^2 + \norm{(\Grad\mu_2(\tau),\Gradg\theta_2(\tau))}_{\mathbfcal{H}^1}^2\dtau\Big)
\end{align*}
for all $t\in[0,T]$ for any $T > 0$, which implies the uniqueness of strong solutions.
\end{proof}

\medskip
\appendix
\section{The convective bulk-surface Cahn--Hilliard equation with non-degenerate mobility}
\label{Section:ConvCH}
\setcounter{equation}{0}

In this section, we study the following convective bulk-surface Cahn--Hilliard equation with non-degenerate mobility:
\begin{subequations}\label{EQ:CONV:SYSTEM}
    \begin{align}
        \label{EQ:CONV:SYSTEM:1}
        &\delt\phi + \Div(\phi\mathbf{v}) = \Div(m_\Om(\phi)\Grad\mu) && \text{in} \ Q, \\
        \label{EQ:CONV:SYSTEM:2}
        &\mu = -\Lap\phi + F'(\phi)   && \text{in} \ Q, \\
        \label{EQ:CONV:SYSTEM:3}
        &\delt\psi + \Divg(\psi\mathbf{w}) = \Divg(m_\Ga(\psi)\Gradg\theta) - \beta m_\Om(\phi)\deln\mu && \text{on} \ \Sigma, \\
        \label{EQ:CONV:SYSTEM:4}
        &\theta = - \Lapg\psi + G'(\psi) + \alpha\deln\phi && \text{on} \ \Sigma, \\
        \label{EQ:CONV:SYSTEM:5}
        &\begin{cases} K\deln\phi = \alpha\psi - \phi &\text{if} \ K\in [0,\infty), \\
        \deln\phi = 0 &\text{if} \ K = \infty
        \end{cases} && \text{on} \ \Sigma, \\
        \label{EQ:CONV:SYSTEM:6}
        &\begin{cases} 
        L m_\Om(\phi)\deln\mu = \beta\theta - \mu &\text{if} \  L\in[0,\infty), \\
        m_\Om(\phi)\deln\mu = 0 &\text{if} \ L=\infty
        \end{cases} &&\text{on} \ \Sigma, \\
        \label{EQ:CONV:SYSTEM:7}
        &\phi\vert_{t=0} = \phi_0 &&\text{in} \ \Om, \\
        \label{EQ:CONV:SYSTEM:8}
        &\psi\vert_{t=0} = \psi_0 &&\text{on} \ \Ga,
    \end{align}
\end{subequations}
where the mobility functions $m_\Om$ and $m_\Ga$ are supposed to satisfy the assumptions as stated in \ref{Assumption:Coefficients}. Moreover, the prescribed velocity fields are such that $(\mathbf{v},\mathbf{w})\in L^2(0,\infty;\mathbfcal{L}^2_\Div)$. The system \eqref{EQ:CONV:SYSTEM} can be regarded as a subsystem of the full Navier--Stokes--Cahn--Hilliard system \eqref{System}. For constant mobility functions, the well-posedness of strong solutions to \eqref{EQ:CONV:SYSTEM} has already been investigated in several works, see, for instance, \cite{Knopf2024, Knopf2025, Giorgini2025}. In the case of non-degenerate mobilities and vanishing vector fields, i.e., $\bv \equiv 0$ and $\bw \equiv 0$, the recent work \cite{Stange2025} establishes the well-posedness of weak solutions together with the propagation of regularity. This section aims to extend these results to the case of non-trivial vector fields. 

We start with the following result regarding the existence of global-in-time weak solutions.
\begin{theorem}\label{Theorem:CCH:Existence}
    Suppose that the mobility functions $m_\Om, m_\Ga\in C([-1,1])$ satisfy \ref{Assumption:Mobility:Bound} and that the potentials satisfy \ref{Assumption:Potential}-\eqref{Assumption:DominationProperty}. Let $K\in(0,\infty)$, $L\in[0,\infty]$, $(\bv,\bw)\in L^2(0,\infty;\mathbfcal{L}^2_\Div)$, and let $(\phi_0,\psi_0)\in\mathcal{H}^1$ be an arbitrary initial datum satisfying
    \begin{subequations}\label{cond:init}
        \begin{align}\label{cond:init:int}
            \norm{\phi_0}_{L^\infty(\Om)} \leq 1 \quad\text{and}\quad \norm{\psi_0}_{L^\infty(\Ga)} \leq 1.
        \end{align}
        In addition, we assume that
        \begin{align}\label{cond:init:mean:L}
            &\beta\,\mean{\phi_0}{\psi_0},\ \mean{\phi_0}{\psi_0}\in(-1,1), \quad \text{if~ }L\in[0,\infty),
        \end{align}
        and 
        \begin{align}\label{cond:init:mean:inf}
            \meano{\phi_0},\ \meang{\psi_0}\in(-1,1), \quad \text{if~ }L=\infty.
        \end{align}
    \end{subequations}
    Then, there exists a global-in-time weak solution $(\phi,\psi,\mu,\theta)$ to \eqref{EQ:CONV:SYSTEM} such that:
    \begin{enumerate}[label=\textnormal{(\roman*)}, topsep=0ex, leftmargin=*, itemsep=1.5ex]
        \item \label{CCH:WS:Reg} The weak solution satisfies
        \begin{align*}
            (\phi,\psi)&\in BC_w([0,\infty);\mathcal{H}^1)\cap L^4_{\mathrm{uloc}}([0,\infty);\mathcal{H}^2)\cap L^2_{\mathrm{uloc}}([0,\infty);\mathcal{W}^{2,p}), \\
            (\delt\phi,\delt\psi)&\in L^2(0,\infty;(\mathcal{H}^1_{L,\beta})^\prime), \\
            (\mu,\theta)&\in L^2_{\mathrm{uloc}}([0,\infty);\mathcal{H}^1_{L,\beta}), \\
            (F^\prime(\phi),G^\prime(\psi))&\in L^2_{\mathrm{uloc}}([0,\infty);\mathcal{L}^p)
        \end{align*}
        for any $2 \leq p < \infty$, and it holds
        \begin{align}\label{CCH:pp:<1}
            \abs{\phi} < 1 \quad\text{a.e.~in~}Q\quad\text{and} \quad\abs{\psi} < 1 \quad\text{a.e.~in~}\Sigma.
        \end{align}
        \item \label{CCH:WS:WeakForm} The weak solution solves \eqref{EQ:CONV:SYSTEM} in the following variational sense
        \begin{align}\label{WF:CCH:PP}
            \begin{split}
                &\big\langle (\delt\phi,\delt\psi), (\zeta,\xi)\big\rangle_{\mathcal{H}^1_{L,\beta}} + \intO \phi\bv\cdot\Grad\zeta\dx + \intG \psi\bw\cdot\Gradg\xi\dG \\
                &\quad = -\intO m_\Om(\phi)\Grad\mu\cdot\Grad\zeta\dx - \intG m_\Ga(\psi)\Gradg\theta\cdot\Gradg\xi\dG - \chi(L) \intG (\beta\theta - \mu)(\beta\xi - \zeta)\dG
            \end{split}
        \end{align}
        a.e. on $(0,\infty)$ for all $(\zeta,\xi)\in\mathcal{H}^1_{L,\beta}$, where $(\mu,\theta)$ are given by
        \begin{align}
            &\mu = -\Lap\phi + F^\prime(\phi) &&\text{a.e.~in } Q, \label{Eq:mu:strong}\\
            &\theta = -\Lapg\psi + G^\prime(\psi) + \alpha\deln\phi &&\text{a.e.~on } \Sigma, \label{Eq:theta:strong}\\
            &K\deln\phi = \alpha\psi - \phi &&  \text{a.e.~on } \Sigma.\label{Eq:bd:strong}
        \end{align}
        Moreover, it holds that
        \begin{align*}
            \phi\vert_{t=0} = \phi_0 \quad\text{a.e.~in~}\Om \quad\text{and}\quad\psi\vert_{t=0} = \psi_0 \quad\text{a.e.~on~}\Ga.
        \end{align*}
        \item \label{CCH:MCL} The functions $\phi$ and $\psi$ satisfy the mass conservation law
        \begin{align}\label{MCL:SING}
            \begin{dcases}
                \beta\intO \phi(t)\dx + \intG \psi(t)\dG = \beta\intO \phi_0 \dx + \intG \psi_0\dG, &\textnormal{if } L\in[0,\infty), \\
                \intO\phi(t)\dx = \intO\phi_0\dx \quad\textnormal{and}\quad \intG\psi(t)\dG = \intG\psi_0\dG, &\textnormal{if } L = \infty
            \end{dcases}
        \end{align}
        for all $t\in[0,\infty)$.
        \item \label{CCH:WS:EnergyIneq} The weak solution satisfies the energy inequality
        \begin{align}\label{CCH:EnergyIneq}
            &E_{\mathrm{free}}(\phi(t),\psi(t)) + \int_0^t\intO m_\Om(\phi)\abs{\Grad\mu}^2\dxs + \int_0^t\intG m_\Ga(\psi)\abs{\Gradg\theta}^2\dGs \nonumber \\
            &\qquad + \chi(L)\int_0^t\intG (\beta\theta - \mu)^2\dGs \\
            &\quad \leq E_{\mathrm{free}}(\phi_0,\psi_0) + \int_0^t\intO \phi\bv\cdot\Grad\mu\dxs + \int_0^t\intG \psi\bw\cdot\Gradg\theta\dGs \nonumber
        \end{align}
        for a.e. $t\in[0,\infty)$ including $t = 0$.
        \item \label{CCH:WS:EnergyEstimates} The following estimates hold 
        \begin{align}\label{CCH:Est:Energy:Thm}
            \begin{split}
                &\norm{(\phi,\psi)}_{L^\infty(0,\infty;\mathcal{H}^1)}^2 + \norm{(\delt\phi,\delt\psi)}_{L^2(0,\infty;(\mathcal{H}^1_{L,\beta})^\prime)}^2 + \norm{(\Grad\mu,\Gradg\theta)}_{L^2(0,\infty;\mathbfcal{L}^2)}^2 \\
                &\qquad + \chi(L)\norm{\beta\theta - \mu}^2_{L^2(0,\infty;L^2(\Ga))} \leq C\big(1 + \norm{(\bv,\bw)}_{L^2(0,\infty;\mathbfcal{L}^2)}^2 \big),
            \end{split}
        \end{align}
        \vspace{-.5em}
        \begin{align}\label{CCH:Est:L^p:Thm}
            \begin{split}
                &\norm{(\phi,\psi)}_{L^2_{\mathrm{uloc}}([0,\infty);\mathcal{W}^{2,p})}^2 + \norm{(F^\prime(\phi),G^\prime(\psi))}_{L^2_{\mathrm{uloc}}([0,\infty);\mathcal{L}^p)}^2 \\
                &\qquad\leq C_p\big(1 + \norm{(\bv,\bw)}_{L^2(0,\infty;\mathbfcal{L}^2)}^2\big), \hspace{-2em}
            \end{split}
        \end{align}
        \vspace{-.5em}
        \begin{align}\label{CCH:Est:H^2:Thm}
            \norm{(\phi,\psi)}_{L^4_{\mathrm{uloc}}([0,\infty);\mathcal{H}^2)}^4 \leq C\big(1 + \norm{(\bv,\bw)}_{L^2(0,\infty;\mathbfcal{L}^2)}^2\big)^2
        \end{align}
        for any $2 \leq p < \infty$. Here, the constants $C$ and $C_p$ are independent of the initial conditions $(\phi_0,\psi_0)$ and the velocity fields $(\bv,\bw)$.
    \end{enumerate}
\end{theorem}

\begin{proof}
    Let $T > 0$. It was already established in \cite[Theorem~3.2 and Theorem~3.3]{Giorgini2025} that there exists a weak solution $(\phi,\psi,\mu,\theta)$ to \eqref{EQ:CONV:SYSTEM} on a finite time interval $(0,T)$ for $T > 0$ such that
    \begin{align*}
        (\phi,\psi)&\in BC_w([0,T);\mathcal{H}^1)\cap L^4(0,T;\mathcal{H}^2)\cap L^2_{\mathrm{uloc}}([0,\infty);\mathcal{W}^{2,p}), \\
        (\delt\phi,\delt\psi)&\in L^2(0,T;(\mathcal{H}^1_{L,\beta})^\prime), \\
        (\mu,\theta)&\in L^2(0,T;\mathcal{H}^1_{L,\beta}), \\
        (F^\prime(\phi),G^\prime(\psi))&\in L^2(0,T;\mathcal{L}^p)
    \end{align*}
    for any $2 \leq p < \infty$, and the weak solution satisfies the condition \ref{CCH:WS:WeakForm}, \ref{CCH:MCL} and \ref{CCH:WS:EnergyIneq}. To establish the estimates in \ref{CCH:WS:EnergyEstimates}, we make use of the energy inequality \eqref{CCH:EnergyIneq}. First, using \eqref{CCH:pp:<1}, we have by Young's inequality
    \begin{align*}
        &\Bigabs{\int_0^t\intO \phi\bv\cdot\Grad\mu\dxs + \int_0^t\intG \psi\bw\cdot\Gradg\theta\dGs} \\
        &\quad\leq \frac{\min\{1,m_\ast\}}{2}\int_0^t\intO\abs{\Grad\mu}^2\dxs + \frac{\min\{1,m_\ast\}}{2}\int_0^t\intG\abs{\Gradg\theta}^2\dGs \\
        &\qquad + \frac{1}{2\min\{1,m_\ast\}}\int_0^t\norm{(\bv,\bw)}_{\mathbfcal{L}^2}^2\ds.
    \end{align*}
    Recalling \eqref{CCH:EnergyIneq}, it readily follows that
    \begin{align*}
        &E_{\mathrm{free}}(\phi(t),\psi(t)) + \frac{\min\{1,m_\ast\}}{2}\int_0^t\norm{(\mu,\theta)}_{L,\beta}^2\ds \\
        &\quad\leq E_{\mathrm{free}}(\phi_0,\psi_0) + \frac{1}{2\min\{1,m_\ast\}}\int_0^t \norm{(\bv,\bw)}_{\mathbfcal{L}^2}^2\ds
    \end{align*}
    for a.e. $t\in(0,\infty)$. Then, since by assumption $(\bv,\bw)\in L^2(0,\infty;\mathbfcal{L}^2)$ and $E_{\mathrm{free}}(\phi_0,\psi_0)$ is finite, we can use again \eqref{CCH:pp:<1} to find that
    \begin{align}\label{Est:PP:H1:Linfty}
        \norm{(\phi,\psi)}_{L^\infty(0,\infty;\mathcal{H}^1)}^2 \leq C\big(1 + \norm{(\bv,\bw)}_{L^2(0,\infty;\mathbfcal{L}^2)}^2\big)
    \end{align}
    as well as
    \begin{align}\label{Est:MT:LB:L2}
        \int_0^\infty \norm{(\mu,\theta)}_{L,\beta}^2\ds \leq C\big(1 + \norm{(\bv,\bw)}_{L^2(0,\infty;\mathbfcal{L}^2)}^2\big).
    \end{align}
    Next, a comparison argument in \eqref{WF:CCH:PP} yields the estimate
    \begin{align}\label{Est:DELT:PP:H1Lb':a.e.}
        \norm{(\delt\phi,\delt\psi)}_{(\mathcal{H}^1_{L,\beta})^\prime} \leq C\big(\norm{(\bv,\bw)}_{\mathbfcal{L}^2} + \norm{(\mu,\theta)}_{L,\beta}\big),
    \end{align}
    from which we deduce that
    \begin{align}\label{EST:DELT:PP:H1Lb':L2}
        \int_0^\infty \norm{(\delt\phi,\delt\psi)}_{(\mathcal{H}^1_{L,\beta})^\prime}^2\ds \leq C\big(1 + \norm{(\bv,\bw)}_{L^2(0,\infty;\mathbfcal{L}^2)}^2\big).
    \end{align}
    Then, as shown in \cite{Knopf2025}, we have the estimate
    \begin{align}\label{Est:Mean:mt}
        \abs{\mean{\mu}{\theta}} \leq C\big(1 + \norm{(\mu,\theta)}_{L,\beta}\big),
    \end{align}
    see also \eqref{Est:mt:mean:n}, which implies together with \eqref{Est:MT:LB:L2} that $(\mu,\theta)\in L^2_{\mathrm{uloc}}([0,\infty);\mathcal{H}^1)$. Now, noticing that
    \begin{alignat*}{2}
        -\Lap\phi + F_0^\prime(\phi) &= \mu^\ast &&\qquad\text{a.e.~in~}\Om, \\
        -\Lapg\psi + G_0^\prime(\psi) + \alpha\deln\phi &= \theta^\ast &&\qquad\text{a.e.~on~}\Ga, \\
        K\deln\phi &= \alpha\psi - \phi &&\qquad\text{a.e.~on~}\Ga,
    \end{alignat*}
    almost everywhere on $(0,\infty)$, where $(\mu^\ast, \theta^\ast) = (\mu - c_F\phi, \theta - c_G\psi)\in\mathcal{H}^1$, an application of \cite[Proposition~6.5]{Giorgini2025} together with \eqref{Est:PP:H1:Linfty} and \eqref{Est:Mean:mt} yields that
    \begin{align}\label{Est:pp:Pot:Lp:a.e.}
        \begin{split}
            \norm{(\phi,\psi)}_{\mathcal{W}^{2,p}} + \norm{(F_0^\prime(\phi),G_0^\prime(\psi))}_{\mathcal{L}^p} &\leq C_p\big(1 + \norm{(\mu,\theta)}_{\mathcal{H}^1}\big) \\
            &\leq C_p(1 + \norm{(\mu,\theta)}_{L,\beta}\big)
        \end{split}
    \end{align}
    almost everywhere on $(0,\infty)$ for all $2 \leq p < \infty$. The latter entails that $(\phi,\psi)\in L^2_{\mathrm{uloc}}([0,\infty);\mathcal{W}^{2,p})$ as well as $(F^\prime(\phi),G^\prime(\psi))\in L^2_{\mathrm{uloc}}([0,\infty);\mathcal{L}^p)$ together with the estimate \eqref{CCH:Est:L^p:Thm}. Lastly, recalling that $K\in(0,\infty)$, we apply \cite[Corollary~5.4]{Giorgini2025} and find that
    \begin{align*}
        \norm{(-\Lap\phi,-\Lapg\psi + \alpha\deln\phi)}_{\mathcal{L}^2}^2 \leq C\big(1 + \norm{(\mu - c_F\phi,\theta - c_G\psi)}_{\mathcal{H}^1}\big) \leq C\big(1 + \norm{(\mu,\theta)}_{\mathcal{H}^1}\big).
    \end{align*}
    Thus, by elliptic regularity theory for systems with bulk-surface coupling (see, e.g., \cite[Theorem~3.3]{Knopf2021}), we have
    \begin{align*}
        \sup_{t\geq 0}\int_t^{t+1}\norm{(\phi,\psi)}_{\mathcal{H}^2}^4\ds \leq \big(1 + \norm{(\bv,\bw)}_{L^2(0,\infty;\mathbfcal{L}^2)}^2\big),
    \end{align*}
    which provides $(\phi,\psi)\in L^4_{\mathrm{uloc}}([0,\infty);\mathcal{H}^2)$. 
\end{proof}

Our next result is concerned with the uniqueness of weak solutions to \eqref{EQ:CONV:SYSTEM}. In the non-convective case, this was first shown for the Cahn--Hilliard equation with dynamic boundary conditions in \cite{Stange2025}.
\begin{theorem}\label{Theorem:CCH:Uniqueness}
    Suppose that the mobility functions satisfy \ref{Assumption:Coefficients} and the potentials \ref{Assumption:Potential}. Let $K\in(0,\infty)$, $L\in[0,\infty]$, and let $(\phi_0^1,\psi_0^1), (\phi_0^2,\psi_0^2)\in\mathcal{H}^1$ be two pairs of initial data satisfying \eqref{cond:init} as well as
    \begin{align}\label{cond:init:uniqueness}
        \begin{dcases}
            \mean{\phi_0^1}{\psi_0^1} = \mean{\phi_0^2}{\psi_0^2}, &\textnormal{if } L\in[0,\infty), \\
            \meano{\phi_0^1} = \meano{\phi_0^2} \quad\textnormal{and}\quad \meang{\psi_0^1} = \meang{\psi_0^2}, &\textnormal{if } L = \infty,
        \end{dcases}
    \end{align}
    and let $(\bv_1,\bw_1), (\bv_2,\bw_2)\in L^2(0,T;\mathbf{L}^2_\Div(\Om)\times\mathbf{L}^2_\tau(\Ga))$ be prescribed velocity fields. Consider two weak solutions $(\phi_1,\psi_1,\mu_1,\theta_1)$ and $(\phi_2,\psi_2,\mu_2,\theta_2)$ originating from $(\phi_0^1,\psi_0^1)$ and $(\phi_0^2,\psi_0^2)$ with velocity fields $(\bv_1,\bw_1)$ and $(\bv_2,\bw_2)$, respectively. Then, for any $T > 0$, there exists a positive constant $C$ such that
    \begin{align}
        \begin{split}\label{Est:ContinuousDependenceConvective}
            &\norm{(\phi_1(t) - \phi_2(t),\psi_1(t) - \psi_2(t))}_{(\mathcal{H}^1_{L,\beta})^\prime}^2 \\
            &\quad\leq \norm{(\phi_0^1 - \phi_0^2,\psi_0^1 - \psi_0^2)}_{(\mathcal{H}^1_{L,\beta})^\prime}^2\exp\Big(\int_0^t P(\tau)\dtau\Big) \\
            &\qquad + \int_0^t \norm{(\mathbf{v}_1(s) - \mathbf{v}_2(s),\mathbf{w}_1(s) - \mathbf{w}_2(s))}_{\mathbf{L}^{1+\omega}(\Om)\times\mathbf{L}^1(\Ga)}^2\exp\Big(\int_0^t P(\tau)\dtau\Big)\ds
        \end{split}
    \end{align}
    for almost all $t\in[0,T]$ and any $\omega > 0$, where
    \begin{align*}
        P = C\big(1 + \norm{(\mathbf{v}_2,\mathbf{w}_2)}_{\mathbfcal{L}^2}^2 + \norm{(\mu_2,\theta_2)}_{L,\beta}^2 + \norm{(\delt\phi_1,\delt\psi_1)}_{(\mathcal{H}^1_{L,\beta})^\prime}^2 + \norm{(\phi_1,\psi_1)}_{\mathcal{H}^2}^4\big),
    \end{align*}
    and the constant $C > 0$ depends only on $\Om$, the initial data and the parameters of the system.
\end{theorem}

\begin{proof}
    Let $(\phi_0^1,\psi_0^1)$ and $(\phi_0^2,\psi_0^2)$ be two pairs of admissible initial data satisfying \eqref{cond:init} as well as \eqref{cond:init:uniqueness}.     Additionally, let $(\bv_1,\bw_1), (\bv_2,\bw_2)\in L^2(0,T;\mathbf{L}^2_\Div(\Om)\times\mathbf{L}^2_\tau(\Ga))$. Consider the corresponding weak solutions $(\phi_1,\psi_1,\mu_1,\theta_1)$ and $(\phi_2,\psi_2,\mu_2,\theta_2)$. 
    Then, defining $(\phi,\psi) = (\phi_1 - \phi_2, \psi_1 - \psi_2)$, we have
    \begin{align}\label{EQ:DELT:PHIPSI}
        \begin{split}
            &\bigang{(\delt\phi,\delt\psi)}{(\zeta,\xi)}_{\mathcal{H}^1_{L,\beta}} \\
            &\quad = \intO \big(\phi_1\bv + \phi\bv_2\big)\cdot\Grad\zeta\dx + \intG \big(\psi_1\bw + \psi\bw_2\big)\cdot\Gradg\xi\dG \\
            &\qquad - \intO m_\Om(\phi_1)\Grad(\mu_1 - \mu_2)\cdot\Grad\zeta\dx - \intG m_\Ga(\psi_1)\Gradg(\theta_1 - \theta_2)\cdot\Gradg\xi\dG \\
            &\qquad - \chi(L) \intG \big(\beta(\theta_1 - \theta_2) - (\mu_1 - \mu_2)\big)(\beta\xi - \zeta)\dG \\
            &\qquad - \intO \big(m_\Om(\phi_1) - m_\Om(\phi_2)\big)\Grad\mu_2\cdot\Grad\zeta\dx - \intG \big(m_\Ga(\psi_1) - m_\Ga(\psi_2)\big)\Gradg\theta_2\cdot\Gradg\xi\dG
        \end{split}
    \end{align}
    a.e. on $(0,\infty)$ for all $(\zeta,\xi)\in\mathcal{H}^1_{L,\beta}$, as well as
    \begin{subequations}\label{EQ:UNIQUENESS}
        \begin{alignat}{2}
        	-\Lap\phi + F^\prime(\phi_1) - F^\prime(\phi_2) &= \mu_1 - \mu_2 &&\qquad\text{a.e.~on~}Q, \label{EQ:PHI}\\
        	-\Lapg\psi + G^\prime(\psi_1) - G^\prime(\psi_2) + \alpha\deln\phi &= \theta_1 - \theta_2 &&\qquad\text{a.e.~on~}\Sigma, \label{EQ:PSI} \\
            K\deln\phi &= \alpha\psi - \phi &&\qquad\text{a.e.~on~}\Sigma. \label{EQ:BC:PHI}
        \end{alignat}
    \end{subequations}
    Now, we multiply \eqref{EQ:PHI} with $\phi$ and \eqref{EQ:PSI} with $\psi$, integrate over $\Om$ and $\Ga$, respectively, and perform integration by parts. Adding the resulting equations leads to
    \begin{align}\label{pre:comp}
        &\norm{(\phi,\psi)}_{K,\alpha}^2 + \intO \big(F_0^\prime(\phi_1) - F_1^\prime(\phi_2)\big)\phi\dx + \intG \big(G_0^\prime(\psi_1) - G_1^\prime(\psi_2)\big)\psi\dG \nonumber \\
        &\qquad - \intO \big(\mu_1 - \mu_2)\phi\dx - \intG \big(\theta_1 - \theta_2\big)\psi\dG \\
        &\quad = c_F\intO \phi^2\dx + c_G\intG \psi^2\dG \nonumber
    \end{align}
    almost everywhere on $(0,\infty)$. Next, we want to rewrite the integrals involving the chemical potentials in terms of the solution operator $\mathcal{S}_{L,\beta}[\phi_1,\psi_1](\delt\phi,\delt\psi)$. To this end, to make the following computations more readable, we use the abbreviation $\mathcal{S}_1(\delt\phi,\delt\psi) = \mathcal{S}_{L,\beta}[\phi_1,\psi_1](\delt\phi,\delt\psi)$ and find
    \begin{align}\label{mu_1-mu_2}
        \begin{split}
            &-\intO (\mu_1 - \mu_2)\phi\dx - \intG (\theta_1 - \theta_2)\psi\dG \\
            &\quad = \big((\phi,\psi),\mathcal{S}_1(\delt\phi,\delt\psi)\big)_{\mathcal{L}^2} - \intO \big(m_\Om(\phi_1) - m_\Om(\phi_2)\big)\Grad\mu_2\cdot\Grad\mathcal{S}_1^\Om(\phi,\psi)\dx \\
            &\qquad - \intG (m_\Ga(\psi_1) - m_\Ga(\psi_1)\big)\Gradg\theta_2\cdot\Gradg\mathcal{S}_1^\Ga(\phi,\psi)\dG + \intO \big(\phi_1\mathbf{v} + \phi\mathbf{v}_2\big)\cdot\Grad\mathcal{S}_1^\Om(\phi,\psi)\dx \\
            &\qquad + \intG \big(\psi_1\mathbf{w} + \psi\mathbf{w}_2\big)\cdot\Gradg\mathcal{S}_1^\Ga(\phi,\psi)\dG
        \end{split}
    \end{align}
    a.e. on $(0,\infty)$.
    Plugging the identity \eqref{mu_1-mu_2} back into \eqref{pre:comp} yields
    \begin{align*}
    	&\norm{(\phi,\psi)}_{K,\alpha}^2 + \intO \big(F_0^\prime(\phi_1) - F_0^\prime(\phi_2)\big)\phi\dx + \intG \big(G_0^\prime(\psi_1) - G_0^\prime(\psi_2)\big)\psi\dG \\
    	&\qquad  + \intO \mathcal{S}_1(\delt\phi,\delt\psi)\phi\dx + \intG \mathcal{S}_1(\delt\phi,\delt\psi)\psi\dG \\
        &\qquad - \intO \big(m_\Om(\phi_1) - m_\Om(\phi_2)\big)\Grad\mu_2\cdot\Grad\mathcal{S}^\Om_1(\phi,\psi)\dx - \intG \big(m_\Ga(\psi_1) - m_\Ga(\psi_2)\big)\Gradg\theta_2\cdot\Gradg\mathcal{S}^\Ga_1(\phi,\psi)\dG \\
        &\qquad + \intO \big(\phi_1\mathbf{v} + \phi\mathbf{v}_2\big)\cdot\Grad\mathcal{S}_1^\Om(\phi,\psi)\dx + \intG \big(\psi_1\mathbf{w} + \psi\mathbf{w}_2\big)\cdot\Gradg\mathcal{S}_1^\Ga(\phi,\psi)\dG \\
    	&\quad = c_F\intO \phi^2 \dx + c_G\intG \psi^2\dG.
    \end{align*}
    Then, in view of the chain rule (cf. \cite[Eqn.~(5.18)]{Stange2025})
    \begin{align}\label{ChainRuleFormula}
	   \begin{split}
	   &\intO \mathcal{S}_1^\Om(\delt\phi,\delt\psi)\phi\dx + \intG \mathcal{S}_1^\Ga(\delt\phi,\delt\psi)\psi\dG  \\
	   &\quad = \ddt\frac12 \norm{(\phi,\psi)}_{L,\beta,[\phi_1,\psi_1],\ast}^2 \\
	   &\qquad + \frac12 \Big(\mathcal{S}_1(\delt\phi_1,\delt\psi_1),\big(m_\Om^\prime(\phi_1)\abs{\Grad\mathcal{S}_1^\Om(\phi,\psi)}^2,m_\Ga^\prime(\psi_1)\abs{\Gradg\mathcal{S}_1^\Ga(\phi,\psi)}^2\big)\Big)_{L,\beta},
	   \end{split}
    \end{align}
    we obtain
    \begin{align}\label{DiffIneq}
    		\ddt\frac12\norm{(\phi,\psi)}_{L,\beta,[\phi_1,\psi_1],\ast}^2 + \norm{(\phi,\psi)}_{K,\alpha}^2 \leq c_F\norm{\phi}_{L^2(\Om)}^2 + c_G\norm{\psi}_{L^2(\Ga)}^2 + I_1 + I_2 + I_3,
    \end{align}
    where
    \begin{align*}
    	I_1 &=  -\frac12 \big(\mathcal{S}_1(\delt\phi_1,\delt\psi_1),(m_\Om^\prime(\phi_1)\abs{\Grad\mathcal{S}_1^\Om(\phi,\psi)}^2,m_\Ga^\prime(\psi_1)\abs{\Gradg\mathcal{S}_1^\Ga(\phi,\psi)}^2)\big)_{L,\beta}, \\
        I_2 &=  \intO \big(m_\Om(\phi_1) - m_\Om(\phi_2)\big)\Grad\mu_2\cdot\Grad\mathcal{S}^\Om_1(\phi,\psi)\dx + \intG \big(m_\Ga(\psi_1) - m_\Ga(\psi_2)\big)\Gradg\theta_2\cdot\Gradg\mathcal{S}^\Ga_1(\phi,\psi)\dG,
    \end{align*}
    and
    \begin{align*}
    	I_3 = \intO \big(\phi_1\mathbf{v} + \phi\mathbf{v}_2\big)\cdot\Grad\mathcal{S}_1^\Om(\phi,\psi)\dx + \intG \big(\psi_1\mathbf{w} + \psi\mathbf{w}_2\big)\cdot\Gradg\mathcal{S}_1^\Ga(\phi,\psi)\dG.
    \end{align*}
    The terms $I_1$ and $I_2$ have already been estimated in \cite{Stange2025}. Namely, it holds that
    \begin{align*}
        &\abs{I_1 + I_2} \\
        &\quad \leq \frac18\norm{(\phi,\psi)}_{K,\alpha}^2 + C\big(\norm{(\mu_2,\theta_2)}_{L,\beta}^2 + \norm{(\delt\phi_1,\delt\psi_1)}_{(\mathcal{H}^1_{L,\beta})^\prime}^2 + \norm{(\phi_1,\psi_1)}_{\mathcal{H}^2}^4\big)\norm{\mathcal{S}_1(\phi,\psi)}_{L,\beta}^2, \nonumber
    \end{align*}
    see \cite[Eqns.~(5.62)--(5.63)]{Stange2025}.
    Furthermore, as \eqref{Est:fg:L^2:K} implies that
    \begin{align}\label{Est:PhiPsi:L2}
    	c_F\norm{\phi}_{L^2(\Om)}^2 + c_G\norm{\psi}_{L^2(\Ga)}^2\leq C\norm{\mathcal{S}_1(\phi,\psi)}_{L,\beta}\norm{(\phi,\psi)}_{K,\alpha} \leq \frac18\norm{(\phi,\psi)}_{K,\alpha}^2 + C\norm{\mathcal{S}_1(\phi,\psi)}_{L,\beta}^2,
    \end{align}
    we only have to estimate the term $I_3$, which contains the convective contributions. To this end, we first observe that
    \begin{align}\label{I3:1}
        \abs{I_3} &= \Bigabs{\intO \big(\phi_1\mathbf{v} + \phi\mathbf{v}_2\big)\cdot\Grad\mathcal{S}_1^\Om(\phi,\psi)\dx + \intG \big(\psi_1\mathbf{w} + \psi\mathbf{w}_2\big)\cdot\Gradg\mathcal{S}_1^\Ga(\phi,\psi)\dG} \nonumber \\
        &\leq \Bigabs{\intO \Grad\phi_1\cdot\bv\,\mathcal{S}_1^\Om(\phi,\psi)\dx + \intG \Gradg\psi_1\cdot\bw\,\mathcal{S}_1^\Ga(\phi,\psi)\dG} \\
        &\quad + \Bigabs{\intO \phi\bv_2\cdot\Grad\mathcal{S}_1^\Om(\phi,\psi)\dx + \intG \psi\bw_2\cdot\Gradg\mathcal{S}_1^\Ga(\phi,\psi)\dG}. \nonumber
    \end{align}
    For the first term on the right-hand side of \eqref{I3:1}, we use Hölder's, Sobolev's and the bulk-surface Poincar\'{e} inequality and deduce that
    \begin{align*}
        &\Bigabs{\intO \Grad\phi_1\cdot\bv\,\mathcal{S}_1^\Om(\phi,\psi)\dx + \intG \Gradg\psi_1\cdot\bw\,\mathcal{S}_1^\Ga(\phi,\psi)\dG} \\
        &\leq \norm{\Grad\phi_1}_{\mathbf{L}^{\frac{2(1+\omega)}{\omega}}(\Om)}\norm{\bv}_{\mathbf{L}^{1+\omega}(\Om)}\norm{\mathcal{S}_1^\Om(\phi,\psi)}_{L^{\frac{2(1+\omega)}{\omega}}(\Om)} + \norm{\Gradg\psi_1}_{\mathbf{L}^\infty(\Ga)}\norm{\bw}_{\mathbf{L}^1(\Ga)}\norm{\mathcal{S}_1^\Ga(\phi,\psi)}_{L^\infty(\Ga)} \\
        &\leq C\norm{(\phi_1,\psi_1)}_{\mathcal{H}^2}\norm{(\bv,\bw)}_{\mathbf{L}^{1+\omega}(\Om)\times\mathbf{L}^1(\Ga)}\norm{\mathcal{S}_1(\phi,\psi)}_{L,\beta} \\
        &\leq \frac12\norm{(\bv,\bw)}_{\mathbf{L}^{1+\omega}(\Om)\times\mathbf{L}^1(\Ga)}^2 + C\norm{(\phi_1,\psi_1)}_{\mathcal{H}^2}^2\norm{\mathcal{S}_1(\phi,\psi)}_{L,\beta}^2
    \end{align*}
    for any $\omega > 0$.
    For the second term on the right-hand side of \eqref{I3:1}, we use a similar argument to \cite[Lemma~5.1]{Giorgini2018}, see also \cite{Giorgini2021}. However, since in our case we have non-degenerate mobility functions, additional terms are coming from \eqref{Est:Sol:G:H^2} that have to be handled suitably. To be precise, exploiting \eqref{Est:Poincare}, \eqref{InterpolEst:L^4}, \eqref{Est:fg:L^2:K} and \eqref{Est:Sol:G:H^2},  we find that
    \begin{align*}
        &\Bigabs{\intO \phi\bv_2\cdot\Grad\mathcal{S}_1^\Om(\phi,\psi)\dx + \intG \psi\bw_2\cdot\Gradg\mathcal{S}_1^\Ga(\phi,\psi)\dG} \\
        &\leq \norm{(\bv_2,\bw_2)}_{\mathbfcal{L}^2}\norm{(\phi,\psi)}_{\mathcal{L}^4}\norm{(\Grad\mathcal{S}_1^\Om(\phi,\psi),\Gradg\mathcal{S}_1^\Ga(\phi,\psi))}_{\mathbfcal{L}^4} \\
        &\leq C\norm{(\bv_2,\bw_2)}_{\mathbfcal{L}^2}\norm{(\phi,\psi)}_{\mathcal{L}^2}^{\frac12}\norm{(\phi,\psi)}_{\mathcal{H}^1}^{\frac12}\norm{(\Grad\mathcal{S}_1^\Om(\phi,\psi),\Gradg\mathcal{S}_1^\Ga(\phi,\psi))}_{\mathbfcal{L}^2}^{\frac12}\norm{(\Grad\mathcal{S}_1^\Om(\phi,\psi),\Gradg\mathcal{S}_1^\Ga(\phi,\psi))}_{\mathbfcal{H}^1}^{\frac12} \\
        &\leq C\norm{(\bv_2,\bw_2)}_{\mathbfcal{L}^2}\norm{(\phi,\psi)}_{\mathcal{L}^2}^{\frac12}\norm{(\phi,\psi)}_{\mathcal{H}^1}^{\frac12}\norm{\mathcal{S}_1(\phi,\psi)}_{L,\beta}^{\frac12} \\
        &\qquad\times \Big(\norm{(\Grad\phi_1,\Gradg\psi_1)}_{\mathbfcal{L}^2}^{\frac12}\norm{(\phi_1,\psi_1)}_{\mathcal{H}^2}^{\frac12}\norm{\mathcal{S}_1(\phi,\psi)}_{L,\beta}^{\frac12} + \norm{(\phi,\psi)}_{\mathcal{L}^2}^{\frac12}\Big) \\
        &\leq C\norm{(\bv_2,\bw_2)}_{\mathbfcal{L}^2}\Big(\norm{(\phi,\psi)}_{\mathcal{L}^2}^{\frac12}\norm{(\phi,\psi)}_{\mathcal{H}^1}^{\frac12}\norm{(\Grad\phi_1,\Gradg\psi_1)}_{\mathbfcal{L}^2}^{\frac12}\norm{(\phi_1,\psi_1)}_{\mathcal{H}^2}^{\frac12}\norm{\mathcal{S}_1(\phi,\psi)}_{L,\beta} \\
        &\qquad + \norm{(\phi,\psi)}_{\mathcal{L}^2}\norm{(\phi,\psi)}_{\mathcal{H}^1}^{\frac12}\norm{\mathcal{S}_1(\phi,\psi)}_{L,\beta}^{\frac12}\Big) \\
        &\leq C\norm{(\bv_2,\bw_2)}_{\mathbfcal{L}^2}\Big(\norm{(\phi,\psi)}_{K,\alpha}\norm{(\phi_1,\psi_1)}_{\mathcal{H}^2}\norm{\mathcal{S}_1(\phi,\psi)}_{L,\beta} + \norm{(\phi,\psi)}_{K,\alpha}\norm{\mathcal{S}_1(\phi,\psi)}_{L,\beta}\Big) \\
        &\leq \frac18\norm{(\phi,\psi)}_{K,\alpha}^2 + C\big(1 + \norm{(\bv_2,\bw_2)}_{\mathbfcal{L}^2}^2 + \norm{(\phi_1,\psi_1)}_{\mathcal{H}^2}^4\big)\norm{\mathcal{S}_1(\phi,\psi)}_{L,\beta}^2. 
    \end{align*}
    Consequently, we find the differential inequality
    \begin{align}\label{DiffIneq:Final}
        \ddt\frac12\norm{(\phi,\psi)}_{L,\beta,[\phi_1,\psi_1],\ast}^2 + \frac12\norm{(\phi,\psi)}_{K,\alpha}^2 \leq \frac12\norm{(\mathbf{v},\mathbf{w})}_{\mathbf{L}^{1+\omega}(\Om)\times\mathbf{L}^1(\Ga)}^2 + P(t)\norm{(\phi,\psi)}_{L,\beta,[\phi_1,\psi_1],\ast}^2,
    \end{align}
    where
    \begin{align*}
        P(\cdot) = C\big(1 + \norm{(\mathbf{v}_2,\mathbf{w}_2)}_{\mathbfcal{L}^2}^2 + \norm{(\mu_2,\theta_2)}_{L,\beta}^2 + \norm{(\delt\phi_1,\delt\psi_1)}_{(\mathcal{H}^1_{L,\beta})^\prime}^2 + \norm{(\phi_1,\psi_1)}_{\mathcal{H}^2}^4\big)\in L^1(0,T).
    \end{align*}
    An application of Gronwall's lemma readily yields the desired conclusion \eqref{Est:ContinuousDependenceConvective}, and thus the uniqueness of weak solutions to \eqref{EQ:CONV:SYSTEM}. This finishes the proof.
\end{proof}

Our next result is concerned with the existence of strong solutions to \eqref{EQ:CONV:SYSTEM}.
\begin{theorem}\label{Theorem:CCH:Strong}
    We suppose that the mobility functions satisfy \ref{Assumption:Coefficients} and that the potentials satisfy \ref{Assumption:Potential}. Let $K\in(0,\infty)$ and $L\in[0,\infty]$, and let the initial condition $(\phi_0,\psi_0)\in\mathcal{H}^1$ satisfy \eqref{cond:init}, as well as the following compatibility condition:
    \begin{enumerate}[label=\textnormal{\bfseries(C)},topsep=0ex,leftmargin=*]
    \item \label{cond:MT:0:app} There exists $\scp{\mu_0}{\theta_0}\in\mathcal{H}^1_{L,\beta}$ such that for all $\scp{\eta}{\vartheta}\in\mathcal{H}^1$ it holds
    \begin{align*}
        \begin{aligned}
            &\intO\mu_0\eta\dx + \intG\theta_0\vartheta\dG 
            \\
            &= \intO\Grad\phi_0\cdot\Grad\eta + F^\prime(\phi_0)\eta\dx + \intG\Gradg\psi_0\cdot\Gradg\vartheta + G^\prime(\psi_0)\vartheta\dG 
            \\
            &\quad + \chi(K)\intG(\alpha\psi_0 - \phi_0)(\alpha\vartheta - \eta)\dG.
        \end{aligned}
    \end{align*}
\end{enumerate}
    For the velocity fields $(\bv,\bw)$, we assume one of the following conditions:
    \begin{enumerate}[label=\textnormal{(\roman*)}, topsep=0ex]
        \item \label{Cond:Conv:i} $(\bv,\bw)\in H^1_{\mathrm{uloc}}(0,\infty;\mathbf{L}^{1+\omega}(\Om)\times\mathbf{L}^1(\Ga)) \cap L^2(0,\infty;\mathbfcal{L}^2)\cap L^\infty(0,\infty;\mathbfcal{L}^2_\Div)$ \ for some $\omega > 0$,
        \item \label{Cond:Conv:ii} $(\bv,\bw)\in L^2(0,\infty;\mathbfcal{H}^1)\cap L^\infty(0,\infty;\mathbfcal{L}^2_\Div)$.
    \end{enumerate}
    Then, the corresponding unique weak solution $(\phi,\psi,\mu,\theta)$ to \eqref{EQ:CONV:SYSTEM} satisfies
    \begin{align*}
        &(\phi,\psi)\in L^\infty(0,\infty;\mathcal{W}^{2,p}), \quad (\delt\phi,\delt\psi)\in L^\infty(0,\infty;(\mathcal{H}^1_{L,\beta})^\prime)\cap L^2_{\mathrm{uloc}}([0,\infty);\mathcal{H}^1), \\
        &(\mu,\theta)\in L^\infty(0,\infty;\mathcal{H}^1_{L,\beta})\cap L^2_{\mathrm{uloc}}([0,\infty);\mathcal{H}^3)\cap H^1_{\mathrm{uloc}}(0,\infty;(\mathcal{H}^1_{L,\beta})^\prime), \\
        &(F^\prime(\phi),G^\prime(\psi)), \ (F^{\prime\prime}(\phi), G^{\prime\prime}(\psi))\in L^\infty(0,\infty;\mathcal{L}^p)
    \end{align*}
    for any $2 \leq p < \infty$. Moreover, the equations \eqref{EQ:CONV:SYSTEM:1}-\eqref{EQ:CONV:SYSTEM:2} are satisfied almost everywhere on $\Om\times(0,\infty)$, while \eqref{EQ:CONV:SYSTEM:3}-\eqref{EQ:CONV:SYSTEM:4} and the boundary conditions \eqref{EQ:CONV:SYSTEM:5}-\eqref{EQ:CONV:SYSTEM:6} are satisfied almost everywhere on $\Ga\times(0,\infty)$.
\end{theorem}

\begin{remark}\label{REM:REG:WEAK}
    Concerning the assumptions on the mobility functions and the velocity fields in Theorem~\ref{Theorem:CCH:Strong}, we remark the following. For $m_\Om, m_\Ga\in C^1([-1,1])$ and Leray-velocity fields, i.e., 
    \begin{align*}
        (\bv,\bw)\in L^2(0,\infty;\mathbfcal{H}^1)\cap L^\infty(0,\infty;\mathbfcal{L}^2_\Div),
    \end{align*}
    one can also show the existence of a strong solution to \eqref{EQ:CONV:SYSTEM}. This can be done by combining the approaches of \cite[Theorem~3.6]{Giorgini2025} and \cite[Theorem~3.4]{Stange2025} (note also \cite[Remark~3.7]{Giorgini2025}). However, as the mobility functions only lie in $C^1([-1,1])$, Theorem~\ref{Theorem:CCH:Uniqueness} about the uniqueness is not applicable. Nonetheless, if $L\in(0,\infty]$, the uniqueness of strong solutions to \eqref{EQ:CONV:SYSTEM} can be achieved. The case $L = 0$ is only admissible if we additionally take $K = 0$, provided that $\alpha\neq 0, \beta\neq 0$, and the following compatibility condition on the potentials:
    \begin{align*}
        F(\alpha s) = \alpha\beta G(s) \qquad\text{for all~}s\in[-1,1].
    \end{align*}
    In this case, we have $(F^\prime(\phi),G^\prime(\psi))\in\mathcal{H}^1_{0,\beta}$, entailing that it is an admissible test function in \eqref{WF:CCH:PP}. We do not present the proof of this uniqueness result at this point, as the computations are the same as in the proof of Theorem~\ref{Theorem:Uniqueness}, see, in particular, \eqref{Uniq:ddt:Ka} and the corresponding estimates \eqref{Uniq:Est:W15:1}-\eqref{W1819}.
\end{remark}

\begin{proof}
    \textbf{Case }\ref{Cond:Conv:i}: In this case, we intend to apply the continuous dependence estimate \eqref{DiffIneq:Final}.
    Namely, for any $h\in(0,1)$, we apply \eqref{DiffIneq:Final} to $(\phi_1,\psi_1) = (\phi(\cdot + h),\psi(\cdot + h))$ and $(\phi_2,\psi_2) = (\phi,\psi)$ with given vector fields $(\bv_1,\bw_1) = (\bv(\cdot + h),\bw(\cdot + h))$ and $(\bv_2,\bw_2) = (\bv,\bw)$. After dividing the resulting inequality by $h^2$, we arrive at
    \begin{align}\label{Conv:HighReg:Gronwall:Step1}
        \begin{split}
        	&\ddt \frac12\norm{(\delth\phi(t),\delth\psi(t))}_{L,\beta,[\phi(t+h),\psi(t+h)],\ast}^2 + \frac12\norm{(\delth\phi(t),\delth\psi(t))}_{K,\alpha}^2 \\
            &\qquad \leq \frac12\norm{(\delth\mathbf{v}(t),\delth\mathbf{w}(t))}_{\mathbf{L}^{1+\omega}(\Om)\times\mathbf{L}^1(\Ga)}^2 + P_h(t)\norm{(\delth\phi(t),\delth\psi(t))}_{L,\beta,[\phi(t+h),\psi(t+h)],\ast}^2,
        \end{split}
    \end{align}
    where
    \begin{align*}
        P_h(\cdot) &= C\big(1 + \norm{(\mathbf{v},\mathbf{w})}_{\mathbfcal{L}^2}^2 + \norm{(\mu,\theta)}_{L,\beta}^2 + \norm{(\delt\phi(\cdot + h),\delt\psi(\cdot + h))}_{(\mathcal{H}^1_{L,\beta})^\prime}^2 \\
        &\qquad + \norm{(\phi(\cdot + h),\psi(\cdot + h))}_{\mathcal{H}^2}^4\big)\in L^1(0,T) \qquad\text{for any~} T > 0.
    \end{align*}
    Before applying the Gronwall lemma to \eqref{Conv:HighReg:Gronwall:Step1}, we aim to control the initial data. To this end, we can argue similarly as for \eqref{DiffIneq}. Now, instead, we obtain
    \begin{align}\label{DiffIneq:InitialData}
        &\ddt\frac12\norm{(\phi - \phi_0,\psi - \psi_0)}_{L,\beta,[\phi,\psi],\ast} + \norm{(\phi - \phi_0,\psi - \psi_0)}_{K,\alpha}^2 \nonumber \\
        &\quad\leq c_F\norm{\phi - \phi_0}_{L^2(\Om)}^2 + c_G\norm{\psi - \psi_0}_{L^2(\Ga)}^2 + \big((\mu_0,\theta_0),\mathcal{S}[\phi,\psi](\phi - \phi_0,\psi - \psi_0)\big)_{L,\beta,[\phi,\psi]} \nonumber \\
        &\qquad + \intO \phi\bv\cdot\Grad\mathcal{S}_{L,\beta}^\Om[\phi,\psi](\phi - \phi_0, \psi - \psi_0)\dx + \intG \psi\bw\cdot\Gradg\mathcal{S}_{L,\beta}^\Ga[\phi,\psi](\phi - \phi_0, \psi - \psi_0)\dG \\ 
        &\qquad - \frac12\Big(\mathcal{S}_{L,\beta}[\phi,\psi](\delt\phi,\delt\psi), \nonumber \\
        &\qquad\qquad\quad \big(m_\Om^\prime(\phi)\abs{\Grad\mathcal{S}_{L,\beta}^\Om[\phi,\psi](\phi - \phi_0,\psi - \psi_0)}^2, m_\Ga^\prime(\psi)\abs{\Gradg\mathcal{S}_{L,\beta}^\Ga[\phi,\psi](\phi - \phi_0, \psi - \psi_0)}^2\big)\Big)_{L,\beta}. \nonumber
    \end{align}
    The first two terms on the right-hand side of \eqref{DiffIneq:InitialData} can be estimated analogously to \eqref{Est:PhiPsi:L2}, namely,
    \begin{align*}
        c_F\norm{\phi - \phi_0}_{L^2(\Om)}^2 + c_G\norm{\psi - \psi_0}_{L^2(\Ga)}^2 \leq \frac14\norm{(\phi - \phi_0, \psi - \psi_0)}_{K,\alpha}^2 + C\norm{(\phi - \phi_0, \psi - \psi_0)}_{L,\beta[\phi,\psi],\ast}^2,
    \end{align*}
    while the last term has already been estimated in \cite[Eqn.~(5.62)]{Stange2025} as
    \begin{align*}
        &\bigg\vert - \frac12\Big(\mathcal{S}_{L,\beta}[\phi,\psi](\delt\phi,\delt\psi), \nonumber \\
        &\qquad\quad\big(m_\Om^\prime(\phi)\abs{\Grad\mathcal{S}_{L,\beta}^\Om[\phi,\psi](\phi - \phi_0,\psi - \psi_0)}^2, m_\Ga^\prime(\psi)\abs{\Gradg\mathcal{S}_{L,\beta}^\Ga[\phi,\psi](\phi - \phi_0, \psi - \psi_0)}^2\big)\Big)_{L,\beta}\bigg\vert \\
        &\quad \leq \frac14\norm{(\phi - \phi_0, \psi - \psi_0)}_{K,\alpha}^2 + C\Big(\norm{(\delt\phi,\delt\psi)}_{(\mathcal{H}^1_{L,\beta})^\prime}^2 + \norm{(\phi,\psi)}_{\mathcal{H}^2}^4\Big)\norm{(\phi - \phi_0, \psi - \psi_0)}_{L,\beta,[\phi,\psi],\ast}^2.
    \end{align*}
    Then, for the third term on the right-hand side of \eqref{DiffIneq:InitialData} we simply use Hölder's inequality and the boundedness of the mobility functions to find that
    \begin{align*}
        \big((\mu_0,\theta_0),\mathcal{S}_{L,\beta}[\phi,\psi](\phi - \phi_0,\psi - \psi_0)\big)_{L,\beta,[\phi,\psi]} \leq C\norm{(\mu_0,\theta_0)}_{L,\beta}\norm{(\phi - \phi_0, \psi - \psi_0)}_{L,\beta,[\phi,\psi],\ast},
    \end{align*}
    whereas for the convective terms, we have
    \begin{align*}
        &\Bigabs{\intO \phi\bv\cdot\Grad\mathcal{S}_{L,\beta}^\Om[\phi,\psi](\phi - \phi_0, \psi - \psi_0)\dx + \intG \psi\bw\cdot\Gradg\mathcal{S}_{L,\beta}^\Ga[\phi,\psi](\phi - \phi_0, \psi - \psi_0)\dG } \\
        &\quad\leq C\norm{(\bv,\bw)}_{\mathbfcal{L}^2}\norm{(\phi - \phi_0, \psi - \psi_0)}_{L,\beta,[\phi,\psi],\ast}
    \end{align*}
    in view of \eqref{CCH:pp:<1}. Collecting these estimates, we have the following differential inequality
    \begin{align}\label{DiffIneq:InitialData:Final}
        &\ddt\frac12\norm{(\phi - \phi_0, \psi - \psi_0)}_{L,\beta[\phi,\psi],\ast}^2 \\
        &\quad \leq C\Big(1 + \norm{(\mu_0,\theta_0)}_{L,\beta} + \norm{(\bv,\bw)}_{\mathbfcal{L}^2} + \big(\norm{(\delt\phi,\delt\psi)}_{(\mathcal{H}^1_{L,\beta})^\prime}^2 + \norm{(\phi,\psi)}_{\mathcal{H}^2}^4\big)\norm{(\phi - \phi_0, \psi - \psi_0)}_{L,\beta,[\phi,\psi],\ast}\Big) \nonumber \\
        &\qquad\times \norm{(\phi - \phi_0, \psi - \psi_0)}_{L,\beta,[\phi,\psi],\ast}. \nonumber
    \end{align}
    An application of a quadratic variant of the Gronwall lemma (see, e.g., \cite[Lemma~A.5]{Brezis1973}) yields
    \begin{align*}
        &\norm{(\phi - \phi_0, \psi - \psi_0)}_{L,\beta,[\phi,\psi],\ast} \\
        &\quad\leq C\big(1 + \norm{(\mu_0,\theta_0)}_{L,\beta} + \norm{(\bv,\bw)}_{L^\infty(0,\infty;\mathbfcal{L}^2)}\big)t \\
        &\qquad + C\int_0^t \big(\norm{(\delt\phi,\delt\psi)}_{(\mathcal{H}^1_{L,\beta})^\prime}^2 + \norm{(\phi,\psi)}_{\mathcal{H}^2}^4\big)\norm{(\phi - \phi_0, \psi - \psi_0)}_{L,\beta,[\phi,\psi],\ast}\ds.
    \end{align*}
    Then, we take $t = h$ and divide the inequality by $h$ to obtain
    \begin{align*}
        &\norm{(\delth\phi(0), \delth\psi(0))}_{L,\beta,[\phi,\psi],\ast} \\
        &\quad\leq C\big(1 + \norm{(\mu_0,\theta_0)}_{L,\beta} + \norm{(\bv,\bw)}_{L^\infty(0,\infty;\mathbfcal{L}^2)}\big) \\
        &\qquad + C\int_0^h \big(\norm{(\delt\phi,\delt\psi)}_{(\mathcal{H}^1_{L,\beta})^\prime}^2 + \norm{(\phi,\psi)}_{\mathcal{H}^2}^4\big)\norm{(\delth\phi(0), \delth\psi(0))}_{L,\beta,[\phi,\psi],\ast}\ds.
    \end{align*}
    Now, we can apply the integrated version of the Gronwall lemma to this inequality and find that
    \begin{align}\label{HighReg:Control:Init}
        &\norm{(\delth\phi(0), \delth\psi(0))}_{L,\beta,[\phi(\cdot+h),\psi(\cdot+h)],\ast} \nonumber \\
        &\quad\leq  C\big(1 + \norm{(\mu_0,\theta_0)}_{L,\beta} + \norm{(\bv,\bw)}_{L^\infty(0,\infty;\mathbfcal{L}^2)}\big) \\
        &\qquad\times\Bigg(1 + \int_0^h \big(\norm{(\delt\phi,\delt\psi)}_{(\mathcal{H}^1_{L,\beta})^\prime}^2 + \norm{(\phi,\psi)}_{\mathcal{H}^2}^4\big)\exp\Big(\int_s^h \big(\norm{(\delt\phi,\delt\psi)}_{(\mathcal{H}^1_{L,\beta})^\prime}^2 + \norm{(\phi,\psi)}_{\mathcal{H}^2}^4\big)\dtau\Big)\ds\Bigg). \nonumber
    \end{align}
    Now that we have a suitable control of the initial data, we use again the Gronwall lemma for the differential inequality \eqref{Conv:HighReg:Gronwall:Step1} and get
    \begin{align}\label{Conv:HighReg:Gronwall:Result}
        \begin{split}
            \norm{(\delth\phi(t),\delth\psi(t))}_{L,\beta,[\phi(\cdot+h),\psi(\cdot+h)],\ast}^2 &\leq \norm{(\delth\phi(0),\delth\psi(0))}_{L,\beta,[\phi(\cdot+h),\psi(\cdot+h)],\ast}^2\exp\Big(C\int_0^t P_h(\tau)\dtau\Big) \\
            &\quad + C\int_0^t \norm{(\delth\bv,\delth\bw)}_{\mathbf{L}^{1+\omega}(\Om)\times\mathbf{L}^1(\Ga)}^2\exp\Big(C\int_s^t P_h(\tau)\dtau\Big)\ds.
        \end{split}
    \end{align}
    Then, noting on 
    \begin{align}\label{Est:delth:vw}
        \begin{split}
            &\norm{(\delth\mathbf{v}(t),\delth\mathbf{w}(t))}_{\mathbf{L}^{1+\omega}(\Om)\times\mathbf{L}^1(\Ga)} \leq \frac1h \int_t^{t+h} \norm{(\delt\mathbf{v}(s),\delt\mathbf{w}(s))}_{\mathbf{L}^{1+\omega}(\Om)\times\mathbf{L}^1(\Ga)}\ds, \\
            &\lim_{h\rightarrow 0} \frac1h \int_t^{t+h} \norm{(\delt\mathbf{v}(s),\delt\mathbf{w}(s))}_{\mathbf{L}^{1+\omega}(\Om)\times\mathbf{L}^1(\Ga)}\ds = \norm{(\delt\mathbf{v}(t),\delt\mathbf{w}(t))}_{\mathbf{L}^{1+\omega}(\Om)\times\mathbf{L}^1(\Ga)}
        \end{split}
    \end{align}
    for a.e. $t\in(0,\infty)$, and \eqref{HighReg:Control:Init}, we observe that the right-hand side of \eqref{Conv:HighReg:Gronwall:Result} is uniformly bounded in $h\in(0,1)$. Moreover, by \eqref{CCH:Est:Energy:Thm} and \eqref{CCH:Est:H^2:Thm}, for any $h\in(0,1)$, we have
    \begin{align}\label{Est:P_h}
        \int_0^t P_h(\tau)\dtau \leq C\int_0^t 1 + \norm{(\bv,\bw)}_{\mathbfcal{L}^2}^2\dtau.
    \end{align}
    Consequently, taking the limit $h\rightarrow 0$ in \eqref{Conv:HighReg:Gronwall:Result}, and noting on \eqref{HighReg:Control:Init}, \eqref{Est:delth:vw} and \eqref{Est:P_h} yields
    \begin{align}\label{Est:delt:pp:Lbast:HighReg:1}
        \norm{(\delt\phi(t),\delt\psi(t))}_{L,\beta,\ast}^2 &\leq C\Big(1 + \norm{(\mu_0,\theta_0)}_{L,\beta}^2 + \norm{(\bv,\bw)}_{L^\infty(0,\infty;\mathbfcal{L}^2)}^2 + \int_0^t\norm{(\delt\bv,\delt\bw)}_{\mathbf{L}^{1+\omega}(\Om)\times\mathbf{L}^1(\Ga)}^2\dtau\Big) \nonumber \\
        &\quad\times \exp\Big(C\int_0^t 1 + \norm{(\bv,\bw)}_{\mathbfcal{L}^2}^2\dtau\Big).
    \end{align}
    Thus, we find a constant $C_0 > 0$ such that
    \begin{align}\label{Est:Conv:1:HighReg:delt:pp:t<1}
        \esssup_{t\in(0,1)}\norm{(\delt\phi(t),\delt\psi(t))}_{(\mathcal{H}^1_{L,\beta})^\prime}^2\leq C_0 \qquad\text{for~a.e.~}t\in(0,1).
    \end{align}
    To derive a global control in time, we report the following uniform Gronwall lemma, which can be found, e.g., in \cite[Chapter~III, Lemma~1.1]{Temam1997}.
    \begin{lemma}\label{Lemma:Gronwall}
        Let $f:[t_0,\infty)\rightarrow\R$ be an absolutely continuous positive function, and $g,h$ two positive locally summable functions on $[t_0,\infty)$, which satisfy
        \begin{align*}
            \ddt f(t) &\leq g(t)f(t) + h(t) \qquad\text{for~a.e.~}t\geq t_0,
        \end{align*}
        and
        \begin{align*}
            \int_t^{t+r} g(s)\ds \leq a_1, \quad \int_t^{t+r} h(s)\ds &\leq a_2, \quad \int_t^{t+r} f(s)\ds \leq a_3 \quad\text{for all~}t\geq t_0,
        \end{align*}
        where $r,a_1,a_2,a_3$ are positive constants. Then it holds
        \begin{align*}
            f(t) \leq \left(\frac{a_3}{r} + a_2\right)e^{a_1} \qquad\text{for all~}t\geq t_0 + r.
        \end{align*}
    \end{lemma}
    \noindent
    Since
    \begin{align}\label{Est:Unif:Gronwall:a123}
        \begin{split}
            &\sup_{t\geq 0}\int_t^{t+1} P_h(s)\ds \leq C\big(1 + \norm{(\bv,\bw)}_{L^\infty(0,\infty;\mathbfcal{L}^2)}^2\big), \\
            &\sup_{t\geq 0}\int_t^{t+1}\norm{(\delth\phi(s),\delth\psi(s))}_{L,\beta,[\phi(\cdot+h),\psi(\cdot+h)],\ast}^2 \ds \leq C\big(1 + \norm{(\bv,\bw)}_{L^\infty(0,\infty;\mathbfcal{L}^2)}^2\big), \\
            &\sup_{t\geq 0} \int_t^{t+1} \norm{(\delth\mathbf{v}(s),\delth\mathbf{w}(s))}_{\mathbf{L}^{1+\omega}(\Om)\times\mathbf{L}^1(\Ga)}^2\ds\leq \norm{(\delt\bv,\delt\bw)}_{L^2_{\mathrm{uloc}}(0,\infty;\mathbf{L}^{1+\omega}(\Om)\times\mathbf{L}^1(\Ga))}^2,
        \end{split}
    \end{align}
    an application of Lemma~\ref{Lemma:Gronwall} with $t_0 = 0$ and $r = 1$ shows that
    \begin{align}\label{Est:HighReg:Delt:pp}
        \begin{split}
            &\norm{(\delt\phi(t),\delt\psi(t))}_{L,\beta,\ast}^2 \\
            &\quad \leq C\big(1 + \norm{(\bv,\bw)}_{L^\infty(0,\infty;\mathbfcal{L}^2)}^2 + \norm{(\delt\bv,\delt\bw)}_{L^2_{\mathrm{uloc}}(0,\infty;\mathbf{L}^{1+\omega}(\Om)\times\mathbf{L}^1(\Ga))}^2\big)\\
            &\qquad\times \exp\Big(C\big(1 + \norm{(\bv,\bw)}_{L^2(0,\infty;\mathbfcal{L}^2)}^2\big)\Big) \qquad\text{for~all~}t\geq1.
        \end{split}
    \end{align}
    Here, we have already taken the limit $h\rightarrow 0$ and used again the norm equivalence on $\mathcal{V}^{-1}_{L,\beta}$ (see \eqref{NormEquivalence:EBS}).
    The latter implies together with \eqref{Est:Conv:1:HighReg:delt:pp:t<1} that $(\delt\phi,\delt\psi)\in L^\infty(0,\infty;(\mathcal{H}^1_{L,\beta})^\prime)$, and the estimate \eqref{Est:HighReg:Delt:pp} holds for a.e. $t\in(0,\infty)$. Next, we integrate \eqref{Conv:HighReg:Gronwall:Step1} over the time interval $[t,t+1]$ for $t\geq0$, and utilize the estimates \eqref{Est:Unif:Gronwall:a123} and \eqref{Est:HighReg:Delt:pp}. By passing to the limit $h\rightarrow 0$ in the corresponding estimate and using again the Poincar\'{e} inequality, we deduce that
    \begin{align}\label{Est:HighReg:pp:K}
        &\sup_{t\geq0}\int_t^{t+1} \norm{(\delt\phi(s),\delt\psi(s))}_{\mathcal{H}^1}^2\ds \nonumber \\
        &\quad\leq C\Big(1 + \norm{(\mu_0,\theta_0)}_{L,\beta}^2 + \norm{(\bv,\bw)}_{L^\infty(0,\infty;\mathbfcal{L}^2)}^2 + \norm{(\delt\bv,\delt\bw)}_{L^2_{\mathrm{uloc}}(0,\infty;\mathbf{L}^{1+\omega}(\Om)\times\mathbf{L}^1(\Ga))}^2\Big) \\
        &\qquad\times \bigg(1 + \norm{(\bv,\bw)}_{L^2(0,\infty;\mathbfcal{L}^2)}^2\Big)\exp\Big(C\big(1 + \norm{(\bv,\bw)}_{L^2(0,\infty;\mathbfcal{L}^2)}^2\big)\bigg), \nonumber
    \end{align}
    which gives $(\delt\phi,\delt\psi)\in L^2_{\mathrm{uloc}}([0,\infty);\mathcal{H}^1)$. Then, a comparison argument in \eqref{WF:CCH:PP} together with \eqref{Est:HighReg:Delt:pp} and $(\bv,\bw)\in L^\infty(0,\infty;\mathbfcal{L}^2)$ further shows that
    \begin{align}\label{Est:HighReg:MT:L:Linfty}
        \esssup_{t\in(0,\infty)}\norm{(\mu(t),\theta(t))}_{L,\beta} \leq C\big(\norm{(\delt\phi,\delt\psi)}_{L^\infty(0,\infty;(\mathcal{H}^1_{L,\beta})^\prime} + \norm{(\bv,\bw)}_{L^\infty(0,\infty;\mathbfcal{L}^2)}\big).
    \end{align}
    Thus, in light of \eqref{Est:Mean:mt} and \eqref{Est:HighReg:Delt:pp}, we have $(\mu,\theta)\in L^\infty(0,\infty;\mathcal{H}^1)$ by exploiting the bulk-surface Poincar\'{e} inequality. By \eqref{Est:pp:Pot:Lp:a.e.}, we learn that
    \begin{align}\label{Est:pp:Pot:Lp:Linfty}
        (\phi,\psi)\in L^\infty(0,\infty;\mathcal{W}^{2,p}), \qquad (F^\prime(\phi),G^\prime(\psi))\in L^\infty(0,\infty;\mathcal{L}^p)
    \end{align}
    for any $2 \leq p < \infty$. Moreover, by arguing as for \cite[Theorem~3.7]{Stange2025}, we find
    \begin{align*}
        (F^{\prime\prime}(\phi), G^{\prime\prime}(\psi)) \in L^\infty(0,\infty;\mathcal{L}^p)
    \end{align*}
    for any $2 \leq p < \infty$ (see also \cite{Lv2024b, Lv2024a}), from which we deduce that
    \begin{align*}
        (\phi,\psi)\in L^\infty(0,\infty;\mathcal{H}^3)
    \end{align*}
    by elliptic regularity theory for systems with bulk-surface coupling (see, e.g., \cite[Theorem~3.3]{Knopf2021}).
    Lastly, to derive higher-order estimates for the chemical potentials, we first notice that
    \begin{align}\label{ID:MT:MEAN:SOL:CONVECTIVE}
        (\mu - \beta\mean{\mu}{\theta},\theta - \mean{\mu}{\theta}) = - \mathcal{S}_{L,\beta}[\phi,\psi](\delt\phi + \Grad\phi\cdot\mathbf{v}, \delt\psi + \Gradg\psi\cdot\mathbf{w}).
    \end{align}
    Then, by \eqref{BSE:MOB:Est:Apriori}, \eqref{Est:fg:L^2:K}, \eqref{Est:Sol:G:H^2}, \eqref{Est:HighReg:MT:L:Linfty} and \eqref{Est:pp:Pot:Lp:Linfty}, we have
    \begin{align*}
        &\norm{(\mu - \beta\mean{\mu}{\theta},\theta - \mean{\mu}{\theta})}_{\mathcal{H}^2} \nonumber \\
        &\quad = \norm{\mathcal{S}_{L,\beta}[\phi,\psi](\delt\phi + \Grad\phi\cdot\mathbf{v}, \delt\psi + \Gradg\psi\cdot\mathbf{w})}_{\mathcal{H}^2} \nonumber \\
        &\quad\leq C\big(\norm{(\Grad\phi,\Gradg\psi)}_{\mathcal{L}^2}\norm{(\phi,\psi)}_{\mathcal{H}^2}\norm{\mathcal{S}_{L,\beta}[\phi,\psi](\delt\phi + \Grad\phi\cdot\mathbf{v}, \delt\psi + \Gradg\psi\cdot\mathbf{w})}_{L,\beta} \nonumber \\
        &\qquad + \norm{(\delt\phi + \Grad\phi\cdot\mathbf{v}, \delt\psi + \Gradg\psi\cdot\mathbf{w})}_{\mathcal{L}^2}\big) \nonumber \\
        &\quad\leq C\big( \norm{\mathcal{S}_{L,\beta}[\phi,\psi](\delt\phi,\delt\psi)}_{L,\beta} + \norm{\mathcal{S}_{L,\beta}[\phi,\psi](\Grad\phi\cdot\mathbf{v},\Gradg\psi\cdot\mathbf{w})}_{L,\beta} + \norm{(\delt\phi,\delt\psi)}_{\mathcal{L}^2} \\
        &\qquad + \norm{(\Grad\phi\cdot\mathbf{v},\Gradg\psi\cdot\mathbf{w})}_{\mathcal{L}^2}\big) \nonumber \\
        &\quad\leq C\big(1 + \norm{(\delt\phi,\delt\psi)}_{\mathcal{L}^2} + \norm{(\Grad\phi,\Gradg\psi)}_{\mathbfcal{L}^\infty}\norm{(\mathbf{v},\mathbf{w})}_{\mathbfcal{L}^2}\big) \nonumber \\
        &\quad\leq C\big(1 + \big(\norm{(\mu,\theta)}_{L,\beta}^{\frac12} + \norm{(\bv,\bw)}_{\mathbfcal{L}^2}^{\frac12}\big)\norm{(\delt\phi,\delt\psi)}_{K,\alpha}^{\frac12} + \norm{(\mathbf{v},\mathbf{w})}_{\mathbfcal{L}^2}\big) \nonumber \\
        &\quad\leq C\big(1 + \norm{(\delt\phi,\delt\psi)}_{K,\alpha}^{\frac12}\big). \nonumber
    \end{align*}
    Integrating the previous estimate in time from $t$ to $t+1$ for $t\geq 0$, it follows that
    \begin{align}\label{Est:HighReg:MT:H^2}
        \sup_{t\geq0}\int_t^{t+1}\norm{(\mu - \beta\mean{\mu}{\theta},\theta - \mean{\mu}{\theta})}_{\mathcal{H}^2}^4\ds \leq C.
    \end{align}
    In light of \eqref{Est:Mean:mt}, the latter entails that $(\mu,\theta)\in L^4_{\mathrm{uloc}}([0,\infty);\mathcal{H}^2)$.

\textbf{Case }\ref{Cond:Conv:ii}: Now, we assume that $(\bv,\bw)\in L^2(0,\infty;\mathbfcal{H}^1)\cap L^\infty(0,\infty;\mathbfcal{L}^2_\Div)$. Using a standard mollification, we find a sequence $\{(\bv_k,\bw_k)\}_{k\in\N}\subset C_c^\infty(0,\infty;\mathbfcal{H}^1\cap\mathbfcal{L}^2_\Div)$ such that 
\begin{align}\label{Conv:vw:ep}
    \begin{split}
        (\mathbf{v}_k,\mathbf{w}_k) \rightarrow (\mathbf{v},\mathbf{w}) &\qquad\text{strongly in~} L^2(0,\infty;\mathbfcal{H}^1), \\
        &\qquad\text{weakly-$\ast$ in~} L^\infty(0,\infty;\mathbfcal{L}^2),
    \end{split}
\end{align}
as $k\rightarrow\infty$. For more details, we refer to \cite[S.~4.13-4.15]{Alt2016}.
In particular, according to \cite[S.~4.13 (2)]{Alt2016}, we have
\begin{align}\label{Est:vw:ep}
    \norm{(\mathbf{v}_k,\mathbf{w}_k)}_{L^\infty(0,\infty;\mathbfcal{L}^2)\cap L^2(0,\infty;\mathbfcal{H}^1)} 
    \leq \norm{(\mathbf{v},\mathbf{w})}_{L^\infty(0,\infty;\mathbfcal{L}^2)\cap L^2(0,\infty;\mathbfcal{H}^1)} 
\end{align}
for all $k\in\N$.
Now, owing to the first part of the proof, the unique global weak solution $(\phi_k,\psi_k,\mu_k,\theta_k)$ to \eqref{EQ:CONV:SYSTEM}, which exists according to Theorem~\ref{Theorem:CCH:Existence} and Theorem~\ref{Theorem:CCH:Uniqueness}, is a strong solution with the regularities
\begin{align*}
    (\phi_k,\psi_k)&\in L^\infty(0,\infty;\mathcal{H}^3), \quad (\delt\phi_k,\delt\psi_k)\in L^\infty(0,\infty;(\mathcal{H}^1_{L,\beta})^\prime)\cap L^2_{\mathrm{uloc}}([0,\infty);\mathcal{H}^1), \\
    (\mu_k,\theta_k)&\in L^\infty(0,\infty;\mathcal{H}^1_{L,\beta})\cap L^2_{\mathrm{uloc}}([0,\infty);\mathcal{H}^2), \\
    (F^\prime(\phi_k)&,G^\prime(\psi_k)), \ (F^{\prime\prime}(\phi_k), G^{\prime\prime}(\psi_k))\in L^\infty(0,\infty;\mathcal{L}^p)
\end{align*}
for any $2 \leq p < \infty$ and all $k\in\N$. Then, in view of \eqref{Eq:mu:strong}-\eqref{Eq:bd:strong}, one can deduce the existence of the weak time derivative $(\delt\mu_k,\delt\theta_k)$ in the sense that $(\delt\mu_k,\delt\theta_k)\in L^2_{\mathrm{uloc}}([0,\infty);(\mathcal{H}^1)^\prime)$, and it holds
\begin{align}\label{WF:DELT:MT}
    \begin{split}
        \big\langle (\delt\mu_k,\delt\theta_k), (\eta,\vartheta)\big\rangle_{\mathcal{H}^1} &= \intO \Grad\delt\phi_k\cdot\Grad\eta + F^{\prime\prime}(\phi_k)\delt\phi_k\eta\dx \\
        &\quad + \intG \Gradg\delt\psi_k\cdot\Gradg\vartheta + G^{\prime\prime}(\psi_k)\delt\psi_k\vartheta\dG \\
        &\quad + \chi(K)\intG(\alpha\delt\psi_k - \delt\phi_k)(\alpha\vartheta - \eta)\dG
    \end{split}
\end{align}
a.e. on $(0,\infty)$ for all $(\eta,\vartheta)\in\mathcal{H}^1$ and all $k\in\N$, see \cite[Proof of Theorem~3.6, Step~5.1]{Giorgini2025} for more details. Now, following verbatim the proof of Theorem~\ref{Theorem:CCH:Existence}, it follows from \eqref{Est:vw:ep} that
\begin{align}
    &\sup_{t\geq 0}\norm{(\phi_k(t),\psi_k(t))}_{\mathcal{H}^1}^2 \leq C\big(1 + \norm{(\bv,\bw)}_{L^2(0,\infty;\mathbfcal{L}^2)}^2\big), \label{Est:HighReg:pp:k}\\
    &\int_0^\infty \norm{(\mu_k,\theta_k)}_{L,\beta}^2 + \norm{(\delt\phi_k,\delt\psi_k)}_{(\mathcal{H}^1_{L,\beta})^\prime}^2 \ds \leq C\big(1 + \norm{(\bv,\bw)}_{L^2(0,\infty;\mathbfcal{L}^2)}^2\big), \label{Est:HighReg:mt:delt:pp:k}\\
    &\sup_{t\geq 0}\int_t^{t+1}\norm{(\phi_k,\psi_k)}_{\mathcal{H}^2}^4\ds \leq C\big(1 + \norm{(\bv,\bw)}_{L^2(0,\infty;\mathbfcal{L}^2)}^2\big) \label{Est:HighReg:pp:H^2:k}
\end{align}
as well as
\begin{align}\label{Est:HighReg:pp:pot:k}
    \norm{(\phi_k,\psi_k)}_{\mathcal{W}^{2,p}} + \norm{(F^\prime(\phi_k),G^\prime(\psi_k))}_{\mathcal{L}^p} \leq C\big(1 + \norm{(\mu_k,\theta_k)}_{L,\beta}) \qquad\text{a.e.~in~}(0,\infty)
\end{align}
for any $2 \leq p < \infty$. To establish the higher regularity estimates, we aim to apply the chain rule \cite[Proposition~A.1]{Stange2025}. To this end, we first have to show that $(\mu_k,\theta_k) \in L^2_{\mathrm{uloc}}(0,\infty;\mathcal{H}^3)$. Recalling \eqref{Est:Sol:G:H^3} and \eqref{ID:MT:MEAN:SOL:CONVECTIVE}, we have
\begin{align}\label{Est:MT:H^3:mean:Conv}
    &\norm{(\mu_k - \beta\mean{\mu_k}{\theta_k},\theta_k - \mean{\mu_k}{\theta_k})}_{\mathcal{H}^3} \nonumber \\
    &\quad\leq C\big(1 + \mathbf{1}_{\{0\}}(L)\norm{(\phi_k,\psi_k)}_{\mathcal{H}^2}\big) \nonumber \\
    &\qquad\times\Bigg( \Bignorm{\bigg(\frac{\delt\phi_k + \Grad\phi_k\cdot\bv_k}{m_\Om(\phi_k)},\frac{\delt\psi_k + \Gradg\psi_k\cdot\bw_k}{m_\Ga(\psi_k)}\bigg)}_{\mathcal{H}^1} + \Bignorm{\bigg(\frac{m_\Om^\prime(\phi_k)\Grad\phi_k\cdot\Grad\mu_k}{m_\Om(\phi_k)},\frac{m_\Ga^\prime(\psi_k)\Gradg\psi_k\cdot\Gradg\theta_k}{m_\Ga(\psi_k)}\bigg)}_{\mathcal{H}^1}\Bigg) \nonumber \\
    &\quad\leq C\big(1 + \mathbf{1}_{\{0\}}(L)\norm{(\phi_k,\psi_k)}_{\mathcal{H}^2}\big)\Bigg( \Bignorm{\bigg(\frac{\delt\phi_k}{m_\Om(\phi_k)},\frac{\delt\psi_k}{m_\Ga(\psi_k)}\bigg)}_{\mathcal{H}^1} + \Bignorm{\bigg(\frac{\Grad\phi_k\cdot\bv_k}{m_\Om(\phi_k)},\frac{\Gradg\psi_k\cdot\bw_k}{m_\Ga(\psi_k)}\bigg)}_{\mathcal{H}^1} \\
    &\qquad + \Bignorm{\bigg(\frac{m_\Om^\prime(\phi_k)\Grad\phi_k\cdot\Grad\mu}{m_\Om(\phi_k)},\frac{m_\Ga^\prime(\psi_k)\Gradg\psi_k\cdot\Gradg\theta_k}{m_\Ga(\psi_k)}\bigg)}_{\mathcal{H}^1}\Bigg) \nonumber \\
    &\quad\eqqcolon C\big(1 + \mathbf{1}_{\{0\}}(L)\norm{(\phi_k,\psi_k)}_{\mathcal{H}^2}\big)\big(I_1 + I_2 + I_3). \nonumber
\end{align}
Using standard computations, one readily shows that
\begin{align}\label{I1}
    I_1 \leq C\big(1 + \norm{(\phi_k,\psi_k)}_{\mathcal{W}^{2,4}}\big)\norm{(\delt\phi_k,\delt\psi_k)}_{K,\alpha}
\end{align}
and
\begin{align}\label{I3}
    I_3 &\leq C\big(1 + \norm{(\phi_k,\psi_k)}_{\mathcal{W}^{2,4}}\big)\norm{(\phi_k,\psi_k)}_{\mathcal{W}^{2,4}}\norm{(\mu_k - \beta\mean{\mu_k}{\theta_k}, \theta_k - \mean{\mu_k}{\theta_k})}_{\mathcal{H}^2},
\end{align}
see \cite[Eqn.~(6.10) ff.]{Stange2025} for more details. For the remaining term, we find with \eqref{Est:vw:ep} that
\begin{align}\label{I2}
    &\Bignorm{\bigg(\frac{\Grad\phi_k\cdot\bv_k}{m_\Om(\phi_k)},\frac{\Gradg\psi_k\cdot\bw_k}{m_\Ga(\psi_k)}\bigg)}_{\mathcal{H}^1} \nonumber \\
    &\quad\leq \Bignorm{\bigg(\frac{\Grad\phi_k\cdot\bv_k}{m_\Om(\phi_k)},\frac{\Gradg\psi\cdot\bw_k}{m_\Ga(\psi_k)}\bigg)}_{\mathcal{L}^2} + \Bignorm{\bigg(\frac{m_\Om^\prime(\phi_k)(\Grad\phi_k\cdot\bv_k)\Grad\phi_k}{m_\Om(\phi_k)^2},\frac{m_\Ga^\prime(\psi_k)(\Gradg\psi_k\cdot\bw_k)\Gradg\psi_k}{m_\Ga(\psi_k)^2}\bigg)}_{\mathcal{L}^2} \nonumber \\
    &\qquad + \Bignorm{\bigg(\frac{\Grad(\Grad\phi_k\cdot\bv_k)}{m_\Om(\phi_k)},\frac{\Gradg(\Gradg\psi_k\cdot\bw_k)}{m_\Ga(\psi_k)}\bigg)}_{\mathcal{L}^2} \\
    &\quad\leq C\big(\norm{(\Grad\phi_k,\Gradg\psi_k)}_{\mathbfcal{L}^\infty}\norm{(\bv_k,\bw_k)}_{\mathbfcal{L}^2} + \norm{(\Grad\phi_k,\Gradg\psi_k)}_{\mathbfcal{L}^\infty}^2\norm{(\bv_k,\bw_k)}_{\mathbfcal{L}^2} + \norm{(\phi_k,\psi_k)}_{\mathcal{W}^{2,4}}\norm{(\bv_k,\bw_k)}_{\mathbfcal{L}^3} \nonumber \\
    &\qquad + \norm{(\Grad\phi_k,\Gradg\psi_k)}_{\mathbfcal{L}^\infty}\norm{(\bv_k,\bw_k)}_{\mathbfcal{H}^1}\big) \nonumber \\
    &\quad\leq C\norm{(\phi_k,\psi_k)}_{\mathcal{W}^{2,4}}\big(1 + \norm{(\bv_k,\bw_k)}_{\mathbfcal{H}^1} + \norm{(\phi_k,\psi_k)}_{\mathcal{W}^{2,4}}\big). \nonumber
\end{align}
In view of these bounds, the regularity of the approximate solution $(\phi_k,\psi_k)$ and the approximate velocity fields $(\bv_k,\bw_k)$ entails that $(\mu_k,\theta_k)\in L^2_{\mathrm{uloc}}([0,\infty);\mathcal{H}^3)$. We stress that this bound is not yet uniform in $k\in\N$. In view of this regularity, we are in a position to apply \cite[Proposition~A.1]{Stange2025} which yields $(\mu_k,\theta_k)\in C([0,\infty);\mathcal{H}^1_{L,\beta})$, and the following chain rule
\begin{align}\label{CONV:ChainRule}
    &\ddt\frac12\Big(\intO m_\Om(\phi_k)\abs{\Grad\mu_k}^2\dx + \intG m_\Ga(\psi_k)\abs{\Gradg\theta_k}^2\dG + \chi(L)\intG (\beta\theta_k - \mu_k)^2\dG \Big) \nonumber \\
    &\quad = \bigang{(\delt\mu_k,\delt\theta_k)}{(-\Div(m_\Om(\phi_k)\Grad\mu_k),-\Divg(m_\Ga(\psi_k)\Gradg\psi_k) + \beta m_\Om(\phi_k)\deln\mu_k)}_{\mathcal{H}^1} \\
    &\qquad + \intO m_\Om^\prime(\phi_k)\delt\phi_k\abs{\Grad\mu_k}^2\dx + \intG m_\Ga^\prime(\psi_k)\delt\psi_k\abs{\Gradg\theta_k}^2\dG \nonumber
\end{align}
holds a.e. on $(0,\infty)$. Taking \eqref{EQ:CONV:SYSTEM:1}-\eqref{EQ:CONV:SYSTEM:3} as well as \eqref{WF:DELT:MT} into account, the dual pairing on the right-hand side of \eqref{CONV:ChainRule} can be rewritten as
\begin{align}\label{CONV:ChainRule:Rewrite}
    \begin{split}
        &\bigang{(\delt\mu_k,\delt\theta_k)}{(-\Div(m_\Om(\phi_k)\Grad\mu_k),-\Divg(m_\Ga(\psi_k)\Gradg\psi_k) + \beta m_\Om(\phi_k)\deln\mu_k)}_{\mathcal{H}^1} \\
        &\quad = -\bigang{(\delt\mu_k,\delt\theta_k)}{(\delt\phi_k + \Div(\phi_k\mathbf{v}_k),\delt\psi_k + \Divg(\psi_k\mathbf{w}_k))}_{\mathcal{H}^1} \\
        &\quad = - \norm{(\delt\phi_k,\delt\psi_k)}_{K,\alpha}^2 - \intO F^{\prime\prime}(\phi_k)\abs{\delt\phi_k}^2\dx - \intG G^{\prime\prime}(\psi_k)\abs{\delt\psi_k}^2\dG \\
        &\qquad - \bigang{(\delt\mu_k,\delt\theta_k)}{(\Div(\phi_k\mathbf{v}_k),\Divg(\psi_k\mathbf{w}_k))}_{\mathcal{H}^1}.
    \end{split}
\end{align}
Hence, exploiting the convexity of $F_0$ and $G_0$ (see \eqref{Pot:F_0:Convex}), it follows from \eqref{CONV:ChainRule} and \eqref{CONV:ChainRule:Rewrite} that
\begin{align}\label{Appl:ChainRule:2}
    &\ddt\frac12\Big(\intO m_\Om(\phi_k)\abs{\Grad\mu_k}^2\dx + \intG m_\Ga(\psi_k)\abs{\Gradg\theta_k}^2\dG + \chi(L)\intG (\beta\theta_k - \mu_k)^2\dG \Big)  \nonumber \\
    &\qquad + \norm{(\delt\phi_k,\delt\psi_k)}_{K,\alpha}^2 \nonumber \\
    &\quad \leq \intO c_F\abs{\delt\phi_k}^2\dx + \intG c_G \abs{\delt\psi_k}^2\dG \\
    &\qquad + \intO m_\Om^\prime(\phi_k)\delt\phi_k\abs{\Grad\mu_k}^2\dx + \intG m_\Ga^\prime(\psi_k)\delt\psi_k\abs{\Gradg\theta_k}^2\dG \nonumber \\
    &\qquad - \bigang{(\delt\mu_k,\delt\theta_k)}{(\Div(\phi_k\mathbf{v}_k),\Divg(\psi_k\mathbf{w}_k))}_{\mathcal{H}^1}. \nonumber
\end{align}
For the first two terms on the right-hand side of \eqref{Appl:ChainRule:2}, we have already seen in the first part of the proof that these can be estimated as
\begin{align}\label{Est:HighReg:LowOrd}
    \begin{split}
        c_F\norm{\delt\phi_k}_{L^2(\Om)}^2 + c_G\norm{\delt\psi_k}_{L^2(\Ga)}^2\leq \frac16\norm{(\delt\phi_k,\delt\psi_k)}_{K,\alpha}^2 + C\Big(\norm{(\bv_k,\bw_k)}_{\mathbfcal{L}^2}^2 + \norm{(\mu_k,\theta_k)}_{L,\beta}^2\Big).
    \end{split}
\end{align}
Then, for the next two terms, we use again the computations made in \cite[Eqns.~(6.17)--(6.18)]{Stange2025}, which show that
\begin{align}\label{Est:HighReg:Mob}
    \begin{split}
        &\Bigabs{\intO m_\Om^\prime(\phi_k)\delt\phi_k\abs{\Grad\mu_k}^2\dx + \intG m_\Ga^\prime(\psi_k)\delt\psi_k\abs{\Gradg\theta_k}^2\dG} \\
        &\quad\leq \frac16\norm{(\delt\phi_k,\delt\psi_k)}_{K,\alpha}^2 + C\Big(\norm{(\phi_k,\psi_k)}_{\mathcal{H}^2}^4 + \norm{(\mu_k,\theta_k)}_{L,\beta}^2\Big)\norm{(\mu_k,\theta_k)}_{L,\beta}^2,
    \end{split}
\end{align}
while for the last term, it has been established in \cite[Eqn.~(7.58)]{Giorgini2025} that
\begin{align}\label{Est:HighReg:Conv}
    \begin{split}
        &\bigabs{\bigang{(\delt\mu_k,\delt\theta_k)}{(\Div(\phi_k\mathbf{v}_k),\Divg(\psi_k\mathbf{w}_k))}_{\mathcal{H}^1}} \\
        &\quad \leq \bigabs{\big((\delt\phi_k,\delt\psi_k),(\Div(\phi_k\mathbf{v}_k),\Divg(\psi_k\mathbf{w}_k))\big)_{K,\alpha}} \\
        &\qquad + \Bigabs{\intO F^{\prime\prime}(\phi_k)\delt\phi_k\Div(\phi_k\mathbf{v}_k)\dx} + \Bigabs{\intG G^{\prime\prime}(\psi_k)\delt\psi_k\Divg(\psi_k\mathbf{w}_k)\dG} \\
        &\quad\leq \frac16\norm{(\delt\phi_k,\delt\psi_k)}_{K,\alpha}^2 + C\norm{(\mu_k,\theta_k)}_{L,\beta}^2 + C\norm{(\bv_k,\bw_k)}_{\mathbfcal{H}^1}^2\norm{(\mu_k,\theta_k)}_{L,\beta}^2.
    \end{split}
\end{align}
Thus, collecting the estimates \eqref{Est:HighReg:LowOrd}, \eqref{Est:HighReg:Mob} and \eqref{Est:HighReg:Conv}, we end up with the differential inequality
\begin{align}\label{Appl:ChainRule:Final}
    \begin{split}
        &\ddt\frac12\Big(\intO m_\Om(\phi_k)\abs{\Grad\mu_k}^2\dx + \intG m_\Ga(\psi_k)\abs{\Gradg\theta_k}^2\dG + \chi(L)\intG (\beta\theta_k - \mu_k)^2\dG \Big) \\
        &\qquad + \frac12\norm{(\delt\phi_k,\delt\psi_k)}_{K,\alpha}^2 \\
        &\quad \leq C\big(\norm{(\bv_k,\bw_k)}_{\mathbfcal{L}^2}^2 + \norm{(\mu_k,\theta_k)}_{L,\beta}^2\big) + C\Big(\norm{(\bv_k,\bw_k)}_{\mathbfcal{H}^1}^2 + \norm{(\phi_k,\psi_k)}_{\mathcal{H}^2}^4 + \norm{(\mu_k,\theta_k)}_{L,\beta}^2\Big) \\
        &\qquad\times\Big(\intO m_\Om(\phi_k)\abs{\Grad\mu_k}^2\dx + \intG m_\Ga(\psi_k)\abs{\Gradg\theta_k}^2\dG + \chi(L)\intG (\beta\theta_k - \mu_k)^2\dG\Big)
    \end{split}
\end{align}
a.e. on $(0,\infty)$. Applying Gronwall's lemma yields
\begin{align*}
    \norm{(\mu_k(t),\theta_k(t)}_{L,\beta}^2 &\leq C\norm{(\mu_k(0),\theta_k(0))}_{L,\beta}^2\exp\Big(C\int_0^t \norm{(\bv_k,\bw_k)}_{\mathbfcal{L}^2}^2 + \norm{(\mu_k,\theta_k)}_{L,\beta}^2\dtau\Big) \nonumber \\
    &\quad + C\int_0^t \Big(\norm{(\bv_k,\bw_k)}_{\mathbfcal{H}^1}^2 + \norm{(\phi_k,\psi_k)}_{\mathcal{H}^2}^4 + \norm{(\mu_k,\theta_k)}_{L,\beta}^2\Big)\\
    &\qquad\qquad\times\exp\Big(C\int_s^t \norm{(\bv_k,\bw_k)}_{\mathbfcal{L}^2}^2 + \norm{(\mu_k,\theta_k)}_{L,\beta}^2\dtau\Big)\ds \nonumber.
\end{align*}
As $(\mu_k(0),\theta_k(0)) = (\mu_0,\theta_0)$, we readily infer from \eqref{Est:vw:ep}, \eqref{Est:HighReg:pp:k}, \eqref{Est:HighReg:mt:delt:pp:k}, and \eqref{Est:HighReg:pp:H^2:k} that
\begin{align}\label{CCH:MT:k:<1}
    \begin{split}
        \sup_{t\in[0,1]}\norm{(\mu_k(t),\theta_k(t)}_{L,\beta}^2&\leq C\Big(1 + \norm{(\mu_0,\theta_0)}_{L,\beta}^2 + \int_0^\infty \norm{(\bv,\bw)}_{\mathbfcal{H}^1}^2\ds\Big)\\
        &\qquad\times\exp\Big(C\big(1 + \int_0^\infty \norm{(\bv,\bw)}_{\mathbfcal{H}^1}^2\ds\big)\Big).
    \end{split}
\end{align}
To derive a global control in time, we aim to use again the uniform Gronwall lemma~\ref{Lemma:Gronwall}. In view of \eqref{Est:vw:ep}, \eqref{Est:HighReg:mt:delt:pp:k} and \eqref{Est:HighReg:pp:H^2:k}, we obtain similarly to the first part of the proof that
\begin{align}\label{CCH:MT:k:>1}
    \begin{split}
        \sup_{t\geq 1}\norm{(\mu_k(t),\theta_k(t))}_{L,\beta}^2 &\leq C\Big(1 + \norm{(\mu_0,\theta_0)}_{L,\beta}^2 + \int_0^\infty \norm{(\bv,\bw)}_{\mathbfcal{H}^1}^2\ds\Big)\\
        &\qquad\times\exp\Big(C\big(1 + \int_0^\infty \norm{(\bv,\bw)}_{\mathbfcal{H}^1}^2\ds\big)\Big).
    \end{split}
\end{align}
Next, integrating \eqref{Appl:ChainRule:Final} in time from $t$ to $t+1$ for any $t\geq 0$, we readily deduce on account of \eqref{CCH:MT:k:<1}, \eqref{CCH:MT:k:>1} and the bulk-surface Poincar\'{e} inequality that
\begin{align}\label{CCH:PP:k}
    \begin{split}
        \sup_{t\geq0}\int_t^{t+1}\norm{(\delt\phi_k,\delt\psi_k)}_{\mathcal{H}^1}^2\ds &\leq C\Big(1 + \norm{(\mu_0,\theta_0)}_{L,\beta}^2 + \int_0^\infty \norm{(\bv,\bw)}_{\mathbfcal{H}^1}^2\ds\Big)\\
        &\quad\times\Big(C\big(1 + \int_0^\infty\norm{(\bv,\bw)}_{\mathbfcal{H}^1}^2\ds\big)\Big)\exp\Big(C\big(1 + \int_0^\infty \norm{(\bv,\bw)}_{\mathbfcal{H}^1}^2\ds\big)\Big).
    \end{split}
\end{align}
Now, based on the estimates \eqref{CCH:MT:k:<1} and \eqref{CCH:MT:k:>1}, it follows from \eqref{Est:HighReg:pp:pot:k} that
\begin{align}\label{Est:HighReg:pp:pot:k:unif}
    \sup_{t\geq 0} \norm{(\phi_k(t),\psi_k(t))}_{\mathcal{W}^{2,p}} + \sup_{t\geq 0}\norm{(F^\prime(\phi_k(t)),G^\prime(\psi_k(t)))}_{\mathcal{L}^p} \leq C_p
\end{align}
for any $2 \leq p < \infty$. Moreover, by arguing as for \cite[Theorem~3.7]{Stange2025} (see also \cite{Lv2024b, Lv2024a}), we find
\begin{align*}
    \sup_{t\geq 0}\norm{(F^{\prime\prime}(\phi_k(t)), G^{\prime\prime}(\psi_k(t))}_{\mathcal{L}^p} + \sup_{t\geq 0}\norm{(\phi_k(t),\psi_k(t))}_{\mathcal{H}^3} \leq C_p^\prime.
\end{align*}
Next, following the same computations as for \eqref{Est:HighReg:MT:H^2}, we find
\begin{align}\label{Est:HighReg:MT:H^2:k:unif}
    \sup_{t\geq 0}\int_t^{t+1}\norm{(\mu_k - \beta\mean{\mu_k}{\theta_k},\theta_k - \mean{\mu_k}{\theta_k})}_{\mathcal{H}^2}^4\ds \leq C.
\end{align}
Then, on account of \eqref{I1}-\eqref{I2}, we deduce with \eqref{Est:vw:ep}, \eqref{Est:HighReg:pp:pot:k:unif} and  \eqref{Est:HighReg:MT:H^2:k:unif} from \eqref{Est:MT:H^3:mean:Conv} that
\begin{align}\label{Est:HighReg:MT:H^3:k:unif}
    \sup_{t\geq 0}\int_t^{t+1}\norm{(\mu_k - \beta\mean{\mu_k}{\theta_k},\theta_k - \mean{\mu_k}{\theta_k})}_{\mathcal{H}^3}^2\ds \leq C.
\end{align}
In light of \eqref{Est:Mean:mt}, the latter estimates entail that
\begin{align*}
    \sup_{t\geq 0}\int_t^{t+1} \norm{(\mu_k,\theta_k)}_{\mathcal{H}^2}^4\ds + \sup_{t\geq 0}\int_t^{t+1}\norm{(\mu_k,\theta_k)}_{\mathcal{H}^3}^2\ds \leq C.
\end{align*}
By standard compactness arguments, these estimates guarantee the existence of a limit quadruple $(\phi_\ast,\psi_\ast,\mu_\ast,\theta_\ast)$ solving \eqref{EQ:CONV:SYSTEM} and satisfying the desired regularity properties. By Theorem~\ref{Theorem:CCH:Uniqueness}, $(\phi_\ast,\psi_\ast,\mu_\ast,\theta_\ast)$ is thus the unique strong solution to \eqref{EQ:CONV:SYSTEM}, which finishes the proof.
\end{proof}

\medskip

\section{On the bulk-surface Stokes equation}
\label{Section:Stokes}
\setcounter{equation}{0}
We consider a bulk-surface Stokes system with non-constant viscosity functions and a non-constant friction parameter on $\Om\subset\R^d$ with $d\in\{2,3\}$. The system reads as
\begin{subequations}\label{System:BSS}
    \begin{alignat}{2}
        -\Div(2\nu_\Om(\phi)\D\bv) + \Grad p &= \f &&\qquad\text{in~}\Om, \\
        \Div\,\bv &= 0 &&\qquad\text{in~}\Om, \\
        -\Divg(2\nu_\Ga(\psi)\Dg\bw) + 2\nu_\Om(\phi)[\D\bv\,\n]_\tau + \Gradg q + \gamma(\phi,\psi)\bw &= \g &&\qquad\text{on~}\Ga, \\
        \Divg\,\bw &= 0 &&\qquad\text{on~}\Ga, \\
        \bv\vert_\Ga = \bw, \ \bv\cdot\n &= 0 &&\qquad\text{on~}\Ga.
    \end{alignat}
\end{subequations}
Here, $\phi:\Om\rightarrow\R$ and $\psi:\Ga\rightarrow\R$ are given measurable functions with $\abs{\phi} \leq 1$ a.e. in $\Om$ and $\abs{\psi} \leq 1$ a.e. on $\Ga$, and the coefficients fulfill the following assumption:
\begin{enumerate}[label=\textnormal{\bfseries(A)}]
    \item\label{Appendix:Assumption:Coefficients:Stokes} It holds $\nu_\Om,\nu_\Ga\in C([-1,1]$, $\gamma\in C([-1,1]^2)$, and there exist positive constants $\nu_\ast,\nu^\ast,\gamma_\ast, \gamma^\ast$ such that
    \begin{align*}
        0 < \nu_\ast \leq \nu_\Om(s), \nu_\Ga(s) \leq \nu^\ast \qquad\text{for all~}s\in[-1,1]
    \end{align*}
    as well as
    \begin{align*}
        0 < \gamma_\ast \leq \gamma(s,r) \leq \gamma^\ast \qquad\text{for all~}(s,r)\in[-1,1]^2.
    \end{align*}
\end{enumerate}
Under this assumption, the authors of \cite{Knopf2025a} proved the following theorem about the well-posedness of weak solutions to \eqref{System:BSS} (see \cite[Theorem~5.4]{Knopf2025a}):

\begin{theorem}\label{App:Lemma:BulkSurfaceStokes:WS}
    Let $(\f,\g)\in(\mathbfcal{H}^1_0)^\prime$. Then there exists a unique weak solution $(\bv,\bw,p,q)$ consisting of a pair of velocity fields $(\bv,\bw)\in\mathbfcal{H}^1_0$ and an associated pressure pair $(p,q)\in\mathcal{L}^2_{(0)}$ such that
    \begin{align*}
        &\intO 2\nu_\Om(\phi)\D\bv:\D\wv\dx + \intG 2\nu_\Ga(\psi)\Dg\bw:\Dg\ww\dG + \intG \gamma(\phi,\psi)\bw\cdot\ww\dG \nonumber \\
        &\qquad - \intO p\,\Div\,\wv\dx - \intG q\,\Divg\,\ww\dG \\
        &\quad = \big\langle (\f,\g),(\wv,\ww)\big\rangle_{\mathbfcal{H}^1_0} \nonumber
    \end{align*}
    for all $(\wv,\ww)\in\mathbfcal{H}^1_0$. Moreover, there exists a constant $C > 0$, depending only on the parameters of the systems, such that
    \begin{align}\label{App:Est:BulkSurfaceStokes:H^1}
        \norm{(\bv,\bw)}_{\mathbfcal{H}^1_0} + \norm{(p,q)}_{\mathcal{L}^2} \leq C\norm{(\f,\g)}_{(\mathbfcal{H}^1_0)^\prime}.
    \end{align}
\end{theorem}

Additionally, in the case of constant viscosity functions and a constant friction function, the authors of \cite{Knopf2025a} proved the existence of a strong solution and investigated the corresponding bulk-surface Stokes operator. The goal of this section is to generalize their results and show that the existence of strong solutions also holds for a more general class of coefficients. Namely, we assume the following condition:
\begin{enumerate}[label=\textnormal{\bfseries(A)$^\prime$}]
    \item \label{Appendix:Assumption':Coefficients:Stokes} It holds $\nu_\Om,\nu_\Ga\in C^2([-1,1]$, $\gamma\in C([-1,1]^2)$, and there exist positive constants $\nu_\ast,\nu^\ast,\gamma_\ast, \gamma^\ast$ such that
    \begin{align}\label{Appendix:Assumption':Bound}
        0 < \nu_\ast \leq \nu_\Om(s), \nu_\Ga(s) \leq \nu^\ast \qquad\text{for all~}s\in[-1,1]
    \end{align}
    as well as
    \begin{align*}
        0 < \gamma_\ast \leq \gamma(s,r) \leq \gamma^\ast \qquad\text{for all~}(s,r)\in[-1,1]^2.
    \end{align*}
\end{enumerate}
Under these assumptions, the goal of this section is to prove the following theorem:
\begin{theorem}\label{App:Theorem:BulkSurfaceStokes:SS}
    Let \ref{Appendix:Assumption':Coefficients:Stokes} hold, and let $(\phi,\psi)\in \mathcal{W}^{1,\infty}$ and $(\f,\g)\in\mathbf{L}^2(\Om)\times\mathbf{L}^2_\tau(\Ga)$. Then, there exists a unique solution $(\bv,\bw,p,q)\in\mathbfcal{H}^2_{0,\Div}\times\big(\mathcal{H}^1\cap\mathcal{L}^2_{(0)}\big)$ to \eqref{System:BSS}. Furthermore, there exists a constant $C > 0$ such that
    \begin{align}\label{App:Est:BulkSurfaceStokes}
        \norm{(\bv,\bw)}_{\mathbfcal{H}^2} + \norm{(p,q)}_{\mathcal{H}^1} \leq C\norm{(\f,\g)}_{\mathbfcal{L}^2}.
    \end{align}
\end{theorem}

Similar to the proof of \cite[Theorem~5.9]{Knopf2025a}, the proof of Theorem~\ref{App:Theorem:BulkSurfaceStokes:SS} is based on Schauder's fixed point theorem, by combining regularity theory for both the bulk and the surface Stokes equations. Therefore, we start by studying both equations separately.

\subsection{Non-homogeneous Stokes equation with non-constant viscosity}

We consider
\begin{subequations}\label{System:BulkStokes:non-const}
    \begin{alignat}{2}
        -\Div(2\nu_\Om(\phi)\D\bv) + \Grad p &= \f &&\qquad\text{in~}\Om, \\
        \Div\,\bv &= 0 &&\qquad\text{in~}\Om, \\
        \bv\vert_\Ga &= \bw_\ast &&\qquad\text{on~}\Ga,
    \end{alignat}
\end{subequations}
for a given tangential vector field $\bw_\ast:\Ga\rightarrow\R^d$. We have the following result regarding the well-posedness of strong solutions to \eqref{System:BulkStokes:non-const}.

\begin{lemma}\label{App:Lemma:BulkStokes}
    Let $\nu_\Om\in C^2([-1,1])$ satisfy \eqref{Appendix:Assumption':Bound}, $\phi\in W^{1,\infty}(\Om)$, $\f\in\mathbf{L}^2(\Om)$ and $\bw_\ast\in\mathbf{H}^{\frac32}_\tau(\Ga)$. Then, there exists a unique solution $(\bv,p)\in\mathbf{H}^2(\Om)\times(H^1(\Om)\cap L^2_{(0)}(\Om))$ to \eqref{System:BulkStokes:non-const}. Furthermore, there exists a constant $C > 0$ such that
    \begin{align}\label{App:Est:BulkStokes}
        \norm{\bv}_{\mathbf{H}^2(\Om)} + \norm{p}_{H^1(\Om)} \leq C\big(\norm{\f}_{\mathbf{L}^2(\Om)} + \norm{\bw_\ast}_{\mathbf{H}^{\frac32}(\Ga)}\big).
    \end{align}
\end{lemma}

\begin{proof}
    \textbf{Uniqueness.} According to the Lemma of Lax--Milgram, the corresponding homogeneous problem has a unique weak solution. As the zero solution trivially solves the homogeneous problem, the uniqueness of \eqref{System:BulkStokes:non-const} readily follows. \\
    \textbf{Existence.} We first consider the system
    \begin{alignat*}{2}
        -\Div(2\D\bv) + \Grad p &= \f &&\qquad\text{in~}\Om, \nonumber \\
        \Div\,\bv &= 0 &&\qquad\text{in~}\Om, \\
        \bv\vert_\Ga &= \bw_\ast &&\qquad\text{on~}\Ga. \nonumber
    \end{alignat*}
    The existence of a unique solution $(\bv_1,p_1)\in\mathbf{H}^2(\Om)\times(H^1(\Om)\cap L^2_{(0)}(\Om))$ has been proven in \cite{Cattabriga1961} for $d = 3$, and in \cite{Temam1984} for $d = 2$, along with the existence of a constant $C_1 > 0$ such that
    \begin{align*}
        \norm{\bv_1}_{\mathbf{H}^2(\Om)} + \norm{p_1}_{H^1(\Om)} \leq C_1\big(\norm{\f}_{\mathbf{L}^2(\Om)} + \norm{\bw_\ast}_{\mathbf{H}^{\frac32}(\Ga)}\big).
    \end{align*}
    Now, we consider
    \begin{subequations}\label{Stokes:Viscous:RHS}
        \begin{alignat}{2}
            -\Div(2\nu_\Om(\phi)\D\bv) + \Grad p &= \f + \Div(2\nu_\Om(\phi)\D\bv_1)&&\qquad\text{in~}\Om, \label{Stokes:Viscous:RHS:1} \\
            \Div\,\bv &= 0 &&\qquad\text{in~}\Om, \\
            \bv\vert_\Ga &= 0 &&\qquad\text{on~}\Ga. 
        \end{alignat}
    \end{subequations}
    As the right-hand side of \eqref{Stokes:Viscous:RHS:1} lies in $\mathbf{L}^2(\Om)$ by assumption on $\nu_\Om$ and the regularity of $\bv_1$, using the Lemma of Lax--Milgram, one can show the existence of a unique $\bv_2\in\mathbf{H}^1_0(\Om)$ satisfying
    \begin{align*}
        \intO 2\nu_\Om(\phi)\D\bv_2:\D\wv\dx = \intO \big(\f + \Div(2\nu_\Om(\phi)\D\bv_1)\big)\cdot\wv\dx
    \end{align*}
    for all $\wv\in\mathbf{H}^1_{0,\Div}(\Om)$. Hence, according to \cite[Section~4, Lemma~4]{Abels2009a}, we additionally find that $\bv_2\in\mathbf{H}^2(\Om)$. Then, applying integration by parts, for all $\wv\in\mathbf{H}^1_{0,\Div}(\Om)$ it holds that
    \begin{align*}
        \intO \Big(-\Div(2\nu_\Om(\phi)\D\bv_2) - \f - \Div(2\nu_\Om(\phi)\D\bv_1)\Big)\cdot\wv\dx = 0.
    \end{align*}
    Consequently, \cite[Lemma~II.2.2.2]{Sohr2001} yields the existence of $p_2\in L^2_{(0)}(\Om)$ such that
    \begin{align}\label{Pressure:p_2}
        \Grad p_2 = -\Div(2\nu_\Om(\phi)\D\bv_2) - \f - \Div(2\nu_\Om(\phi)\D\bv_1)
    \end{align}
    in the sense of distributions. As the right-hand side of \eqref{Pressure:p_2} is an element of $\mathbf{L}^2(\Om)$, we conclude $\Grad p_2\in\mathbf{L}^2(\Om)$, and thus, $(\bv_2,p_2)$ is a strong solution to \eqref{Stokes:Viscous:RHS}. Moreover, we have the following estimate
    \begin{align*}
        \norm{\bv_2}_{\mathbf{H}^2(\Om)} + \norm{p_2}_{H^1(\Om)} \leq C_2\big(\norm{\bv_1}_{\mathbf{H}^2(\Om)} + \norm{\f}_{\mathbf{L}^2(\Om)}\big)
    \end{align*}
    for a constant $C_2 > 0$. Finally, defining $\bv = \bv_1 + \bv_2$ and $p = p_2$, we infer that $(\bv,p)$ is a solution to \eqref{System:BulkStokes:non-const} with the desired regularities and satisfies the estimate \eqref{App:Est:BulkStokes}.
\end{proof}

\subsection{Surface Stokes equations with non-constant viscosity and friction}

We consider
\begin{subequations}\label{System:SurfacesStokes:non-const}
    \begin{alignat}{2}
        -\Divg(2\nu_\Ga(\psi)\Dg\bw) + \Gradg q + \gamma(\phi,\psi)\bw &= \g &&\qquad\text{on~}\Ga, \\
        \Divg\,\bw &= 0 &&\qquad\text{on~}\Ga.
    \end{alignat}
\end{subequations}
We prove the following result concerning the well-posedness of strong solutions to \eqref{System:SurfacesStokes:non-const}.

\begin{lemma}\label{App:Lemma:SurfaceStokes}
    Let $\nu_\Ga\in C^2([-1,1])$ and $\gamma\in C([-1,1]^2)$ satisfy \eqref{Appendix:Assumption':Bound}, $\psi\in W^{1,\infty}(\Ga)$ and $\g\in\mathbf{L}^2_\tau(\Ga)$. Then there exists a unique solution $(\bw,q)\in \mathbf{H}^2_\Div(\Ga)\times(H^1(\Ga)\cap L^2_{(0)}(\Ga))$ to \eqref{System:SurfacesStokes:non-const}. Additionally, there exists a constant $C > 0$ such that
    \begin{align}\label{App:Est:SurfaceStokes}
        \norm{\bw}_{\mathbf{H}^2(\Ga)} + \norm{q}_{H^1(\Ga)} \leq C\norm{\g}_{\mathbf{L}^2(\Ga)}.
    \end{align}
\end{lemma}

\begin{proof}
    We start by recalling the following surface Korn inequality,
    \begin{align*}
        \norm{\tw}_{\mathbf{H}^1(\Ga)}^2 \leq C_{\Ga} \big(\norm{\Dg\tw}_{\mathbf{L}^2(\Ga)}^2 + \norm{\tw}_{\mathbf{L}^2(\Ga)}^2\big) \qquad\text{for all~}\tw\in\mathbf{H}^1(\Ga)
    \end{align*}
    for some constant $C_{\Ga} > 0$, see \cite[Eqn.~(4.7)]{Jankuhn2018}. Consequently, as both $\nu_\Ga$ and $\gamma$ are bounded from below by a positive constant, the Lemma of Lax--Milgram guarantees the existence of a unique $\bw\in\mathbf{H}^1_\Div(\Ga)$ satisfying
    \begin{align}\label{SurfaceStokes:WF}
        \intG 2\nu_\Ga(\psi)\Dg\bw:\Dg\ww\dG + \intG \gamma(\phi,\psi)\bw\cdot\ww\dG = \intG \g\cdot\ww\dG
    \end{align}
    for all $\ww\in\mathbf{H}^1_\Div(\Ga)$ along with a constant $C > 0$ such that
    \begin{align}\label{App:Est:w:H^1:g}
        \norm{\bw}_{\mathbf{H}^1(\Ga)} \leq C\norm{\g}_{\mathbf{L}^2(\Ga)}.
    \end{align}
    Next, we observe that \eqref{SurfaceStokes:WF} can be reformulated as
    \begin{align}\label{SurfaceStokes:WF:2}
        &\intG 2\nu_\Ga(\psi)\Dg\bw:\Dg\ww\dG + \intG \bw\cdot\ww\dG = \intG \G\cdot\ww\dG
    \end{align}
    for all $\ww\in\mathbf{H}^1_\Div(\Ga)$, where $\G = \g + (1-\gamma(\phi,\psi))\bw\in\mathbf{L}^2(\Ga)$. In view of \eqref{App:Est:w:H^1:g} and \eqref{Appendix:Assumption':Bound}, it holds $\norm{\G}_{\mathbf{L}^2(\Ga)} \leq C\norm{\g}_{\mathbf{L}^2(\Ga)}$. Consequently, utilizing regularity theory for \eqref{SurfaceStokes:WF:2} (see \cite[Lemma~7.4]{Abels2024}), we deduce $\bw\in\mathbf{H}^2(\Ga)$ with
    \begin{align*}
        \norm{\bw}_{\mathbf{H}^2(\Ga)} \leq C\norm{\G}_{\mathbf{L}^2(\Ga)} \leq C\norm{\g}_{\mathbf{L}^2(\Ga)}
    \end{align*}
    for some constant $C > 0$. Finally, as $\bw\in\mathbf{H}^2(\Ga)$, an application of integration by parts yields that
    \begin{align*}
        \intG \mathbf{P}_\Div^\Ga\big(-\Divg(2\nu_\Ga(\psi)\Dg\bw) + \gamma(\phi,\psi)\bw - \g\big)\cdot\ww\dG = 0
    \end{align*}
    for all $\ww\in\mathbf{H}^1(\Ga)$. Thus, by definition of the projection $\mathbf{P}_\Div^\Ga$, we find $q\in \dot{H}^1(\Ga)$ such that
    \begin{align}\label{Pressure:q:Construct}
        \Gradg q = -\Divg(2\nu_\Ga(\psi)\Dg\bw) + \gamma(\phi,\psi)\bw - \g \qquad\text{on~}\Ga.
    \end{align}
    Moreover, the choice of $q$ is unique if we additionally require that $\meang{q} = 0$. Then, by the Poincar\'{e}--Wirtinger inequality, we obtain $q\in H^1(\Ga)$. In this way, $(\bw,q)$ is the unique solution to \eqref{System:SurfacesStokes:non-const} satisfying the desired regularities as well as the estimate \eqref{App:Est:SurfaceStokes}.
\end{proof}

\subsection{Bulk-surface Stokes equation}

Now, having Lemma~\ref{App:Lemma:BulkStokes} and Lemma~\ref{App:Lemma:SurfaceStokes} at hand, we can finally prove Theorem~\ref{App:Theorem:BulkSurfaceStokes:SS}.

\begin{proof}[Proof of Theorem~\ref{App:Theorem:BulkSurfaceStokes:SS}]
    Let $\widetilde{\bw}\in\mathbf{H}^{\frac32}_\Div(\Ga)$, and consider the unique pair $(\bv,p)\in\mathbf{H}^2_\Div(\Om)\times(H^1(\Om)\cap L^2_{(0)}(\Om))$ satisfying
    \begin{alignat}{2}\label{App:FixedPoint:v}
        -\Div(2\nu_\Om(\phi)\D\bv) + \Grad p &= \f &&\qquad\text{in~}\Om, \nonumber \\
        \Div\,\bv &= 0 &&\qquad\text{in~}\Om, \\
        \bv\vert_\Ga &= \tw &&\qquad\text{on~}\Ga, \nonumber
    \end{alignat}
    which exists according to Lemma~\ref{App:Lemma:BulkStokes}. Additionally, there exists a constant $C > 0$ such that
    \begin{align}\label{App:Est:BulkStokes:Proof}
        \norm{\bv}_{\mathbf{H}^2(\Om)} + \norm{p}_{H^1(\Om)} \leq C\big(\norm{\f}_{\mathbf{L}^2(\Om)} + \norm{\tw}_{\mathbf{H}^{\frac32}(\Ga)}\big).
    \end{align}
    Then, since by the trace theorem it holds that $\g - 2\nu_\Om(\phi)[\D\bv\,\n]_\tau\in\mathbf{L}^2_{\tau}(\Ga)$, Lemma~\ref{App:Lemma:SurfaceStokes} implies the existence of a unique pair $(\bw,q)\in\mathbf{H}^2_\Div(\Ga)\times (H^1(\Ga)\cap L^2_{(0)}(\Ga))$ solving
    \begin{equation}\label{App:FixedPoint:w}
        \begin{aligned}
            -\Divg(2\nu_\Ga(\psi)\Dg\bw) + \Gradg q + \gamma(\phi,\psi)\bw &= \g - 2\nu_\Om(\phi)[\D\bv\,\n ]_\tau&&\qquad\text{on~}\Ga, \\
        \Divg\,\bw &= 0 &&\qquad\text{on~}\Ga.
        \end{aligned}
    \end{equation}
   This defines a function $\mathbf{T}:\mathbf{H}^{\frac32}_\Div(\Ga)\rightarrow\mathbf{H}^2_\Div(\Ga)$ via $\mathbf{T}\big(\widetilde{\bw}\big) = \bw$, with $\bw$ being the unique solution to \eqref{App:FixedPoint:w}. Now, let $\widetilde{\bw}_1,\widetilde{\bw}_2\in\mathbf{H}^{\frac32}_\Div(\Ga)$, and consider the corresponding solutions $\bw_1, \bw_2$ to \eqref{App:FixedPoint:w} with $\bv_1,\bv_2$ being the associated solutions to \eqref{App:FixedPoint:v}. Then, in view of the estimates \eqref{App:Est:SurfaceStokes} and \eqref{App:Est:BulkStokes:Proof}, we find that
    \begin{align*}
        \norm{\mathbf{T}\big(\widetilde{\bw}_1) - \mathbf{T}\big(\widetilde{\bw}_2)}_{\mathbf{H}^2(\Ga)} &\leq C\norm{\D\bv_1\,\n - \D\bv_2\,\n}_{\mathbf{L}^2(\Ga)} \leq C\norm{\bv_1 - \bv_2}_{\mathbf{H}^2(\Ga)} \\
        &\leq C\norm{\widetilde{\bw}_1 - \widetilde{\bw}_2}_{\mathbf{H}^{\frac32}(\Ga)}.
    \end{align*}
    In particular, this shows that $\mathbf{T}$ is continuous. As we have the compact embedding $\mathbf{H}^2(\Ga)\emb\mathbf{H}^{\frac32}(\Ga)$, Schauder's fixed-point theorem guarantees the existence of a fixed-point $\bw_\ast\in\mathbf{H}^{\frac32}(\Ga)$ such that $\mathbf{T}\big(\bw_\ast\big) = \bw_\ast$. Then, consider the corresponding solution $\bv_\ast$ to \eqref{App:FixedPoint:v}, and the pair $(p_\ast,q_\ast)$ consisting of the bulk and surface pressure, respectively. In view of Lemma~\ref{App:Lemma:BulkSurfaceStokes:WS}, $(\bv_\ast,\bw_\ast,p_\ast,q_\ast)$ is unique.
    Lastly, to establish \eqref{App:Est:BulkSurfaceStokes}, we first note that due to \eqref{InterpolEst:Trace:L^2}, \eqref{App:Est:BulkSurfaceStokes:H^1} and \eqref{App:Est:SurfaceStokes} we have
    \begin{align}\label{App:Est:SurfaceStokes:Final}
        \begin{split}
            \norm{\bw_\ast}_{\mathbf{H}^2(\Ga)} + \norm{q_\ast}_{H^1(\Ga)} &\leq C\big(\norm{\g}_{\mathbf{L}^2(\Ga)} + \norm{\D\bv_\ast\,\n}_{\mathbf{L}^2(\Ga)}\big) \\
            &\leq C\big(\norm{\g}_{\mathbf{L}^2(\Ga)} + \norm{\bv_\ast}_{\mathbf{H}^1(\Om)}^{\frac12}\norm{\bv_\ast}_{\mathbf{H}^2(\Om)}^{\frac12}\big) \\
            &\leq \varepsilon\norm{\bv_\ast}_{\mathbf{H}^2(\Om)} + C_\varepsilon\norm{(\f,\g)}_{\mathbfcal{L}^2}
        \end{split}
    \end{align}
    for some $\varepsilon > 0$ and a constant $C_\varepsilon > 0$. Therefore, plugging this estimate into \eqref{App:Est:BulkStokes:Proof}, and choosing $\varepsilon$ small enough, we deduce
    \begin{align}\label{App:Est:BulkStokes:Final}
        \norm{\bv_\ast}_{\mathbf{H}^2(\Om)} + \norm{p_\ast}_{H^1(\Om)} \leq C\big(\norm{\f}_{\mathbf{L}^2(\Ga)} + \norm{\bw_\ast}_{\mathbf{H}^{\frac32}(\Ga)}\big) \leq C\norm{(\f,\g)}_{\mathbfcal{L}^2}.
    \end{align}
    Finally, combining \eqref{App:Est:SurfaceStokes:Final} and \eqref{App:Est:BulkStokes:Final} yields \eqref{App:Est:BulkSurfaceStokes}, which finishes the proof.
\end{proof}

\subsection{Bulk-surface Stokes operator}

We now consider the case of constant, non-negative coefficients. Without loss of generality, we assume that $\nu_\Om = \nu_\Ga \equiv 1$ as well as $\gamma \equiv 1$. Then, the bulk-surface operator is defined as
\begin{align*}
    \mathbfcal{A}:D(\mathbfcal{A})\subset\mathbfcal{L}^2_\Div&\rightarrow\mathbfcal{L}^2_\Div, \\
    (\bv,\bw)&\mapsto \big(-\mathbf{P}_\Div^\Om(\Div(2\D\bv)), -\mathbf{P}_\Div^\Ga(\Divg(2\Dg\bw) + 2[\D\bv\,\n]_\tau + \bw)\big),
\end{align*}
where
\begin{align*}
    D(\mathbfcal{A}) \coloneqq \mathbfcal{H}^2_{0,\Div}.
\end{align*}
It was shown in \cite{Knopf2025} that $\mathbfcal{A}$ is a positive, unbounded, and self-adjoint operator on $\mathbfcal{L}^2_\Div$ with compact inverse.
Then, thanks to the above regularity results, we deduce that the operator $\mathbfcal{A}^{-1}:\mathbfcal{L}^2_\Div\rightarrow\mathbfcal{H}^2_{0,\Div}$ is such that, for any $(\f,\g)\in\mathbfcal{L}^2_\Div$, there exists $\mathbfcal{A}^{-1}(\f,\g) = \big(\mathbf{A}_\Om^{-1}(\f,\g),\mathbf{A}_\Ga^{-1}(\f,\g)\big)\in D(\mathbfcal{A})$ and $(p,q)\in\mathcal{H}^1\cap\mathcal{L}^2_{(0)}$ which solve
\begin{subequations}\label{App:System:StokesOperator}
    \begin{alignat}{2}
        -\Div(2\D\mathbf{A}_\Om^{-1}(\f,\g)) + \Grad p &= \f &&\qquad\text{in~}\Om, \\
        -\Divg(2\Dg\mathbf{A}_\Ga^{-1}(\f,\g)) + 2[\D\mathbf{A}_\Om^{-1}(\f,\g)\,\n]_\tau + \Gradg q + \mathbf{A}_\Ga^{-1}(\f,\g) &= \g  &&\qquad\text{on~}\Ga.
    \end{alignat}
\end{subequations}
Furthermore, in view of \eqref{App:Est:BulkSurfaceStokes}, we have
\begin{align*}
    \norm{\mathbfcal{A}^{-1}(\f,\g)}_{\mathbfcal{H}^2} + \norm{(p,q)}_{\mathcal{H}^1} \leq C\norm{(\f,\g)}_{\mathbfcal{L}^2}.
\end{align*}
In particular, it holds that
\begin{align}\label{NormEquivalence:Stokes}
    \norm{(\bv,\bw)}_{\mathbfcal{H}^2} \leq C\norm{\mathbfcal{A}(\bv,\bw)}_{\mathbfcal{L}^2} \qquad\text{for all~}(\bv,\bw)\in\mathbfcal{H}^2_{0,\Div}.
\end{align}

Lastly, we provide an $\mathcal{L}^2$ estimate of the pressure pair $(p,q)$ in \eqref{App:System:StokesOperator}, which can be seen as a generalization of \cite[Lemma~B.2]{Giorgini2019}.
\begin{lemma}
    Let $(\f,\g)\in\mathbf{L}^2_\Div$. Then there exists a constant $C > 0$, independent of $\f$ and $\g$, such that
    \begin{align}\label{App:Est:Pressure}
        \norm{(p,q)}_{\mathcal{L}^2} \leq C\norm{\mathbfcal{A}^{-1}(\f,\g)}_{\mathbfcal{H}^1_{0}}^{\frac12}\norm{(\f,\g)}_{\mathbfcal{L}^2}^{\frac12}.
    \end{align}
\end{lemma}

\begin{proof}
    Consider $(\wv,\ww)\in\mathbfcal{H}^1_0$. Then, using the definition of the Stokes operator $\mathbfcal{A}$ and integration by parts, we find that
    \begin{align*}
        &\bigscp{(\f,\g)}{(\wv,\ww)}_{\mathbfcal{L}^2} \\
        &\quad = \bigscp{\mathbfcal{A}\mathbfcal{A}^{-1}(\f,\g)}{(\wv,\ww)}_{\mathbfcal{L}^2} \\
        &\quad = \intO \mathbf{A}_\Om\mathbf{A}_\Om^{-1}(\f,\g)\cdot\wv\dx + \intG \mathbf{A}_\Ga\mathbf{A}_\Ga^{-1}(\f,\g)\cdot\ww\dG \\
        &\quad = - \intO \mathbf{P}_\Div^\Om\big(\Div(2\D\mathbf{A}_\Om^{-1}(\f,\g))\big)\cdot\wv \dx - \intG \mathbf{P}_\Div^\Ga\big(\Divg(2\Dg\mathbf{A}_\Ga^{-1}(\f,\g))\big)\cdot\ww \dG \\
        &\qquad + \intG \mathbf{P}_\Div^\Ga\big(2[\D\mathbf{A}_\Om^{-1}(\f,\g)\,\n]_\tau\big)\cdot\ww\dG + \intG \mathbf{A}_\Ga^{-1}(\f,\g)\cdot\ww\dG \\
        &\quad = - \intO \Div(2\D\A_\Om^{-1}(\f,\g))\cdot\mathbf{P}_\Div^\Om\wv\dx - \intG \Divg(2\Dg\A_\Ga^{-1}(\f,\g))\cdot\mathbf{P}_\Div^\Ga\ww\dG \\
        &\qquad + \intG 2\D\mathbf{A}_\Om^{-1}(\f,\g)\n\,\cdot\mathbf{P}_\Div^\Ga\ww\dG + \intG \mathbf{A}_\Ga^{-1}(\f,\g)\cdot\mathbf{P}_\Div^\Ga\ww\dG \\
        &\quad = \intO 2\D\A_\Om^{-1}(\f,\g):\D\mathbf{P}_\Div^\Om\wv\dx + \intG 2\Dg\A_\Ga^{-1}(\f,\g):\Dg\mathbf{P}_\Div^\Ga\ww\dG \\
        &\qquad - \intG 2\D\A_\Om^{-1}(\f,\g)\n\,\cdot\mathbf{P}_\Div^\Om\wv\dx + \intG 2\D\A_\Om^{-1}(\f,\g)\n\,\cdot\mathbf{P}_\Div^\Ga\ww\dG + \intG \mathbf{A}_\Ga^{-1}(\f,\g)\cdot\mathbf{P}_\Div^\Ga\ww\dG.
    \end{align*}
    Recalling that $\mathbf{P}_\Div^\Om:\mathbf{H}^1(\Om)\rightarrow\mathbf{H}^1(\Om)\cap\mathbf{L}^2_\Div(\Om)$ as well as $\mathbf{P}_\Div^\Ga:\mathbf{H}^1(\Ga)\rightarrow\mathbf{H}^1(\Ga)\cap\mathbf{L}^2_\Div(\Ga)$ are both bounded operators, respectively, we infer
    \begin{align*}
        &\Bigabs{\intO 2\D\A_\Om^{-1}(\f,\g):\D\mathbf{P}_\Div^\Om\wv\dx + \intG 2\Dg\A_\Ga^{-1}(\f,\g):\Dg\mathbf{P}_\Div^\Ga\ww\dG  + \intG \mathbf{A}_\Ga^{-1}(\f,\g)\cdot\mathbf{P}_\Div^\Ga\ww\dG} \\
        &\quad\leq C\big(\norm{\D\mathbf{A}_\Om^{-1}(\f,\g)}_{\mathbf{L}^2(\Om)} + \norm{\Dg\mathbf{A}_\Ga^{-1}(\f,\g)}_{\mathbf{L}^2(\Ga)} + \norm{\mathbf{A}_\Ga^{-1}(\f,\g)}_{\mathbf{L}^2(\Ga)}\big)\norm{(\wv,\ww)}_{\mathbfcal{H}^1} \\
        &\quad\leq C\norm{\mathbfcal{A}^{-1}(\f,\g)}_{\mathbfcal{H}^1_{0}}\norm{(\wv,\ww)}_{\mathbfcal{H}^1}.
    \end{align*}
    Next, invoking Lemma~\ref{Prelim:Lemma:Interpol:Trace} together with the estimate \eqref{App:Est:BulkSurfaceStokes}, we deduce that
    \begin{align*}
        &\Bigabs{- \intG 2\D\A_\Om^{-1}(\f,\g)\n\,\cdot\mathbf{P}_\Div^\Om\wv\dx + \intG 2\D\A_\Om^{-1}(\f,\g)\n\,\cdot\mathbf{P}_\Div^\Ga\ww\dG} \\
        &\quad\leq C\norm{\D\A_\Om^{-1}(\f,\g)}_{\mathbf{L}^2(\Ga)}\norm{(\wv,\ww)}_{\mathbfcal{H}^1} \\
        &\quad\leq C\norm{\D\A_\Om^{-1}(\f,\g)}_{\mathbf{L}^2(\Om)}^{\frac12}\norm{\D\A_\Om^{-1}(\f,\g)}_{\mathbf{H}^1(\Om)}^{\frac12}\norm{(\wv,\ww)}_{\mathbfcal{H}^1} \\
        &\quad\leq C\norm{\mathbfcal{A}^{-1}(\f,\g)}_{\mathbfcal{H}^1_{0}}^{\frac12}\norm{(\f,\g)}_{\mathbfcal{L}^2}^{\frac12}\norm{(\wv,\ww)}_{\mathbfcal{H}^1}.
    \end{align*}
    Consequently, taking the supremum over all $(\wv,\ww)\in\mathbfcal{H}^1_0$ such that $\norm{(\wv,\ww)}_{\mathbfcal{H}^1} \leq 1$, we find that
    \begin{align*}
        \norm{(\f,\g)}_{(\mathbfcal{H}^1_0)^\prime} &\leq C\norm{\mathbfcal{A}^{-1}(\f,\g)}_{\mathbfcal{H}^1_{0}} + C\norm{\mathbfcal{A}^{-1}(\f,\g)}_{\mathbfcal{H}^1_0}^{\frac12}\norm{(\f,\g)}_{\mathbfcal{L}^2}^{\frac12},
    \end{align*}
    from which the desired inequality \eqref{App:Est:Pressure} immediately follows.
\end{proof}

\bigskip
\noindent 
\section*{Acknowledgments} 
This work was supported by the Deutsche Forschungsgemeinschaft (DFG, German Research Foundation) - Project 52469428, and partially supported by the Deutsche Forschungsgemeinschaft (DFG, German Research Foundation) - RTG 2339. The support is gratefully acknowledged.


\section*{Conflict of Interests and Data Availability Statement}

There is no conflict of interest.

There is no associated data with the manuscript.


\footnotesize

\bibliographystyle{abbrv}
\bibliography{S_NSCH}

\end{document}